\documentclass{amsart}

\usepackage{amsmath}
\usepackage{amsthm}
\usepackage{amsfonts}
\usepackage{amssymb}

\newcommand\R{{\mathbb{R}}}
\newcommand\C{{\mathbb{C}}}

\renewcommand\P{{\mathbf{P}}}
\newcommand\E{{\mathbf{E}}}

\newcommand\Var{\mathbf{Var}}
\renewcommand\Im{{\operatorname{Im}}}
\renewcommand\Re{{\operatorname{Re}}}
\newcommand\eps{{\varepsilon}}

\newcommand\tr{\operatorname{trace}}

\newcommand\Dyson{{\operatorname{Sine}}}
\renewcommand\th{{\operatorname{th}}}

\newcommand\condo{{{\bf C0}}}
\newcommand\condone{{{\bf C1}}}

\newcommand\CN{{\mathcal N}}

\subjclass{15A52}

\newcommand\cl{{\operatorname{cl}}}
\renewcommand\sc{{\operatorname{sc}}}
\newcommand\Ai{{\operatorname{Ai}}}
\newcommand\Sine{{\operatorname{Sine}}}

\theoremstyle{plain}
  \newtheorem{theorem}{Theorem}

  \newtheorem{proposition}[theorem]{Proposition}
  
  \newtheorem{lemma}[theorem]{Lemma}
  \newtheorem{corollary}[theorem]{Corollary}

\theoremstyle{definition}
  \newtheorem{definition}[theorem]{Definition}
  \newtheorem{example}[theorem]{Example}
  \newtheorem{remark}[theorem]{Remark}

\include{psfig}

\begin{document}

\title[Universality for Wigner ensembles]{Random matrices:\\  The Universality phenomenon for Wigner ensembles}

\author{Terence Tao}
\address{Department of Mathematics, UCLA, Los Angeles CA 90095-1555}
\email{tao@math.ucla.edu}
\thanks{T. Tao is supported by a grant from the MacArthur Foundation, by NSF grant DMS-0649473, and by the NSF Waterman award.}

\author{Van Vu}
\address{Department of Mathematics, Yale, New Haven, CT 06520}
\email{van.vu@yale.edu}
\thanks{V. Vu is supported by research grants DMS-0901216 and AFOSAR-FA-9550-09-1-0167.}

\begin{abstract}  In this paper, we survey some recent progress on rigorously establishing the universality of various spectral
 statistics of Wigner Hermitian random matrix ensembles, focusing on the Four Moment Theorem and its 
refinements and applications, including 
the universality of the sine kernel  and the
Central limit theorem of several spectral parameters.  

We also take the opportunity here to issue some errata for some of our previous papers in this area.
\end{abstract}

\maketitle

\setcounter{tocdepth}{2}

\section{Introduction}


\emph{Random matrix theory} is a central topic in probability and mathematical physics, with many connections to various areas such as statistics, number theory, combinatorics, numerical analysis and theoretical computer science. 

One of the primary goal of random matrix theory is to derive limiting laws for the eigenvalues and eigenvectors of ensembles of large ($n \times n$) Hermitian random matrices\footnote{One can of course also study non-Hermitian or non-square random matrices, though in the latter case the concept of an eigenvalue needs to be replaced with that of a singular value.  In this survey we will focus almost exclusively on the square Hermitian case.}, in the asymptotic limit $n \to \infty$.  There are many random matrix ensembles of interest, but to focus the discussion and to simplify the exposition we shall restrict attention to an important model class of ensembles, the \emph{Wigner matrix ensembles}.

\begin{definition}[Wigner matrices]\label{def:Wignermatrix}  Let $n \geq 1$ be an integer (which we view as a parameter going off to infinity). An $n \times n$ \emph{Wigner Hermitian matrix} $M_n$ is defined to be a  random Hermitian $n \times n$ matrix $M_n = (\xi_{ij})_{1 \leq i,j \leq n}$, in which the $\xi_{ij}$ for $1 \leq i \leq j \leq n$ are jointly independent with $\xi_{ji} = \overline{\xi_{ij}}$ (in particular, the $\xi_{ii}$ are real-valued).  For $1 \leq i < j \leq n$, we require that the $\xi_{ij}$ have mean zero and variance one, while for $1 \leq i=j \leq n$ we require that the $\xi_{ij}$ (which are necessarily real) have mean zero and variance $\sigma^2$ for some $\sigma^2>0$ independent of $i,j,n$.   To simplify some of the statements of the results here, we will also assume that the $\xi_{ij} \equiv \xi$ are identically distributed for $i < j$, and the $\xi_{ii} \equiv \xi'$ are also identically distributed for $i=j$, and furthermore that the real and imaginary parts of $\xi$ are independent.  We refer to the distributions $\Re \xi$, $\Im \xi$, and $\xi'$ as the \emph{atom distributions} of $M_n$.

We say that the Wigner matrix ensemble \emph{obeys Condition {\condo}} if we have the exponential decay condition
\begin{equation*}
\P(|\xi_{ij}|\ge t^C) \le e^{-t} 
\end{equation*}
for all $1 \leq i,j \leq n$ and $t \ge C'$, and some constants $C, C'$ (independent of $i,j,n$).  We say that the Wigner matrix ensemble \emph{obeys condition {\condone} with constant $C_0$} if one has
$$ \E |\xi_{ij}|^{C_0} \leq C$$
for some constant $C$ (independent of $n$).
\end{definition}

Of course, Condition {\condo} implies Condition {\condone} for any $C_0$, but not conversely.

We refer to the matrix $W_n := \frac{1}{\sqrt{n}} M_n$ as the \emph{coarse-scale normalised Wigner Hermitian matrix}, and $A_n := \sqrt{n} M_n$ as the \emph{fine-scale normalised Wigner Hermitian matrix}.

\begin{example}[Invariant ensembles]\label{gue}  An important special case of a Wigner Hermitian matrix $M_n$ is the \emph{gaussian unitary ensemble} (GUE), in which $\xi_{ij} \equiv N(0,1)_\C$ are complex gaussians with mean zero and variance one for $i \neq j$, and $\xi_{ii} \equiv N(0,1)_\R$ are real gaussians with mean zero and variance one for $1 \leq i \leq n$ (thus $\sigma^2 = 1$ in this case).  Another important special case is the \emph{gaussian orthogonal ensemble} (GOE), in which $\xi_{ij} \equiv N(0,1)_\R$ are real gaussians with mean zero and variance one for $i \neq j$, and $\xi_{ii} \equiv N(0,1/2)_\R$ are real gaussians with mean zero and variance $1/2$ for $1 \leq i \leq n$ (thus $\sigma^2=1/2$ in this case).  These ensembles obey Condition {\condo}, and hence Condition {\condone} for any $C_0$.  For these two ensembles, the probability distribution of $M_n$ can be expressed invariantly in either case as
\begin{equation}\label{cmm}
 \frac{1}{Z_n} e^{-c \operatorname{tr} M_n M_n^*} dM_n
\end{equation}
for some quantity $Z_n>0$ depending only on $n$, where $dM_n$ is Haar measure on the vector space of $n \times n$ Hermitian matrices (in the case of GUE) or real symmetric matrices (in the case of GOE), and $c$ is equal to $1/2$ (for GUE) or $1/4$ (for GOE).  From \eqref{cmm} we easily conclude that the probability distribution of GUE is invariant with respect to conjugations by unitary matrices, and similarly the probability distribution of GOE is invariant with respect to conjugations by orthogonal matrices.  However, a general Wigner matrix ensemble will not enjoy invariances with respect to such large classes of matrices.  (For instance, the Bernoulli ensembles described below are only invariant with respect to conjugation by a discrete group of matrices, which include permutation matrices and reflections around the coordinate axes.)
\end{example}

\begin{example}[Bernoulli ensembles]\label{bernoulli}  At the opposite extreme from the invariant ensembles are the \emph{Bernoulli ensembles}, which are discrete instead of continuous.  In the \emph{real Bernoulli ensemble} (also known as \emph{symmetric random sign matrices}), each of the $\xi_{ij}$ are equal to $+1$ with probability $1/2$ and $-1$ with probability $1/2$.  In the \emph{complex Bernoulli ensemble}, the diagonal entries $\xi_{ii}$ still have this distribution, but the off-diagonal entries now take values\footnote{We use $\sqrt{-1}$ to denote the imaginary unit, in order to free up the symbol $i$ as an index variable.} $\pm \frac{1}{\sqrt{2}} \pm \frac{\sqrt{-1}}{\sqrt{2}}$, with each of these four complex numbers occuring with probability $1/4$.
\end{example}

\begin{remark}  Many of the results given here have been extended to somewhat broader classes of matrices.  For instance, one can consider generalised Wigner ensembles in which entries are not identically distributed; for instance, one can allow the variances $\sigma_{ij}^2$ of each entry $\xi_{ij}$ to vary in $i,j$, and even vanish for some $i,j$; the latter situation occurs for instance in band-limited random matrices; one can also consider allowing the mean $\mu_{ij}$ of the entries $\xi_{ij}$ to be non-zero (this is for instance the situation with the adjacency matrices of Erd\H{o}s-Renyi random graphs).  One can also consider sparse Wigner random matrices, in which only a small (randomly selected) number of entries are non-zero.  For simplicity, though, we shall mostly restrict attention in this survey to ordinary Wigner ensembles.  We do remark, however, that in all of these generalisations, it remains crucial that the entries $\xi_{ij}$ for $1 \leq i \leq j \leq n$ are jointly independent, as many of the techniques currently available to control the fine-scale spectral structure of Wigner matrices rely heavily on joint independence.  
\end{remark}

Given an $n \times n$ Hermitian matrix $A$, we denote its $n$ eigenvalues in increasing order\footnote{It is also common in the literature to arrange eigenvalues in decreasing order instead of increasing.  Of course, the results remain the same under this convention except for minor notational changes.} as
$$ \lambda_1(A) \leq \ldots \leq \lambda_n(A),$$
and write $\lambda(A) := (\lambda_1(A),\ldots,\lambda_n(A))$.  We also let $u_1(A),\ldots,u_n(A) \in \C^n$ be an orthonormal basis of eigenvectors of $A$ with $A u_i(A) = \lambda_i(A) u_i(A)$; these eigenvectors $u_i(A)$ are only determined up to a complex phase even when the eigenvalues are simple (or up to a sign in the real symmetric case), but this ambiguity will not cause much of a difficulty in our results as we will usually only be interested in the \emph{magnitude} $|u_i(A)^* X|$ of various inner products $u_i(A)^* X$ of $u_i(A)$ with other vectors $X$.

We also introduce the \emph{eigenvalue counting function}
\begin{equation}\label{eigencount}
 N_I(A) := | \{ 1 \leq i \leq n: \lambda_i(A) \in I \}|
\end{equation}
for any interval $I \subset \R$.  We will be interested in both the \emph{coarse-scale} eigenvalue counting function $N_I(W_n)$ and the \emph{fine-scale} eigenvalue counting function $N_I(A_n)$, which are of course transformable to each other by the identity $N_I(W_n) =N_{nI}(A_n)$.

\section{Global and local semi-circular laws}\label{semi-sec}

We first discuss the coarse-scale spectral structure of Wigner matrices, that is to say the structure of the eigenvalues of $W_n$ at unit scales (or equivalently, the eigenvalues of $M_n$ at scale $\sqrt{n}$, or $A_n$ at scale $n$).  The fundamental result in this topic is the \emph{(global) Wigner semi-circular law}.  Denote by $\rho_{sc}$ the semi-circle density function with
support on $[-2,2]$,
\begin{equation}\label{semi}
 \rho_{sc} (x):= \begin{cases} \frac{1}{2\pi} \sqrt {4-x^2}, &|x| \le 2 \\ 0,
&|x| > 2. \end{cases} 
\end{equation}

\begin{theorem}[Global semi-circular law]\label{theorem:Wigner} Let $M_n$ be a Wigner Hermitian matrix.  Then for any interval $I$ (independent of $n$), one has
$$\lim_{n \rightarrow \infty} \frac{1}{n} N_I[W_n]
 = \int_I \rho_{sc}(y)\ dy$$
in the sense of probability (and also in the almost sure sense, if the $M_n$ are all minors of the same infinite Wigner Hermitian matrix).
\end{theorem}

\begin{remark} Wigner \cite{wig} proved this theorem for special ensembles. The general
 version  above is due to  Pastur \cite{Pas} (see \cite{BSbook, AGZ} for a detailed discussion).  The semi-circular law in fact holds under substantially more general hypotheses than those given in Definition \ref{def:Wignermatrix}, but we will not discuss this matter further here. One consequence of Theorem \ref{theorem:Wigner} is that we expect most of the eigenvalues of $W_n$ to lie in the interval $(-2+\eps,2+\eps)$ for $\eps > 0$ small; we shall thus informally refer to this region as the \emph{bulk} of the spectrum.
\end{remark} 

An essentially equivalent\footnote{This formulation is slightly stronger because it also incorporates the upper bound $\|W_n\|_{op} \leq 2+o(1)$ on the operator norm of $W_n$, which is consistent with, but not implied by, the semi-circular law, and follows from the work of Bai and Yin \cite{bai-yin}.}  formulation of the semi-circular law is as follows: if $1 \leq i \leq n$, then one has\footnote{We use the asymptotic notation $o(1)$ to denote any quantity that goes to zero as $n \to \infty$, and $O(X)$ to denote any quantity bounded in magnitude by $CX$, where $C$ is a constant independent of $n$.}
\begin{equation}\label{wwn}
 \lambda_i(W_n) = \lambda_i^{\cl}(W_n) + o(1)
\end{equation}
with probability $1-o(1)$ (and also almost surely, if the $M_n$ are minors of an infinite matrix), where the \emph{classical location} $\lambda_i^{\cl}(W_n)$ of the $i^{th}$ eigenvalue is the element of $[-2,2]$ defined by the formula
$$ \int_{-2}^{\lambda_i^{\cl}(W_n)} \rho_\sc(y)\ dy = \frac{i}{n}.$$

A remarkable feature of the semi-circular law is its \emph{universality}: the precise distribution of the atom variables $\xi_{ij}$ are irrelevant for the conclusion of the law, so long as they are normalised to have mean zero and variance one (or variance $\sigma^2$, on the diagonal), and are jointly independent on the upper-triangular portion of the matrix.  In particular, continuous matrix ensembles such as GUE or GUE, and discrete matrix ensembles such as the Bernoulli ensembles, have the same asymptotic spectral distribution when viewed at the coarse scale (i.e. by considering eigenvalue ranges of size $\sim 1$ for $W_n$, or equivalently of size $\sim \sqrt{n}$ for $M_n$ or $\sim n$ for $A_n$).

However, as stated, the semi-circular law does not give good control on the \emph{fine-scale} behaviour of the eigenvalues, for instance in controlling $N_I(W_n)$ when $I$ is a very short interval (of length closer to $1/n$ than to $1$).  The fine-scale theory for Wigner matrices is much more recent than the coarse-scale theory given by the semi-circular law, and is the main focus of this survey.

There are several ways to establish the semi-circular law.  For invariant ensembles such as GUE or GOE, one can use explicit formulae for the probability distribution of $N_I$ coming from the theory of determinantal processes, giving precise estimates all the way down to infinitesimally small scales; see e.g. \cite{AGZ} or Section \ref{gue-sec} below.  However, these techniques rely heavily on the invariance of the ensemble, and do not directly extend to more general Wigner ensembles.  Another popular technique is the \emph{moment method}, based on the basic moment identities
\begin{equation}\label{moment-sum}
 \sum_{i=1}^n \lambda_i(W_n)^k = \tr W_n^k = \frac{1}{n^{k/2}} \tr M_n^k
\end{equation}
for all $k \geq 0$.  The moment method already is instructive for revealing at least one explanation for the universality phenomenon.  If one takes expectations in the above formula, one obtains
$$ \E \sum_{i=1}^n \lambda_i(W_n)^k = \frac{1}{n^{k/2}} \sum_{1 \leq i_1,\ldots,i_k \leq n} \E \xi_{i_1 i_2} \ldots \xi_{i_k i_1}.$$
For those terms for which each edge $\{i_j,i_{j+1}\}$ appears at most twice, the summand can be explicitly computed purely in terms of the mean and variances of the $\xi_{ij}$, and are thus universal.  Terms for which an edge appears three or more times are sensitive to higher moments of the atom distribution, but can be computed to give a contribution of $o(1)$ (at least assuming a decay condition such as Condition \condo).  This already explains universality for quantities such as $\E N_I(W_n)$ for intervals $I$ of fixed size (independent of $n$), at least if one assumes a suitable decay condition on the entries (and one can use standard truncation arguments to relax such hypotheses substantially).

At the edges $\pm 2$ of the spectrum, the moment method can be pushed further, to give quite precise control on the most extreme eigenvalues of $W_n$ (which dominate the sum in \eqref{moment-sum}) if $k$ is sufficiently large; see \cite{ FK, sinai1, sinai2, Sos1, SP,  Vunorm} and the references therein.  However, the moment method is quite poor at controlling the spectrum in the bulk.  To improve the understanding of the bulk spectrum of Wigner matrices, Bai \cite{Bai93a, Bai93b} (see also \cite[Chapter 8]{BSbook}) used the \emph{Stieltjes transform method} (building upon the earlier work of Pastur \cite{Pas}) to show that the speed of convergence to the semi-circle was $O(n^{-1/2})$.  Instead of working with moments, one instead studied the Stieltjes transform
$$ s_n(z) = \frac{1}{n} \tr(W_n - z)^{-1} = \frac{1}{n} \sum_{i=1}^n \frac{1}{\lambda_i(W_n) - z}$$
of $W_n$, which is well-defined for $z$ outside of the spectrum of $W_n$ (and in particular for $z$ in the upper half-plane $\{ z \in \C: \operatorname{Im}(z) > 0 \}$).  To establish the semi-circular law, it suffices to show that $s_n(z)$ converges (in probability or almost surely) to $s_{\sc}(z)$ for each $z$, where
$$ s_{\sc}(z) := \int_\R \frac{1}{x-z} \rho_{\sc}(x)\ dx$$
is the Stieltjes transform of the semi-circular distribution $\rho_\sc$.  The key to the argument is to establish a \emph{self-consistent equation} for  $s_n$, which roughly speaking takes the form
\begin{equation}\label{sce}
s_n(z) \approx \frac{-1}{s_n(z)+z}.
\end{equation}
One can explicitly compute by contour integration that
$$ s_\sc(z) = \frac{1}{2} (-z + \sqrt{z^2-4})$$
for $z \neq [-2,2]$, where $\sqrt{z^2-4}$ is the branch of the square root that is asymptotic to $z$ at infinity, and in particular that
\begin{equation}\label{scz}
 s_\sc(z) = \frac{-1}{s_\sc(z)+z}.
\end{equation}
From a stability analysis of the elementary equation \eqref{scz} and a continuity argument, one can then use \eqref{sce} to show that 
\begin{equation}\label{snak}
s_n(z) \approx s_\sc(z), 
\end{equation}
which then implies convergence to the semi-circular law.  If one can obtain quantitative control on the approximation in \eqref{sce}, one can then deduce quantitative versions of the semi-circular law that are valid for certain short intervals.

We briefly sketch why one expects the self-consistent equation to hold.  One can expand 
\begin{equation}\label{expansion}
 s_n(z) = \frac{1}{n} \sum_{i=1}^n ( (W_n - zI)^{-1} )_{ii}
\end{equation}
where $( (W_n - zI)^{-1} )_{ii}$ denotes the $ii^{\th}$ entry of the matrix $(W_n-zI)^{-1}$.  Let us consider the $i=n$ term for sake of concreteness.  If one expands $W_n$ as a block matrix
\begin{equation}\label{block-matrix}
W_n-zI := \begin{pmatrix} \tilde W_{n-1} - z I& \frac{1}{\sqrt{n}} X_n \\ \frac{1}{\sqrt{n}} X_n^* & \frac{1}{\sqrt{n}} \xi_{nn} - z\end{pmatrix},
\end{equation}
where $\tilde W_{n-1} = \frac{1}{\sqrt{n}} M_{n-1}$ is the top left $n-1 \times n-1$ minor of $M_n$, and $X_n$ is the $n-1 \times 1$ column vector with entries $\xi_{n1},\ldots,\xi_{n(n-1)}$, then an application of Schur's complement yields the identity
\begin{equation}\label{snn}
 ((W_n-zI)^{-1})_{nn} = \frac{-1}{z + \frac{1}{n} X_n^* (\tilde W_{n-1}- z I)^{-1} X_n - \frac{1}{\sqrt{n}} \xi_{nn}}.
\end{equation}
The term $\frac{1}{\sqrt{n}} \xi_{nn}$ is usually negligible and will be ignored for this heuristic discussion.
Let us temporarily freeze (or condition on) the entries of the random matrix $M_{n-1}$, and hence $\tilde W_{n-1}$.  Due to the joint independence of the entries of $M_n$, the entries of $X_n$ remain jointly independent even after this conditioning.  As these entries also have mean zero and variance one, we easily compute that
$$ \E \frac{1}{n} X_n^* (\tilde W_{n-1}- z I)^{-1} X_n = \frac{1}{n} \tr (\tilde W_{n-1}- z I)^{-1}.$$
Using the Cauchy interlacing theorem
\begin{equation}\label{interlace}
\lambda_{i-1}(M_{n-1}) \leq \lambda_i(M_n) \leq \lambda_i(M_{n-1})
\end{equation}
property between the eigenvalues of $M_{n-1}$ and the eigenvalues of $M_n$ (which easily follows from the Courant-Fisher minimax formula 
$$ \lambda_i(M_n) = \inf_{V \subset \C^n; \dim(V)=i} \sup_{u \in V: \|u\|=1} u^* M_n u $$
for the eigenvalues), one easily obtains an approximation of the form
$$ \frac{1}{n} \tr (\tilde W_{n-1}- z I)^{-1} \approx s_n(z).$$
Assuming that the expression $\frac{1}{n} X_n^* (\tilde W_{n-1}- z I)^{-1} X_n$ concentrates around its mean (which can be justified under various hypotheses on $M_n$ and $z$ using a variety of concentration-of-measure tools, such as Talagrand's concentration inequality, see e.g. \cite{ledoux}), one thus has
\begin{equation}\label{snn-2}
\frac{1}{n} X_n^* (\tilde W_{n-1}- z I)^{-1} X_n \approx s_n(z)
\end{equation}
and thus
$$ ((W_n-zI)^{-1})_{nn} \approx \frac{-1}{z + s_n(z)}.$$
Similarly for the other diagonal entries $((W_n-zI)^{-1})_{ii}$.  Inserting this approximation back into \eqref{expansion} gives the desired approximation \eqref{sce}, heuristically at least.

The above argument was optimized\footnote{There are further refinements to this method in Erd\H os, Yau, and Yin \cite{EYY}, \cite{EYY2} and Erd\H os-Knowles-Yau-Yin \cite{ekyy}, \cite{ekyy2} that took advantage of some additional cancellation between the error terms in \eqref{snn} (generalised to indices $i=1,\ldots,n$) that could be obtained (via the moment method in \cite{EYY}, \cite{EYY2}, and via decoupling arguments in \cite{ekyy}, \cite{ekyy2}), to improve the error estimates further.  These refinements are not needed for the application to Wigner matrices assuming a strong decay condition such as Condition \condo, but is useful in generalised Wigner matrix models, such as sparse Wigner matrices in which most of the entries are zero, or in models where one only has a weak amount of decay (e.g. Condition {\condone}\ with $C_0=4+\eps$).  These refinements also lead to the very useful \emph{eigenvalue rigidity bound} \eqref{eigenrigid}.} in a sequence of papers \cite{ESY1, ESY2, ESY3} by Erd\H os, Schlein, and Yau (see also \cite[Section 5.2]{TVlocal1} for a slightly simplified proof).  As a consequence, one was able to obtain good estimates of the form \eqref{snak} even when $z$ was quite close to the spectrum $[-2,2]$ (e.g. at distance $O(n^{-1+\eps})$ for some small $\eps>0$, which in turn leads (by standard arguments) to good control on the eigenvalue counting function $N_I(W_n)$ for intervals $I$ of length as short as $n^{-1+\eps}$.  (Note that as there are only $n$ eigenvalues in all, such intervals are expected to only have about $O(n^\eps)$ eigenvalues in them.)  Such results are known as \emph{local semi-circular laws}.  A typical such law (though not the strongest such law known) is as follows:

\begin{theorem}[Local semi-circle law]\label{lsc}  Let $M_n$ be a Wigner matrix obeying Condition \condo, let $\eps > 0$, and let $I \subset \R$ be an interval of length $|I| \geq n^{-1+\eps}$.  Then with overwhelming probability\footnote{By this, we mean that the event occurs with probability $1 - O_A(n^{-A})$ for each $A>0$.}, one has
\begin{equation}\label{nawi}
 N_I(W_n) = n\int_I \rho_{\sc}(x)\ dx + o(n|I|).
\end{equation}
\end{theorem}

\begin{proof} See e.g. \cite[Theorem 1.10]{TVlocal2}.  For the most precise estimates currently known of this type (and with the weakest decay hypotheses on the entries), see \cite{ekyy}.
\end{proof}

A variant of Theorem \ref{lsc}, which was established\footnote{The result in \cite{EYY2} actually proves a more precise result that also gives sharp results in the edge of the spectrum, though due to the sparser nature of the $\lambda_i^\cl(W_n)$ in that case, the error term $O_\eps(n^{-1+\eps})$ must be enlarged.} in the subsequent paper \cite{EYY2}, is the extremely useful \emph{eigenvalue rigidity property}
\begin{equation}\label{eigenrigid}
 \lambda_i(W_n) = \lambda_i^{\cl}(W_n) + O_\eps(n^{-1+\eps}),
\end{equation}
valid with overwhelming probability in the bulk range $\delta n \leq i \leq (1-\delta) n$ for any fixed $\delta>0$ (and assuming Condition \condo), and which significantly improves upon \eqref{wwn}.   This result is key in some of the strongest applications of the theory. See Section \ref{rigidity} for the precise form of this result and recent developments. 

Roughly speaking, results such as Theorem \ref{lsc} and \eqref{eigenrigid} control the spectrum of $W_n$ at scales $n^{-1+\eps}$ and above.  However, they break down at the fine scale $n^{-1}$; indeed, for intervals $I$ of length $|I|=O(1/n)$, one has $n\int_I \rho_{\sc}(x)\ dx = O(1)$, while $N_I(W_n)$ is clearly a natural number, so that one can no longer expect an asymptotic of the form \eqref{nawi}.  Nevertheless, local semicircle laws are an essential part of the fine-scale theory.  One particularly useful consequence of these laws is that of \emph{eigenvector delocalisation}:

\begin{corollary}[Eigenvalue delocalisation]\label{deloc} Let $M_n$ be a Wigner matrix obeying Condition \condo, and let $\eps > 0$.  Then with overwhelming probability, one has $u_i(W_n)^* e_j = O( n^{-1/2 + \eps} )$ for all $1 \leq i,j \leq n$.
\end{corollary}

Note from Pythagoras' theorem that $\sum_{j=1}^n |u_i(M_n)^* e_j|^2 = \|u_i(M_n)\|^2 = 1$; thus Corollary \ref{deloc} asserts, roughly speaking, that the coefficients of each eigenvector are as spread out (or \emph{delocalised}) as possible.

\begin{proof} (Sketch)  By symmetry we may take $e_j=n$.  Fix $i$, and set $\lambda := \lambda_i(W_n)$; then the eigenvector equation $(W_n-\lambda) u_i(W_n) = 0$ can be expressed using \eqref{block-matrix} as
$$
\begin{pmatrix} \tilde W_{n-1} - \lambda I& \frac{1}{\sqrt{n}} X_n \\ \frac{1}{\sqrt{n}} X_n^* & \frac{1}{\sqrt{n}} \xi_{nn} - \lambda\end{pmatrix}
\begin{pmatrix} \tilde u_i \\ u_i(W_n)^* e_n \end{pmatrix} = 0$$
where $\tilde u_i$ are the first $n-1$ coefficients of $u_i(W_n)$.  After some elementary algebraic manipulation (using the normalisation $\|u_i(W_n)\| = 1$), this leads to the identity
$$ |u_i(W_n)^* e_n|^2 = \frac{1}{1 + \| (\tilde W_{n-1}-\lambda)^{-1} X_n \|^2 / n}$$
and hence by eigenvalue decomposition
$$ |u_i(W_n)^* e_n|^2 = \frac{1}{1 + \sum_{j=1}^{n-1} (\lambda_j(\tilde W_{n-1})-\lambda)^{-2} |u_j(\tilde W_{n-1})^* X_n|^2 / n}.$$
Suppose first that we are in the bulk case when $\delta n \leq i \leq (1-\delta) n$ for some fixed $\delta > 0$.  From the local semicircle law, we then see with overwhelming probability that there are $\gg n^{\eps/2}$ eigenvalues $\lambda_j(\tilde W_{n-1})$ that lie within $n^{-1-\eps/2}$ of $\lambda$.  Letting $V$ be the span of the corresponding eigenvectors, we conclude that
$$ |u_i(W_n)^* e_n|^2 \ll n^{-1+\eps} / \|\pi_V(X_n)\|^2$$
where $\pi_V$ is the orthogonal projection to $V$.  If we freeze (i.e. condition) on $\tilde W_{n-1}$ and hence on $V$, then the coefficients of $X_n$ remain jointly independent with mean zero and variance one.  A short computation then shows that
$$ \E \|\pi_V(X_n)\|^2 = \dim(V) \gg n^{\eps/2}$$
and by using concentration of measure tools such as Talagrand's concentration inequality (see e.g. \cite{ledoux}), one can then conclude that $\| \pi_V(X_n)\| \gg 1$ with overwhelming probability.  This concludes the claim of the Corollary in the bulk case.

The edge case is more delicate, due to the sparser spectrum near $\lambda$.  Here, one takes advantage of the identity
$$ \lambda + \frac{1}{n} X_n^* (\tilde W_{n-1}- \lambda I)^{-1} X_n - \frac{1}{\sqrt{n}} \xi_{nn} = 0,$$
(cf. \eqref{snn}), which we can rearrange as
\begin{equation}\label{arrange}
 \frac{1}{n} \sum_{j=1}^{n-1} (\lambda_j(\tilde W_{n-1})-\lambda)^{-1} |u_j(\tilde W_{n-1})^* X_n|^2 = \frac{1}{\sqrt{n}} \xi_{nn} - \lambda.
 \end{equation}
In the edge case, the right-hand side is close to $\pm 2$, and this can be used (together with the local semicircle law) to obtain the lower bound
$$ \sum_{j=1}^{n-1} (\lambda_j(\tilde W_{n-1})-\lambda)^{-2} |u_j(\tilde W_{n-1})^* X_n|^2 \gg n^{-1-\eps}$$
with overwhelming probability, which gives the claim.  See \cite{TVlocal2} for details.
\end{proof}

\begin{remark}  A slicker approach to eigenvalue delocalisation proceeds via control of the resolvent (or Green's function) $(W_n - zI)^{-1}$, taking advantage of the identity
$$ \Im ((W_n-zI)^{-1})_{jj} = \sum_{i=1}^n \frac{\eta}{(\lambda_i(W_n)-E)^2 + \eta^2} |u_i(W_n)^* e_j|^2$$
for $z = E+i\eta$; see for instance \cite{Erd} for details of this approach.  Note from \eqref{snn} that the Stieltjes transform arguments used to establish the local semicircle law already yield control on quantities such as $((W_n-zI)^{-1})_{jj}$ as a byproduct.
\end{remark}

\section{Fine-scale spectral statistics: the case of GUE}\label{gue-sec}

We now turn to the question of the fine-scale behavior of eigenvalues of Wigner matrices, starting with the model case of GUE.  Here, it is convenient to work with the fine-scale normalisation $A_n := \sqrt{n} M_n$.  For simplicity we will restrict attention to the bulk region of the spectrum, which in the fine-scale normalisation corresponds to eigenvalues $\lambda_i(A_n)$ of $A_n$ that are near $nu$ for some fixed $-2<u<2$ independent of $n$.

There are several quantities at the fine scale that are of interest to study.  For instance, one can directly study the distribution of individual (fine-scale normalised) eigenvalues $\lambda_i(A_n)$ for a single index $1 \leq i \leq n$, or more generally study the joint distribution of a $k$-tuple $\lambda_{i_1}(A_n),\ldots,\lambda_{i_k}(A_n)$ of such eigenvalues for some $1 \leq i_1 < \ldots < i_k \leq n$.  Equivalently, one can obtain estimates for expressions of the form
\begin{equation}\label{ifk}
 \E F( \lambda_{i_1}(A_n), \ldots, \lambda_{i_k}(A_n) )
\end{equation}
for various test functions $F: \R^k \to \R$.  By specializing to the case $k=2$ and to translation-invariant functions $F(x,y) := f(x-y)$, one obtains distributional information on individual eigenvalue gaps $\lambda_{i+1}(A_n) - \lambda_i(A_n)$.

A closely related set of objects to the joint distribution of individual eigenvalues are the \emph{$k$-point correlation functions} $R_n^{(k)} = R_n^{(k)}(A_n): \R^k \to \R^+$, defined via duality to be the unique symmetric function (or measure) for which one has
\begin{equation}\label{rfk}
 \int_{\R^k} F(x_1,\ldots,x_k) R_n^{(k)}(x_1,\ldots,x_k)\ dx_1 \ldots dx_k = k! \sum_{1 \leq i_1 < \ldots < i_k} \E F(\lambda_{i_1}(A_n), \ldots, \lambda_{i_k}(A_n) )
\end{equation}
for all symmetric continuous compactly supported functions $F: \R^k \to \R$.  For discrete ensembles, $R_n^{(k)}$ is only defined as a measure on $\R^k$ (which, with our conventions, has total mass $\frac{n!}{(n-k)!}$); but for continuous ensembles, $R_n^{(k)}$ is a continuous function, and for $x_1 < \ldots < x_k$, one can equivalently define $R_n^{(k)}(x_1,\ldots,x_k)$ in this case by the formula
$$R_n^{(k)}(x_1,\ldots,x_k) = \lim_{\eps \to 0} \frac{1}{\eps^k} \P( E_{\eps,x_1,\ldots,x_k} )$$
where $E_{\eps,x_1,\ldots,x_k}$ is the event that there is an eigenvalue of $A_n$ in each of the intervals $[x_i,x_i+\eps]$ for $i=1,\ldots,k$.  Alternatively, one can write
$$ R_n^{(k)}(x_1,\ldots,x_k) = \frac{n!}{(n-k)!} \int_{\R^{n-k}} \rho_n(x_1,\ldots,x_n)\ dx_{k+1} \ldots dx_n$$
where $\rho_n := \frac{1}{n!} R_n^{(n)}$ is the symmetrized joint probability distribution of all $n$ eigenvalues of $A_n$.

Note from \eqref{rfk} that control on the expressions \eqref{ifk} implies (in principle, at least) control on the $k$-point correlation function by summing over the relevant indices $i_1,\ldots,i_k$; and from eigenvalue rigidity estimates such as \eqref{ifk} we see that (for fixed $k$, at least) there are only $n^{o(1)}$ choices for the $k$-tuple $(i_1,\ldots,i_k)$ that contribute to this sum.  

From the semi-circular law, we expect that at the energy level $nu$ for some $-2 < u < 2$, the eigenvalues of $A_n$ will be spaced with average spacing $1/\rho_\sc(u)$.  It is thus natural to consider the \emph{normalised $k$-point correlation function} $\rho^{(k)}_{n,u} = \rho^{(k)}_{n,u}(A_n): \R^k \to \R^+$, defined by the formula
\begin{equation}\label{roo}
 \rho^{(k)}_{n,u}(x_1,\ldots,x_k) := R_n^{(k)}\left( nu + \frac{x_1}{\rho_\sc(u)}, \ldots, nu + \frac{x_k}{\rho_\sc(u)} \right).
\end{equation}
Informally, for infinitesimal $\eps>0$, $\eps^k \rho^{(k)}_{n,u}(x_1,\ldots,x_k)$ is approximately equal to the probability that there is an eigenvalue in each of the intervals $[nu + \frac{x_i}{\rho_\sc(u)}, nu + \frac{x_i+\eps}{\rho_\sc(u)}]$ for

The Stieltjes transform $s_n(z) = \tr( (W_n-z)^{-1} )$ introduced previously is related to the fine-scale normalised eigenvalues by the formula
$$ s_n(z) = \sum_{i=1}^n \frac{1}{\lambda_i(A_n) - nz}.$$
More generally, one can consider the random variables
\begin{equation}\label{wnz}
 \tr( (W_n-z_1)^{-1} \ldots (W_n-z_k)^{-1} ) = \sum_{1 \leq i \leq n} \frac{1}{\lambda_{i}(A_n) - nz_1} \ldots \frac{1}{\lambda_{i}(A_n) - nz_k}.
\end{equation}
The joint distribution of such random variables can be expressed in terms of the $k$-point correlations; for instance, one has
$$ \E \tr( (W_n-z_1)^{-1} \ldots (W_n-z_k)^{-1} ) = \int_\R \frac{R_n^{(1)}(x)}{(x-nz_1) \ldots (x-nz_k)}\ dx.$$
Conversely, it is possible (with some combinatorial effort) to control the $k$-point correlation function in terms of the joint distribution of such random variables; see \cite[\S 8]{EYY}.

The distribution of the eigenvalue counting functions $N_I(A_n) = N_{I/n}(W_n)$ can be expressed in terms of the distribution of the individual eigenvalues or from the correlation function.  For instance, one has (for continuous ensembles, at least) the formula
$$ \E \binom{N_I(A_n)}{k} = \frac{1}{k!} \int_{I^k} R_n^{(k)}(x_1,\ldots,x_k)\ dx_1 \ldots dx_k$$
for any $k \geq 0$ (with the convention that $\binom{n}{k}=0$ whenever $n<k$).  

Finally, to close the circle of relationships, by coupling the eigenvalue counting function $N_I(W_n)$ for intervals such as $I = [-2,x]$ with previously mentioned quantities such as the Stieltjes transforms, as well as additional \emph{level repulsion} estimates that prevent two consecutive eigenvalues from getting too close to each other too often, one can recover control of individual eigenvalues; see \cite{knowles}.  

It has been generally believed (and in many cases explicitly conjectured; see e.g. \cite[page 9]{Meh}) that the asymptotic statistics for the quantities mentioned above are \emph{universal}, in the sense that the limiting laws do not depend on the distribution of the atom variables (assuming of course that they have been normalised as stated in Definition \ref{def:Wignermatrix}).  This phenomenon was motivated by examples of similarly universal laws in physics, such as the laws of
 thermodynamics or of critical percolation; see e.g. \cite{Meh, Deibook, Deisur} for further discussion.

It is clear that if one is able to prove the universality of a limiting law, then it suffices to 
compute this law for one specific model in order to describe the asymptotic behaviour for all other models. A natural choice for the specific model is GUE, as for this model, many 
limiting laws can be computed directly thanks to the availability of an explicit formula for the joint distribution of the eigenvalues, as well as the useful identities of determinantal processes.  For instance, one has \emph{Ginibre's formula}
\begin{equation}\label{rhun}
\rho_n(x_1,\ldots,x_n) = \frac{1}{(2\pi n)^{n/2}} e^{-|x|^2/2n} \prod_{1 \leq i<j \leq n} |x_i-x_j|^2,
\end{equation}
for the joint eigenvalue distribution, as can be verified from \eqref{cmm} and a change of variables; see \cite{gin}.  From \eqref{rhun} and a downwards induction on $k$ one can then obtain the \emph{Gaudin-Mehta formula}
\begin{equation}\label{gaudin}
R^{(k)}_n(x_1,\ldots,x_k) = \det( K_n(x_i,x_j) )_{1 \leq i,j \leq k}
\end{equation}
for all $0 \leq k \leq n$, where $K_n$ is the kernel
$$ K_n(x,y) = \frac{1}{\sqrt{n}} \sum_{k=0}^{n-1} P_k(\frac{x}{\sqrt{n}}) e^{-x^2/2n} P_k(\frac{y}{\sqrt{n}}) e^{-y^2/2n}$$
and $P_0,P_1,\ldots$ are the Hermite polynomials (thus each $P_n$ is a degree $n$ polynomial, with the $P_n$ being orthonormal with respect to the measure $e^{-x^2/2}\ dx$).  This formula, combined with the classical Plancherel-Rotach asymptotics for Hermite polynomials, gives the limiting law
\begin{equation}\label{k-asym}
 \lim_{n \to \infty} \rho^{(k)}_{n,u}(x_1,\ldots,x_k) = \rho^{(k)}_{\Dyson}(x_1,\ldots,x_k)
\end{equation}
locally uniformly in $x_1,\ldots,x_k$ where
$$ \rho^{(k)}_{\Dyson}(x_1,\ldots,x_k) := \det( K_{\Dyson}(x_i,x_j) )_{1 \leq i,j \leq k}$$
and $K_{\Dyson}$ is the \emph{Dyson sine kernel}
$$ K_{\Dyson}(x,y) := \frac{\sin(\pi(x-y))}{\pi(x-y)}$$
(with the usual convention that $\frac{\sin x}{x}$ equals $1$ at the origin); see \cite{gin,Meh}.  Standard determinantal process identities then give an asymptotic
\begin{equation}\label{nnu}
 N_{[nu+a/\rho_\sc(u), nu+b/\rho_\sc(u)]}(A_n) \to \sum_{j=1}^\infty \xi_j
\end{equation}
for any fixed real numbers $a<b$ and $-2<u<2$, where the $\xi_j \in \{0,1\}$ are independent Bernoulli variables with $\E(\xi_j) = \lambda_j$, the $\lambda_j$ are the eigenvalues of the integral operator $T_{[a,b]}: L^2([a,b]) \to L^2([a,b])$ defined by
$$ T_{[a,b]} f(x) := \int_{[a,b]} K_{\Dyson}(x,y) f(y)\ dy,$$
and the convergence is in the sense of probability distributions; see e.g. \cite{HKPV}.  Thus, for instance, the probability that the interval $[nu+a/\rho_\sc(u), nu+b/\rho_\sc(u)]$ is devoid of eigenvalues converges as $n \to\infty$ to the Fredholm determinant\footnote{See \cite{JMM} for a more explicit description of this determinant in terms of a solution to an ODE.}
$$ \det(1-T_{[a,b]}) := \prod_{j=1}^\infty (1-\lambda_j).$$
Using this formula one can obtain a limiting law for an (averaged) eigenvalue spacing.	 More precisely, given an intermediate scale parameter $t_n$ such that $t_n, n/t_n \to \infty$ as $n \to \infty$, define the quantity
$$ S_n(s,u,t_n) := \frac{1}{t_n} |\{ 1 \leq i \leq n: |\lambda_i(A_n)-nu| \leq t_n/\rho_\sc(u); \lambda_{i+1}(A_n)-\lambda_i(A_n) \leq s/\rho_\sc(u) \}|.$$
Then one can establish for fixed $-2 < u < 2$ that
\begin{equation}\label{snsu}
 \E S_n(s,u,t_n) \to \int_0^s \rho(\sigma)\ d\sigma
\end{equation}
where $\rho$ is the \emph{Gaudin distribution}
$$ \rho(s) := \frac{d^2}{ds^2} \det(1-T_{[0,s]});$$
see \cite{DKMVZ}.  Informally, $\rho$ is the asymptotic distribution of the normalised gap $\rho_\sc(u) (\lambda_{i+1}(A_n)-\lambda_i(A_n))$ for typical $\lambda_i(A_n)$ near $nu$.

A variant of the computations that lead to \eqref{nnu} (and more precisely, a general central limit theorem for determinantal processes due to Costin-Leibowitz \cite{CLe} and Soshnikov \cite{Sos2}) can give a limiting law for $N_I(A_n)$ in the case of the macroscopic intervals $I=[nu,+\infty)$.  More precisely, one has the central limit theorem
$$ \frac{N_{[nu,+\infty)}(A_n) - n \int_u^\infty \rho_\sc(y)\ dy}{\sqrt{\frac{1}{2\pi^2} \log n}} \to N(0,1)_\R$$
in the sense of probability distributions, for any $-2 < u < 2$; see \cite{Gus}.  By using the counting functions $N_{[nu,+\infty)}$ to solve for the location of individual eigenvalues $\lambda_i(A_n)$, one can then conclude the central limit theorem
\begin{equation}\label{gustav}
 \frac{\lambda_i(A_n) - \lambda_i^\cl(A_n)}{\sqrt{\log n/2\pi} / \rho_\sc(u)} \to N(0,1)_\R
 \end{equation}
whenever $\lambda_i^\cl(A_n) := n \lambda_i^\cl(W_n)$ is equal to $n(u+o(1))$ for some fixed $-2 < u < 2$; see \cite{Gus}.  Informally, this asserts (in the GUE case, at least) that each eigenvalue $\lambda_i(A_n)$ typically deviates by $O( \sqrt{\log n} / \rho_\sc(u) )$ around its classical location; this result should be compared with \eqref{eigenrigid}, which has a slightly worse bound on the deviation (of shape $O_\eps(n^\eps)$ instead of $O(\sqrt{\log n})$) but which holds with overwhelming probability (and for general Wigner ensembles).

The above analysis extends to many other classes of invariant ensembles (such as GOE\footnote{There are some further direct relationships between the GOE and GUE eigenvalue distributions that can be used to deduce control of the former from that of the latter; see \cite{rains}.}, for which the joint eigenvalue distribution has a form similar to \eqref{rhun} (namely, an exponential factor and a power of a Vandermonde determinant).  However, the Hermite polynomials are usually replaced by some other family of orthogonal polynomials, and one needs additional tools (such as the theory of Riemann-Hilbert problems) to obtain enough asymptotic control on those polynomials to recover the other results of the type given here.  See \cite{Deibook} for further discussion.  However, we will not pursue this important aspect of the universality phenomenon for random matrices here, as our focus is instead on the Wigner matrix models.

\section{Extending beyond the GUE case I.  Heat flow methods}\label{heatflow}

The arguments used to establish the results in the previous section relied heavily on the special structure of the GUE ensemble, and in particular the fact that the joint probability distribution had a determinantal structure (cf. \eqref{gaudin}).  To go significantly beyond the GUE case, there are two families of techniques used.  One family are the \emph{heat flow} methods, based on applying an Ornstein-Uhlenbeck process to a Wigner ensemble $M_n^0$ to obtain a \emph{gauss divisible} ensemble $M_n^t$ that is closer to the GUE, and showing that the latter obeys approximately the same statistics as the GUE.  The other family of methods are the \emph{swapping} methods, in which one replaces the entries of one Wigner ensemble $M_n$ with another ensemble $M'_n$ which are close in some suitable sense (e.g. in the sense of matching moments), and shows that the statistics for both ensembles are close to each other.  The two methods are complementary, and many of the strongest results known about universality for Wigner matrices use a combination of both methods.

Our focus shall largely be on the swapping methods (and in particular on the \emph{four moment theorem}), but in this section we will briefly survey the heat flow techniques.  (For a more detailed survey of these methods, see \cite{Erd} or \cite{schlein}.)

Let $M_n^0$ be a Wigner matrix.  One can then define the \emph{matrix Ornstein-Uhlenbeck process} $M_n^t$ for times $t \in [0,+\infty)$ by the stochastic differential equation\footnote{One can omit the normalising term $- \frac{1}{2} M_n^t\ dt$ in this process to obtain a Brownian process rather than an Ornstein-Uhlenbeck process, but we retain the normalising term in order to keep the variance of each (off-diagonal) entry of the matrix $M_n^t$ fixed.}
$$ dM_n^t = d\beta_t - \frac{1}{2} M_n^t\ dt$$
with initial data $M_n^t|_{t=0} = M_n^0$, where $\beta_t$ is a Hermitian matrix process whose entries are standard real Brownian motions on the diagonal and standard complex Brownian motions off the diagonal, with the $\beta_t$ being independent of $M_n^0$, and with the upper-triangular entries of $\beta_t$ being jointly independent.  A standard stochastic calculus computation shows that $M_n^t$ is distributed according to the law
$$ M_n^t \equiv e^{-t/2} M_n^0 + (1-e^{-t})^{1/2} G_n,$$
where $G_n$ is a GUE matrix independent of $M_n^0$.  In particular, the random matrix $M_n^t$ is distributed as $M_n^0$ for $t=0$ and then continuously deforms towards the GUE distribution as $t \to +\infty$.  We say that a Wigner matrix is \emph{gauss divisible} (also known as a \emph{Johansson matrix}) with parameter $t$ if it has the distribution of $e^{-t/2} M_n^0 + (1-e^{-t})^{1/2} G_n$ for some Wigner matrix $M_n^0$; thus $M_n^t$ is gauss divisible with parameter $t$.  Note that not every Wigner matrix is gauss divisible; among other things, gauss divisible matrices necessarily have a continuous (and in fact smooth) distribution rather than a discrete one.  The larger one makes $t$, the more restrictive the requirement of gauss divisibility becomes, until in the asymptotic limit $t=+\infty$ the only gauss divisible ensemble remaining is GUE.

The dynamics of the (fine-scale normalised) eigenvalues $\lambda_i(A_n^t) = \sqrt{n} \lambda_i(M_n^t)$ were famously established by Dyson \cite{Dys} to be governed by the (normalised) \emph{Dyson Brownian motion} equations
\begin{equation}\label{dbm}
d\lambda_i(A_n^t) = \sqrt{n} dB_i + n \sum_{1 \leq j \leq n: j \neq i} \frac{dt}{\lambda_i(A_n^t) - \lambda_j(A_n^t)} - \frac{1}{2} \lambda_i(A_n^t)\ dt,
\end{equation}
where $B_1,\ldots,B_n$ are independent standard real Brownian motions.  

One can phrase the Dyson Brownian motion in a dual form in terms of the joint eigenvalue distribution function $\rho_n^t: \R^n \to \R^+$ at time $t$.  Namely, $\rho_n^t$ obeys the \emph{Dyson Fokker-Planck equation}
\begin{equation}\label{dfp}
\frac{\partial}{\partial t} \rho_n^t = D \rho_n^t
\end{equation}
where $D$ is the differential operator
$$ D \rho := \frac{n}{2} \sum_{i=1}^n \partial_j^2 \rho - n \sum_{1 \leq i,j \leq n: i \neq j} \partial_j\left( \frac{\rho}{x_i - x_j} \right) + \frac{1}{2} \sum_{j=1}^n \partial_j(x_j \rho)$$
and we write $\partial_j$ as shorthand for the partial derivative $\frac{\partial}{\partial x_j}$.  Observe that the Ginibre distribution $\rho_n^\infty$ defined by \eqref{rhun} is annihilated by $D$ and is thus an equilibrium state of the Dyson Fokker-Planck equation; this is of course consistent with the convergence of the distribution of $M_n^t$ to the distribution of GUE.  

The Dyson Fokker-Planck equation \eqref{dfp} can in fact be solved explicitly by observing the identity
$$ D( \Delta_n u) = \Delta_n( L u + \frac{n(n-1)}{4} u)$$
where $\Delta_n(x) := \prod_{1 \leq i<j \leq n} (x_i-x_j)$ is the Vandermonde determinant, and $L$ is the Ornstein-Uhlenbeck operator
$$ L u := \frac{n}{2} \sum_{i=1}^n \partial_j^2 u + \frac{1}{2} \sum_{i=1}^n \partial_j(x_j u).$$
This allows one to reduce the Dyson Fokker-Planck equation by a change of variables to the Ornstein-Uhlenbeck Fokker-Planck equation $\frac{\partial}{\partial t} u = Lu$, which has an explicit fundamental solution.  Using this\footnote{The derivation of these formulae in \cite{Joh1} is somewhat different, proceeding via the Harish-Chandra/Itzykson-Zuber formula \cite{harish}; however, as noted in that paper, one can also use Dyson Brownian motion to derive the formula, which is the approach taken here.}  one can obtain an explicit formula for $\rho_n^t$ in terms of $\rho_n^0$, and with a bit more effort one can also obtain a (slightly messy) determinantal formula for the associated correlation functions $(R_n^{(k)})^t$; see \cite{brezin}, \cite{Joh1}.  By exploiting these explicit formulae, Johansson \cite{Joh1} was able\footnote{Some additional technical hypotheses were assumed in \cite{Joh1}, namely that the diagonal variance $\sigma^2$ was equal to $1$, that the real and imaginary parts of each entry of $M'_n$ were independent, and that Condition \condone held for some $C_0>6$.} to extend the asymptotic \eqref{k-asym} for the $k$-point correlation function from GUE to the more general class of gauss divisible matrices with fixed parameter $t>0$ (independent of $n$).

It is of interest to extend this analysis to as small a value of $t$ as possible, since if one could set $t=0$ then one would obtain universality for all Wigner ensembles.  By optimising Johansson's method (and taking advantage of the local semi-circle law), Erd\H os, Peche, Ramirez, Schlein, and Yau \cite{EPRSY} were able to extend the universality of \eqref{k-asym} (interpreted in a suitably weak convergence topology, such as vague convergence) to gauss divisible ensembles for $t$ as small as $n^{-1+\eps}$ for any fixed $\eps>0$.

An alternate approach to these results was developed by Erd\H os, Ramirez, Schlein, Yau, and Yin \cite{ERSY}, \cite{ESY}, \cite{ESYY}.  The details are too technical to be given here, but the main idea is to use standard tools such as log-Sobolev inequalities to control the rate of convergence of the Dyson Fokker-Planck equation to the equilibrium measure $\rho_n^\infty$, starting from the initial data $\rho_n^0$; informally, if one has a good convergence to this equilibrium measure by time $t$, then one can obtain universality results for gauss divisible ensembles with this parameter $t$.  A simple model to gain heuristic intuition on the time needed to converge to equilibrium is given by the one-dimensional Ornstein-Uhlenbeck process
\begin{equation}\label{ou}
 dx = \sigma d\beta_t - \theta(x - \mu)\ dt
\end{equation}
for some parameters $\sigma, \theta > 0$, $\mu \in \R$ and some standard Brownian motion $\beta_t$.  Standard computations (or dimensional analysis) suggest that this process should converge to the equilibrium measure (in this case, a normal distribution $N(\mu,\sigma^2/2\theta)$) in time\footnote{If the initial position $x(0)$ is significantly farther away from the mean $\mu$ than the standard deviation $\sqrt{\sigma^2/2\theta}$, say $|x(0)-\mu| \sim K \sqrt{\sigma^2/2\theta}$, then one acquires an additional factor of $\log K$ in the convergence to equilibrium, because it takes time about $\log K/\theta$ for the drift term $-\theta(x-\mu)\ dt$ in \eqref{ou} to move $x$ back to within $O(1)$ standard deviations of $\mu$.  These sorts of logarithmic factors will be of only secondary importance in this analysis, ultimately being absorbed in various $O(n^\eps)$ error factors.}  $O(1/\theta)$, in the sense that the probability distribution of $x$ should differ from the equilibrium distribution by an amount that decays exponentially in $t/\theta$.

As was already observed implicity by Dyson, the difficulty with the Dyson Fokker-Planck equation \eqref{dfp} (or equivalently, the Dyson Brownian motion \eqref{dbm}) is that different components of the evolution converge to equilibrium at different rates.  Consider for instance the trace variable $T := \lambda_1(A_n^t) + \ldots + \lambda_n(A_n^t)$.  Summing up \eqref{dbm} we see that this variable evolves by the Ornstein-Uhlenbeck process
$$ dT = n d\beta_t - \frac{1}{2} T\ dt$$
for some standard Brownian motion $\beta_t$, and so we expect convergence to equilibrium for this variable in time $O(1)$.  At the other extreme, consider an eigenvalue gap $s_i := \lambda_{i+1}(A_n^t) - \lambda_i(A_n^t)$ somewhere in the bulk of the spectrum.  Subtracting two consecutive cases of \eqref{dbm}, we see that
\begin{equation}\label{sissy}
 ds_i = 2\sqrt{n} d\beta_{t,i} - \theta_i s_i\ dt + \frac{2n}{s_i}\ dt
\end{equation}
where
$$ \theta_i := n \sum_{1 \leq j \leq n: j \neq i,i+1} \frac{1}{(\lambda_{i+1}(A_n^t) - \lambda_j(A_n^t))(\lambda_i(A_n^t) - \lambda_j(A_n^t))} + \frac{1}{2}.$$
Using the heuristic $\lambda_j(A_n) \approx \lambda_j^\cl(A_n)$, we expect $\theta_i$ to be of size comparable to $n$ and $s_i$ to be of size comparable to $1$; comparing \eqref{sissy} with \eqref{ou} we thus expect $s_i$ to converge to equilibrium in time $O(1/n)$.

One can use a standard log-Sobolev argument of Bakry and Emery \cite{bakry} (exploiting the fact that the equilibrium measure $\rho_n^\infty$ is the negative exponential of a strictly convex function $H$) to show (roughly speaking) that the Dyson Brownian motion converges to global equilibrium in time $O(1)$; see e.g. \cite{Erd}.  Thus the trace variable $T$ is among the slowest of the components of the motion to converge to equilibrium.  However, for the purposes of controlling local statistics such as the normalised $k$-point correlation function
$$ \rho^{(k)}_{n,u}(x_1,\ldots,x_k) := R_n^{(k)}( nu + \frac{x_1}{\rho_\sc(u)}, \ldots, nu + \frac{x_k}{\rho_\sc(u)} ),$$
and particularly the \emph{averaged} normalised $k$-point correlation function
$$ \frac{1}{2b} \int_{u_0-b}^{u_0+b} R_n^{(k)}( nu + \frac{x_1}{\rho_\sc(u_0)}, \ldots, nu + \frac{x_k}{\rho_\sc(u_0)} )\ du,$$
these ``slow'' variables turn out to essentially be irrelevant, and it is the ``fast'' variables such as $s_i$ which largely control these expressions.  As such, one expects these particular quantities to converge to their equilibrium limit at a much faster rate.  By replacing the global equilibrium measure $\rho_n^\infty$ with a localized variant which has better convexity properties in the slow variables, it was shown in \cite{ESY}, \cite{ESYY} by a suitable modification of the Bakry-Emery argument that one in fact had convergence to equilibrium for such expressions in time $O(n^{-1+\eps})$ for any fixed $\eps$; a weak version\footnote{Roughly speaking, the rigidity result that is needed is that one has $\lambda_i(W_n) = \lambda_i^\cl(W_n) + O(n^{1/2-c})$ in an $\ell^2$-averaged sense for some absolute constant $c>0$.  See \cite{Erd} for details.} of the rigidity of eigenvalues statement \eqref{eigenrigid} is needed in order to show that the error incurred by replacing the actual equilibrium measure with a localized variant is acceptable.  Among other things, this argument reproves a weaker version of the result in \cite{EPRSY} mentioned earlier, in which one obtained universality for the asymptotic \eqref{k-asym} after an additional averaging in the energy parameter $u$.  However, the method was simpler and more flexible than that in \cite{EPRSY}, as it did not rely on explicit identities, and has since been extended to many other types of ensembles, including the real symmetric analogue of gauss divisible ensembles in which the role of GUE is replaced instead by GOE.  Again, we refer the reader to \cite{Erd} for more details.

\section{Extending beyond the GUE case II.  Swapping and the Four Moment Theorem}\label{swap-sec}

The heat flow methods discussed in the previous section enlarge the class of Wigner matrices to which GUE-type statistics are known to hold, but do not cover all such matrices, and in particular leave out discrete ensembles such as the Bernoulli ensembles, which are not gauss divisible for any $t>0$.  To complement these methods, we have a family of \emph{swapping} methods to extend spectral asymptotics from one ensemble to another, based on individual replacement of each coefficient of a Wigner matrix, as opposed to deforming all of the coefficients simultaneously via heat flow.

The simplest (but rather crude) example of a swapping method is based on the \emph{total variation distance} $d(X,Y)$ between two random variables $X, Y$ taking values in the same range $R$, defined by the formula
$$ d(X,Y) := \sup_{E \subset R} |\P( X \in E ) - \P( Y \in E )|,$$
where the supremum is over all measurable subsets $E$ of $R$ 	Clearly one has
$$ |\E F(X) - \E F(Y)| \leq \|F\|_{L^\infty} d(X,Y)$$
for any measurable function $F: R \to \R$.  As such, if $d(M_n,M'_n)$ is small, one can approximate various spectral statistics of $M_n$ by those of $M'_n$, or vice versa.  For instance, one has
$$ |\E F( \lambda_{i_1}(A_n),\ldots,\lambda_{i_k}(A_n) ) - F( \lambda_{i_1}(A'_n),\ldots,\lambda_{i_k}(A'_n) )|
\leq \|F\|_{L^\infty} d(M_n,M'_n),$$
and thus from \eqref{rfk} we have the somewhat crude bound
\begin{align*}
|\int_{\R^k} F(x_1,\ldots,x_k) R_n^{(k)}(A_n)(x_1,\ldots,x_k)&- 
 F(x_1,\ldots,x_k) R_n^{(k)}(A'_n)(x_1,\ldots,x_k)\ dx_1 \ldots dx_k|\\
& \leq n^k \|F\|_{L^\infty} d(M_n,M'_n).
\end{align*}
and hence by \eqref{roo}
\begin{align*}
|\int_{\R^k} F(x_1,\ldots,x_k) \rho_{n,u}^{(k)}(A_n)(x_1,\ldots,x_k)&- 
F(x_1,\ldots,x_k) \rho_{n,u}^{(k)}(A'_n)(x_1,\ldots,x_k)\ dx_1 \ldots dx_k| \\
& \leq (\rho_\sc(u) n)^k \|F\|_{L^\infty} d(M_n,M'_n)
\end{align*}
for any test function $F$.  On the other hand, by swapping the entries of $M_n$ with $M'_n$ one at a time, we see that
$$ d(M_n,M'_n) \leq \sum_{i=1}^n \sum_{j=1}^n d(\xi_{ij}, \xi'_{ij}).$$
We thus see that if $d(\xi_{ij},\xi'_{ij}) \leq n^{-C}$ for a sufficiently large constant $C$ (depending on $k$), then the $k$-point correlation functions of $M_n$ and $M'_n$ are asymptotically equivalent.  This argument was quite crude, costing many more powers of $n$ than is strictly necessary, and by arguing more carefully one can reduce this power; see \cite{EPRSY}.  However, it does not seem possible to eliminate the factors of $n$ entirely from this type of argument.

By combining this sort of total variation-based swapping argument with the heat flow universality results for time $t = n^{-1+\eps}$, the asymptotic \eqref{k-asym} was demonstrated to hold in \cite{EPRSY} for Wigner matrices with sufficiently smooth distribution; in particular, the $k=2$ case of \eqref{k-asym} was established if the distribution function of the atom variables were $C^6$ (i.e. six times continuously differentiable) and obeyed a number of technical decay and positivity conditions that we will not detail here.  The basic idea is to approximate the distribution $\rho$ of an atom variable $\xi_{ij}$ in total variation distance (or equivalently, in $L^1$ norm) by the distribution $e^{tL} \tilde \rho$ of a gauss-divisible atom variable $\xi'_{ij}$ with an accuracy that is better than $n^{-C}$ for a suitable $C$, where $L$ is the generator of the Ornstein-Uhlenbeck process and $t = n^{-1+\eps}$; this can be accomplished by setting $\tilde\rho$ to essentially be a partial Taylor expansion of the (formal) backwards Ornstein-Uhlenbeck evolution $e^{-tL} \rho$ of $\rho$ to some bounded order, with the smoothness of $\rho$ needed to ensure that this partial Taylor expansion $\tilde \rho$ remains well-defined as a probability distribution, and that $e^{tL} \tilde \rho$ approximates $\rho$ sufficiently well.  See \cite{EPRSY} for more details of this method (referred to as the \emph{method of reverse heat flow} in that paper).

Another fundamental example of a swapping method is the \emph{Lindeberg exchange strategy} \footnote{We would like to thank S. Chatterjee and M. Krisnapur
for introducing this method to us}, introduced in Lindeberg's classic proof \cite{lindeberg} of the central limit theorem, and first applied to Wigner ensembles in \cite{Chat}.  We quickly sketch that proof here.  Suppose that $X_1,\ldots,X_n$ are iid real random variables with mean zero and variance one, and let $Y_1,\ldots,Y_n$ be another set of iid real random variables and mean zero and variance one (which we may assume to be independent of $X_1,\ldots,X_n$).  We would like to show that the averages $\frac{X_1+\ldots+X_n}{\sqrt{n}}$ and $\frac{Y_1+\ldots+Y_n}{\sqrt{n}}$ have asymptotically the same distribution, thus
$$ \E F\left( \frac{X_1+\ldots+X_n}{\sqrt{n}} \right) = \E F\left( \frac{Y_1+\ldots+Y_n}{\sqrt{n}} \right) + o(1)$$
for any smooth, compactly supported function $F$.  The idea is to swap the entries $X_1,\ldots,X_n$ with $Y_1,\ldots,Y_n$ one at a time and obtain an error of $o(1/n)$ on each such swap.  For sake of illustration we shall just establish this for the first swap:
\begin{equation}\label{swap}
 \E F\left( \frac{X_1+\ldots+X_n}{\sqrt{n}} \right) = \E F\left( \frac{X_1+\ldots+X_{n-1}+Y_n}{\sqrt{n}} \right) + o(1/n).
 \end{equation}
We write $\frac{X_1+\ldots+X_n}{\sqrt{n}} = S + n^{-1/2} X_n$, where $S := \frac{X_1+\ldots+X_{n-1}}{\sqrt{n}}$.  From Taylor expansion we see (for fixed smooth, compactly supported $F$) that
$$ F\left( \frac{X_1+\ldots+X_n}{\sqrt{n}} \right) = F(S) + n^{-1/2} X_n F'(S) + \frac{1}{2} n^{-1} X_n^2 F''(S) + O( n^{-3/2} |X_n|^3 ).$$
We then make the crucial observation that $S$ and $X_n$ are independent.  On taking expectations (and assuming that $X_n$ has a bounded third moment) we conclude that
$$ \E F\left( \frac{X_1+\ldots+X_n}{\sqrt{n}} \right) = \E F(S) + n^{-1/2} (\E X_n) \E F'(S) + \frac{1}{2} n^{-1}(\E X_n^2) \E F''(S) + O( n^{-3/2} ).$$
Similarly one has
$$ \E F\left( \frac{X_1+\ldots+X_{n-1}+Y_n}{\sqrt{n}} \right) = \E F(S) + n^{-1/2} (\E Y_n) \E F'(S) + \frac{1}{2} n^{-1}(\E Y_n^2) \E F''(S) + O( n^{-3/2} ).$$
But by hypothesis, $X_n$ and $Y_n$ have \emph{matching moments to second order}, in the sense that $\E X_n^i = \E Y_n^i$ for $i = 0,1,2$.  Thus, on subtracting, we obtain \eqref{swap} (with about a factor of $n^{-1/2}$ to spare; cf. the Berry-Ess\'een theorem \cite{berry}, \cite{esseen}).

Note how the argument relied on the matching moments of the two atom variables $X_i, Y_i$; if one had more matching moments, one could continue the Taylor expansion and obtain further improvements to the error term in \eqref{swap}, with an additional gain of $n^{-1/2}$ for each further matching moment.

We can apply the same strategy to control expressions such as $\E F(M_n) - F(M'_n)$, where $M_n, M'_n$ are two (independent) Wigner matrices.  If one can obtain bounds such as
$$\E F(M_n) - \E F(\tilde M_n) = o(1/n)$$
when $\tilde M_n$ is formed from $M_n$ by replacing\footnote{Technically, the matrices $\tilde M_n$ formed by such a swapping procedure are not Wigner matrices as defined in Definition \ref{def:Wignermatrix}, because the diagonal or upper-triangular entries are no longer identically distributed.  However, all of the relevant estimates for Wigner matrices can be extended to the non-identically-distributed case at the cost of making the notation slightly more complicated.  As this is a relatively minor issue, we will not discuss it further here.} one of the diagonal entries $\xi_{ii}$ of $M_n$ by the corresponding entry $\xi'_{ii}$ of $M'_n$, and bounds such as
$$\E F(M_n) - \E F(\tilde M_n) = o(1/n^2)$$
when $\tilde M_n$ is formed from $M_n$ by replacing one of the off-diagonal entries $\xi_{ij}$ of $M_n$ with the corresponding entry $\xi'_{ij}$ of $M'_n$ (and also replacing $\xi_{ji} = \overline{\xi_{ij}}$ with $\xi'_{ji} = \overline{\xi'_{ij}}$, to preserve the Hermitian property), then on summing an appropriate telescoping series, one would be able to conclude asymptotic agreement of the statistics $\E F(M_n)$ and $\E F(M'_n)$:
\begin{equation}\label{famn}
\E F(M_n) - \E F(M'_n) = o(1)
\end{equation}

As it turns out, the numerology of swapping for matrices is similar to that for the central limit theorem, in that each matching moment leads to an additional factor of $O(n^{-1/2})$ in the error estimates.  From this, one can expect to obtain asymptotics of the form \eqref{famn} when the entries of $M_n, M'_n$ match to second order on the diagonal and to fourth order off the diagonal; informally, this would mean that $\E F(M_n)$ depends only on the first four moments of the entries (and the first two moments of the diagonal entries).  In the case of statistics arising from eigenvalues or eigenvectors, this is indeed the case, and the precise statement is known as the \emph{Four Moment Theorem}.  

We first state the Four Moment Theorem for eigenvalues.  

\begin{definition}[Matching moments]  Let $k \geq 1$.  Two complex random variables $\xi, \xi'$ are said to \emph{match to order $k$} if one has $\E \Re(\xi)^a \Im(\xi)^b = \E \Re(\xi')^a \Im(\xi')^b$ whenever $a, b \geq 0$ are integers such that $a+b \leq k$.
\end{definition}

In the model case when the real and imaginary parts of $\xi$ or of $\xi'$ are independent, the matching moment condition simplifies to the assertion that $\E \Re(\xi)^a = \E \Re(\xi')^a$ and $\E \Im(\xi)^b = \E \Im(\xi')^b$ for all $0 \leq a, b \leq k$.
 
\begin{theorem}[Four Moment Theorem for eigenvalues]\label{theorem:Four} 
Let $c_0 > 0$ be a sufficiently small constant.
 Let $M_n = (\xi_{ij})_{1 \leq i,j \leq n}$ and $M'_n = (\xi'_{ij})_{1 \leq i,j \leq n}$ be
 two Wigner matrices obeying Condition \condone for some sufficiently large absolute constant $C_0$. Assume furthermore that for any $1 \le  i<j \le n$, $\xi_{ij}$ and
 $\xi'_{ij}$  match to order $4$ 
  and for any $1 \le i \le n$, $\xi_{ii}$ and $\xi'_{ii}$ match  to order $2$.  Set $A_n := \sqrt{n} M_n$ and $A'_n := \sqrt{n} M'_n$, let $1 \leq k \leq n^{c_0}$ be an integer,  and let $G: \R^k \to \R$ be a smooth function obeying the derivative bounds
\begin{equation}\label{G-deriv}
|\nabla^j G(x)| \leq n^{c_0}
\end{equation}
for all $0 \leq j \leq 5$ and $x \in \R^k$.
 Then for any $1 \le i_1 < i_2 \dots < i_k \le n$, and for $n$ sufficiently large we have
\begin{equation} \label{eqn:approximation}
 |\E ( G(\lambda_{i_1}(A_n), \dots, \lambda_{i_k}(A_n))) -
 \E ( G(\lambda_{i_1}(A'_n), \dots, \lambda_{i_k}(A'_n)))| \le n^{-c_0}.
\end{equation}
\end{theorem}

We remark that in the papers \cite{TVlocal1}, \cite{TVlocal2}, \cite{TVlocal3}, a variant of the above result, which we called the \emph{three moment theorem}, was asserted, in which the hypothesis of four matching moments off the diagonal was relaxed to three matching moments (and no moment matching was required on the diagonal), but for which the bound \eqref{G-deriv} was improved to $|\nabla^j G(x)| \leq n^{-Cjc_0}$ for some sufficiently large absolute constant $C>0$.  Unfortunately, the proof given of the three moment theorem in these papers was not correct as stated, although the claim can still be proven in most cases by other means; see Appendix \ref{erratum}.

A preliminary version of Theorem \ref{theorem:Four} was first established by the authors in \cite{TVlocal1}, in the case\footnote{In the paper, $k$ was held fixed, but an inspection of the argument reveals that it extends without difficulty to the case when $k$ is as large as $n^{c_0}$, for $c_0$ small enough.} of bulk eigenvalues (thus $\delta n \leq i_1,\ldots,i_k \leq (1-\delta) n$ for some absolute constant $\delta > 0$) and assuming Condition \condo\ instead of Condition \condone.  In \cite{TVlocal2}, the restriction to the bulk was removed; and in \cite{TVlocal3}, Condition {\condo} was relaxed to Condition \condone\ for a sufficiently large value of $C_0$.  We will discuss the proof of this theorem in Section \ref{sketch}.

The following technical generalization of the Four Moment Theorem, in which the entries of $M_n,M'_n$ only match approximately rather than exactly, is useful for some applications.

\begin{proposition}[Four Moment Theorem, approximate moment matching case]\label{approx-moment}  The conclusions of Theorem \ref{theorem:Four} continue to hold if the requirement that $\xi_{ij}$ and $\xi'_{ij}$ match to order $4$ is relaxed to the conditions
$$ |\E \Re(\xi_{ij})^a \Im(\xi_{ij})^b - \E \Re(\xi'_{ij})^a \Im(\xi'_{ij})^b| \leq \eps_{a+b}$$
whenever $a,b \geq 0$ and $a+b \leq 4$, where
$$ \eps_0=\eps_1=\eps_2:=0; \quad \eps_3 := n^{-1/2-Cc_0}; \quad \eps_4 := n^{-Cc_0}$$
for some absolute constant $C>0$.
\end{proposition}

This proposition follows from an inspection of the proof of Theorem \ref{theorem:Four}: see Section \ref{sketch}.

A key technical result used in the proof of the Four Moment Theorem, which is also of independent interest, is the \emph{gap theorem}:

\begin{theorem}[Gap theorem]\label{gap}  Let $M_n$ be a Wigner matrix obeying Condition \condone\ for a sufficiently large absolute constant $C_0$.  Then for every $c_0>0$ there exists a $c_1>0$ (depending only on $c_0$) such that
$$ \P( |\lambda_{i+1}(A_n) - \lambda_i(A_n)| \leq n^{-c_0} ) \ll n^{-c_1}$$
for all $1 \leq i < n$.
\end{theorem}

We discuss this theorem in Section \ref{sketch}.  Among other things, the gap theorem tells us that eigenvalues of a Wigner matrix are usually simple.  Closely related \emph{level repulsion} estimates were established (under an additional smoothness hypothesis on the atom distributions) in \cite{ESY3}.

Another variant of the Four Moment Theorem was subsequently introduced in \cite{EYY}, in which the eigenvalues $\lambda_{i_j}(A_n)$ appearing in Theorem \ref{theorem:Four} were replaced by expressions such as \eqref{wnz} that are derived from the resolvent (or Green's function) $(W_n-z)^{-1}$, but with slightly different technical hypotheses on the matrices $M_n, M'_n$; see \cite{EYY} for full details.  As the resolvent-based quantities \eqref{wnz} are averaged statistics that sum over many eigenvalues, they are far less sensitive to the eigenvalue repulsion phenomenon than the individual eigenvalues, and as such the version of the Four Moment Theorem for Green's function has a somewhat simpler proof (based on resolvent expansions rather than the Hadamard variation formulae and Taylor expansion).  Conversely, though, to use the Four Moment Theorem for Green's function to control individual eigenvalues, while possible, requires a significant amount of additional argument; see \cite{knowles}.
 
We now discuss the extension of the Four Moment Theorem to eigenvectors rather than eigenvalues.  Recall that we are using $u_1(M_n),\ldots,u_n(M_n)$ to denote the unit eigenvectors of a Hermitian matrix $M_n$ associated to the eigenvalues $\lambda_1(M_n),\ldots,\lambda_n(M_n)$, thus $u_1(M_n),\ldots,u_n(M_n)$ lie in the unit sphere $S^{2n-1} := \{ z \in \C^n: |z| = 1 \}$ is the unit sphere of $\C^n$.  We write $u_{i,p}(M_n) = u_i(M_n)^* e_p$ for the $p$-th coefficient of $u_i(M_n)$ for each $1 \leq i,p \leq n$.  If $M_n$ is not  Hermitian, but is in fact real symmetric, then we can require the $u_i(M_n)$ to have real coefficients, thus taking values in the unit sphere $S^{n-1} := \{ x \in \R^n: |x|=1\}$ of $\R^n$.

Unfortunately, the eigenvectors $u_i(M_n)$ are not unique in either the Hermitian or real symmetric cases; even if one assumes that the spectrum of $M_n$ is \emph{simple}, in the sense that
$$ \lambda_1(M_n) <  \ldots < \lambda_n(M_n),$$ 
one has the freedom to rotate each $u_i(M_n)$ by a unit phase $e^{\sqrt{-1}\theta} \in U(1)$.  In the real symmetric case, in which we force the eigenvectors to have real coefficients, one only has the freedom to multiply each $u_i(M_n)$ by a sign $\pm \in O(1)$.  There are a variety of ways to eliminate this ambiguity.  For sake of concreteness we will remove the ambiguity by working with the orthogonal projections $P_i(M_n)$ to the eigenspace at eigenvalue $\lambda_i(M_n)$; if this eigenvalue is simple, we simply have $P_i(M_n) := u_i(M_n) u_i(M_n)^*$.

\begin{theorem}[Four Moment Theorem for eigenvectors]\label{theorem:Four2}  Let $c_0,M_n,M'_n,C_0,A_n,A'_n,k$ be as in Theorem \ref{theorem:Four}.  Let $G: \R^k \times \C^{k} \to \R$ be a smooth function obeying the derivative bounds
\begin{equation}\label{G-deriv-2}
|\nabla^j G(x)| \leq n^{c_0}
\end{equation}
for all $0 \leq j \leq 5$ and $x \in \R^k \times \C^{k}$.
 Then for any $1 \le i_1, i_2, \ldots, i_k \le n$ and $1 \leq p_1,\ldots,p_k,q_1,\ldots,q_k \leq n$, and for $n$ sufficiently large depending on $\eps,c_0, C_0$ we have
\begin{equation} \label{eqn:approximation-0}
|\E G( \Phi(A_n) ) - \E G( \Phi(A'_n) )| \leq n^{-c_0}
\end{equation}
where for any matrix $M$ of size $n$, $\Phi(M) \in \R^k \times \C^{k}$ is the tuple
$$ \Phi(M) := \left( (\lambda_{i_a}(M))_{1 \leq a \leq k}, ( n P_{i_a,p_a,q_a}(M) )_{1 \leq a \leq k} \right),$$
and $P_{i,p,q}(M)$ is the $pq$ coefficient of the projection $P_i(M)$.  The bounds are uniform in the choice of $i_1,\ldots,i_k,p_1,\ldots,p_k,q_1,\ldots,q_k$.
\end{theorem}

Theorem \ref{theorem:Four2} extends (the first part of) Theorem \ref{theorem:Four2}, which deals with the case where the function $G$ only depends on the $\R^k$ component of $\R^k \times \C^k$.  This theorem as stated appears in \cite{TV-vector}; a slight variant\footnote{Besides the differences in the methods of proof, the hypotheses of the result in \cite{knowles} differ in some technical aspects from those in Theorem \ref{theorem:Four2}.  For instance, Condition \condone is replaced with Condition \condo, and $k$ is restricted to be bounded, rather than being allowed to be as large as $n^{c_0}$.   On the other hand, the result is sharper at the edge of the spectrum (one only requires matching up to two moments, rather than up to four), and the result can be extended to ``generalized Wigner matrices'' for which the variances of the entries are allowed to differ, provided that a suitable analogue of Theorem \ref{gap} holds.}  of the theorem (proven via the Four Moment Theorem for Green's function) was simultaneously\footnote{More precisely, the results in \cite{TV-vector} were announced at the AIM workshop ``Random matrices'' in December 2010 and appeared on the arXiv in March 2011.  A preliminary version of the results in \cite{knowles} appeared in February 2011, with a final version appearing in March 2011.} established in \cite{knowles}.

We also remark that the Four Moment Theorem for eigenvectors (in conjunction with the eigenvalue rigidity bound \eqref{eigenrigid}) can be used to establish a variant of the Four Moment Theorem for Green's function, which has the advantage of being applicable all the way up to the real axis (assuming a level repulsion hypothesis); see \cite{TVeigenvector}.

\subsection {The necessity of Four Moments} 

It is a natural question to ask whether the requirement of four matching moments (or four approximately matching moments, as in Proposition \ref{approx-moment}) is genuinely necessary.  As far as the distribution of individual eigenvalues $\lambda_i(A_n)$ are concerned, the answer is essentially ``yes'', even in the identically distributed case, as the following result from \cite{TVnec} shows.

\begin{theorem}[Necessity of fourth moment hypothesis]\label{theorem:nec}  Let $M_n, M'_n$ be real symmetric Wigner matrices whose atom variables $\xi,\xi'$ have vanishing third moment $\E \xi^3 = \E (\xi')^3=0$ but with distinct fourth moments $\E \xi^4 \neq \E (\xi')^4$.  Then for all sufficiently large $n$, one has
$$ \frac{1}{n} \sum_{i=1}^n | \E \lambda_i(A_n) - \E \lambda_i(A'_n)  | \geq \kappa$$
for some $\kappa$ depending only on the atom distributions.
\end{theorem}

This result is established by combining a computation of the fourth moment $\sum_{i=1}^n \lambda_i(A_n)^4$ with eigenvalue rigidity estimates such as \eqref{eigenrigid}.  Informally, it asserts that on average, the \emph{mean} value of $\lambda_i(A_n)$ is sensitive to the fourth moment of the atom distributions at the scale of the mean eigenvalue spacing (which is comparable to $1$ in the bulk at least).  In contrast, the Four Moment Theorem morally\footnote{This is not quite true as stated, because of the various error terms in the Four Moment Theorem, and the requirement that the function $G$ in that theorem is smooth.  A more accurate statement (cf. the proof of Theorem \ref{gust} below) is that if the median of $\lambda_i(A_n)$ is $M$ (thus $\P( \lambda_i(A_n) \leq M ) = 1/2$, in the continuous case at least), then one has 
$$\P( \lambda_i(A'_n) \leq M + n^{-c_0} ), \P( \lambda_i(A'_n) \geq M - n^{-c_0} ) \geq 1/2-n^{-c_0},$$
which almost places the median of $\lambda_i(A'_n)$ within $O(n^{-c_0})$ of $M$.} asserts that when the atom variables of $M_n$ and $M'_n$ match to fourth order, then the \emph{median} of $\lambda_i(A_n)$ and of $\lambda_i(A'_n)$ only differ by $O(n^{-c_0})$.  Thus, Theorem \ref{theorem:nec} and Theorem \ref{theorem:Four} are not directly comparable to each other.  Nevertheless it is expected that the mean and median of $\lambda_i(A_n)$ should be asymptotically equal at the level of the mean eigenvalue spacing, although this is just beyond the known results (such as \eqref{eigenrigid}) on the distribution of these eigenvalues.  As such, Theorem \ref{theorem:nec} provides substantial evidence that the Four Moment Theorem breaks down if one does not have any sort of matching at the fourth moment.

By computing higher moments of $\lambda_i(A_n)$, it was conjectured in \cite{TVnec} that one has an asymptotic of the form
\begin{equation}\label{lamconj}
\E \lambda_i(A_n) = n \lambda_i^{\cl}(W_n) + C_{i,n} + \frac{1}{4} (\lambda_i^{\cl}(W_n)^3 - 2 \lambda_i^\cl(W_n)) \E \xi^4 + O(n^{-c} )
\end{equation}
for all $i$ in the bulk region $\delta n \leq i \leq (1-\delta) n$, where $C_{i,n}$ is a quantity independent of the atom distribution $\xi$.  (At the edge, the dependence on the fourth moment is weaker, at least when compared against the (now much wider) mean eigenvalue spacing; see Section \ref{dist-sec}.)

We remark that while the statistics of \emph{individual} eigenvalues are sensitive to the fourth moment, \emph{averaged} statistics such as the $k$-point correlation functions $\rho_{n,u}^{(k)}$ are much less sensitive to this moment (or the third moment).  Indeed, this is already visible from the results in Section \ref{heatflow}, as gauss divisible matrices can have a variety of possible third or fourth moments for their atom distributions (see Lemma \ref{match}).

\section{Sketch of proof of four moment theorem}\label{sketch}

In this section we discuss the proof of Theorem \ref{theorem:Four} and Theorem \ref{gap}, following the arguments that originated in \cite{TVlocal1} and refined in \cite{TVlocal3}.  To simplify the exposition, we will just discuss the four moment theorem; the proof of the approximate four moment theorem in Proposition \ref{approx-moment} is established by a routine modification of the argument.

For technical reasons, the two theorems need to be proven together.  Let us say that a Wigner matrix $M_n$ has the \emph{gap property} if it obeys the conclusion of Theorem \ref{gap}; thus Theorem \ref{gap} asserts that all Wigner matrices obeying Condition \condone\ for sufficiently large $C_0$ have the gap property.  We do not know of a direct proof of this result that does not go through the Four Moment Theorem; however, it is possible to establish an independent proof of a more restrictive result:

\begin{theorem}[Gap theorem, special case]\label{gap-special}  Any Wigner matrix obeying Condition \condo\ has the gap property.
\end{theorem}

We discuss this theorem (which is \cite[Theorem 19]{TVlocal1}) later in this section.  Another key ingredient is the following truncated version of the Four Moment Theorem, in which one removes the event that two consecutive eigenvalues are too close to each other.  For technical reasons, we need to introduce quantities
$$ Q_i(A_n) := \sum_{j \neq i} \frac{1}{|\lambda_j(A_n) - \lambda_i(A_n)|^2} $$
for $i=1,\ldots,n$, which is a regularised measure of extent to which $\lambda_i(A_n)$ is close to any other eigenvalue of $A_n$.

\begin{theorem}[Truncated Four Moment Theorem]\label{trunc}
Let $c_0 > 0$ be a sufficiently small constant.
 Let $M_n = (\xi_{ij})_{1 \leq i,j \leq n}$ and $M'_n = (\xi'_{ij})_{1 \leq i,j \leq n}$ be
 two Wigner matrices obeying Condition \condone for some sufficiently large absolute constant $C_0$. Assume furthermore that for any $1 \le  i<j \le n$, $\xi_{ij}$ and
 $\xi'_{ij}$  match to order $4$ 
  and for any $1 \le i \le n$, $\xi_{ii}$ and $\xi'_{ii}$ match  to order $2$.  Set $A_n := \sqrt{n} M_n$ and $A'_n := \sqrt{n} M'_n$, let $1 \leq k \leq n^{c_0}$ be an integer, and let 
$$ G = G(\lambda_{i_1},\ldots,\lambda_{i_k},Q_{i_1},\ldots,Q_{i_k})$$
be a smooth function from $\R^k \times \R^k_+$ to $\R$ that is supported in the region
\begin{equation}\label{qii}
 Q_{i_1},\ldots,Q_{i_k} \leq n^{c_0}
\end{equation}
and obeys the derivative bounds
\begin{equation}\label{geo}
 |\nabla^j G(\lambda_{i_1},\ldots,\lambda_{i_k},Q_{i_1},\ldots,Q_{i_k})| \leq n^{c_0}
\end{equation}
for all $0 \leq j \leq 5$.  Then
\begin{equation}\label{beg}
\begin{split}
& \E G( \lambda_{i_1}(A_n),\ldots,\lambda_{i_k}(A_n),Q_{i_1}(A_n),\ldots,Q_{i_k}(A_n)) =\\
&\quad 
  \E G( \lambda_{i_1}(A'_n),\ldots,\lambda_{i_k}(A'_n),Q_{i_1}(A'_n),\ldots,Q_{i_k}(A'_n)) + O(n^{-1/2+O(c_0)}.
\end{split}
 \end{equation}
\end{theorem}

We will discuss the proof of this theorem shortly.  Applying Theorem \ref{trunc} with $k=1$ and a function $G$ that depends only a single variable $Q_i$, and using the gap property to bound $Q_i$ (cf. \cite[Lemma 49]{TVlocal1}), one can show a four moment property for the gap theorem: if $M_n, M'_n$ are Wigner matrices obeying Condition \condone\ for a sufficiently large $C_0$ which match to fourth order, and $M_n$ obeys the gap property, then $M'_n$ also obeys the gap property.  Using this and Theorem \ref{gap-special}, one can then obtain Theorem \ref{gap} in full generality.  Using Theorem \ref{gap}, one can then deduce Theorem \ref{theorem:Four} from Theorem \ref{trunc} by smoothly truncating in the $Q$ variables: see \cite[\S 3.3]{TVlocal1}.

It remains to establish Theorem \ref{gap-special} and Theorem \ref{trunc}.  We begin with Theorem \ref{trunc}.  To simplify the exposition slightly, let us assume that the matrices $M_n, M'_n$ are real symmetric rather than Hermitian.  To reduce the number of parameters, we will also set $C_0 := 1/c_0$.

As indicated in Section \ref{swap-sec}, the basic idea is to use the Lindeberg exchange strategy.  To illustrate the idea, let $\tilde M_n$ be the matrix formed from $M_n$ by replacing a single entry $\xi_{pq}$ of $M_n$ with the corresponding entry $\xi'_{pq}$ of $M'_n$ for some $p<q$, with a similar swap also being performed at the $\xi_{qp}$ entry to keep $\tilde M_n$ Hermitian.  Strictly speaking, $\tilde M_n$ is not a Wigner matrix as defined in Definition \ref{def:Wignermatrix}, as the entries are no longer identically distributed, but this will not significantly affect the arguments.  (One also needs to perform swaps on the diagonal, but this can be handled in essentially the same manner.)

Set $\tilde A_n := \sqrt{n} \tilde M_n$ as usual.  
We will sketch the proof of the claim that
\begin{align*}
& \E G( \lambda_{i_1}(A_n),\ldots,\lambda_{i_k}(A_n),Q_{i_1}(A_n),\ldots,Q_{i_k}(A_n)) \\
&\quad = 
  \E G( \lambda_{i_1}(\tilde A_n),\ldots,\lambda_{i_k}(\tilde A_n),Q_{i_1}(\tilde A_n),\ldots,Q_{i_k}(\tilde A_n)) + O(n^{-5/2+O(c_0)};
  \end{align*}
by telescoping together $O(n^2)$ estimates of this sort one can establish \eqref{beg}.  (For swaps on the diagonal, one only needs an error term of $O(n^{-3/2+O(c_0)})$, since there are only $O(n)$ swaps to be made here rather than $O(n^2)$.  This is ultimately why there are two fewer moment conditions on the diagonal than off it.)

We can write $A_n = A(\xi_{pq})$, $\tilde A_n = A(\xi'_{pq})$, where 
$$
A(t) = A(0) + t A'(0)
$$
is a (random) Hermitian matrix depending linearly\footnote{If we were working with Hermitian matrices rather than real symmetric matrices, then one could either swap the real and imaginary parts of the $\xi_{ij}$ separately (exploiting the hypotheses that these parts were independent), or else repeat the above analysis with $t$ now being a complex parameter (or equivalently, two real parameters) rather than a real one.  In the latter case, one needs to replace all instances of single variable calculus below (such as Taylor expansion) with double variable calculus, but aside from notational difficulties, it is a routine matter to perform this modification.}  on a real parameter $t$, with $A(0)$ being a Wigner matrix with one entry (and its adjoint) zeroed out, and $A'(0)$ is the explicit elementary Hermitian matrix 
\begin{equation}\label{apo}
A'(0) = e_p e_q^* + e_p^* e_q.
\end{equation}
We note the crucial fact that the random matrix $A(0)$ is independent of both $\xi_{pq}$ and $\xi'_{pq}$.  Note from Condition \condone\ that we expect $\xi_{pq}, \xi'_{pq}$ to have size $O(n^{O(c_0)})$ most of the time, so we should (heuristically at least) be able to restrict attention to the regime $t = O(n^{O(c_0)})$.  If we then set
\begin{equation}\label{ftt}
 F(t) := \E G( \lambda_{i_1}(A(t)),\ldots,\lambda_{i_k}(A(t)),Q_{i_1}(A(t)),\ldots,Q_{i_k}(A(t))) 
\end{equation}
then our task is to show that
\begin{equation}\label{fxij}
 \E F(\xi_{pq}) = \E F(\xi'_{pq}) + O(n^{-5/2+O(c_0)}).
\end{equation}

Suppose that we have Taylor expansions of the form
\begin{equation}\label{lait}
 \lambda_{i_l}(A(t)) = \lambda_{i_l}(A(0)) + \sum_{j=1}^4 c_{l,j} t^j + O( n^{-5/2+O(c_0)} )
\end{equation}
for all $t = O(n^{O(c_0)})$ and $l=1,\ldots,k$, where the Taylor coefficients $c_{l,j}$ have size $c_{l,j} = O(n^{-j/2+O(c_0)}$, and similarly for the quantities $Q_{i_l}(A(t))$.  Then by using the hypothesis \eqref{geo} and further Taylor expansion, we can obtain a Taylor expansion
$$
F(t) = F(0) + \sum_{j=1}^4 f_j t^j + O( n^{-5/2+O(c_0)} )
$$
for the function $F(t)$ defined in \eqref{ftt}, where the Taylor coefficients $f_j$ have size $f_j = O(n^{-j/2+O(c_0)})$.  Setting $t$ equal to $\xi_{pq}$ and taking expectations, and noting that the Taylor coefficients $f_j$ depend only on $F$ and $A(0)$ and is thus independent of $\xi_{ij}$, we conclude that
$$ \E F(\xi_{pq}) = \E F(0) + \sum_{j=1}^4 (\E f_j) (\E \xi_{pq}^j) + O( n^{-5/2+O(c_0)} )$$
and similarly for $\E F(\xi'_{pq})$.  If $\xi_{pq}$ and $\xi'_{pq}$ have matching moments to fourth order, this gives \eqref{fxij}.  (Note that a similar argument also would give the Three Moment Theorem, as well as Proposition \ref{approx-moment}.)

It remains to establish \eqref{lait} (as well as the analogue for $Q_{i_l}(A(t))$, which turns out to be analogous).  We abbreviate $i_l$ simply as $i$.  By Taylor's theorem with remainder, it would suffice to show that 
\begin{equation}\label{lat}
 \frac{d^j}{dt^j} \lambda_i(A(t)) = O( n^{-j/2+O(c_0)} )
\end{equation}
for $j=1,\ldots,5$. As it turns out, this is not quite true as stated, but it becomes true (with overwhelming probability\footnote{Technically, each value of $t$ has a different exceptional event of very small probability for which the estimates fail.  Since there are uncountably many values of $t$, this could potentially cause a problem when applying the union bound.  In practice, though, it turns out that one can restrict $t$ to a discrete set, such as the multiples of $n^{-100}$, in which case the union bound can be applied without difficulty.  See \cite{TVlocal1} for details.}) if one can assume that $Q_i(A(t))$ is bounded by $n^{O(c_0)}$.  In principle, one can reduce to this case due to the restriction \eqref{qii} on the support of $G$, although there is a technical issue because one will need to establish the bounds \eqref{lat} for values of $t$ other than $\xi_{pq}$ or $\tilde \xi_{pq}$.  This difficulty can be overcome by a continuity argument; see \cite{TVlocal1}.  For the purposes of this informal discussion, we shall ignore this issue and simply assume that we may restrict to the case where
\begin{equation}\label{qiat}
Q_i(A(t)) \ll n^{O(c_0)}.
\end{equation}
In particular, the eigenvalue $\lambda_i(A(t))$ is simple, which ensures that all quantities depend smoothly on $t$ (locally, at least).

To prove \eqref{lat}, one can use the classical \emph{Hadamard variation formulae} for the derivatives of $\lambda_i(A(t))$, which can be derived for instance by repeatedly differentiating the eigenvector equation $A(t) u_i(A(t)) = \lambda_i(A(t)) u_i(A(t))$.  The formula for the first derivative is
$$ \frac{d}{dt} \lambda_i(A(t)) = u_i(A(t))^* A'(0) u_i(A(t)).$$
But recall from eigenvalue delocalisation (Corollary \ref{deloc}) that with overwhelming probability, all coefficients of $u_i(A(t))$ have size $O(n^{-1/2+o(1)})$; given the nature
of the matrix \eqref{apo}, we can then obtain \eqref{lat} in the $j=1$ case.

Now consider the $j=2$ case.  The second derivative formula reads
$$ \frac{d^2}{dt^2} \lambda_i(A(t)) = - 2 \sum_{j \neq i} \frac{|u_i(A(t))^* A'(0) u_j(A(t))|^2}{\lambda_j(A(t)) - \lambda_i(A(t))}$$
(compare with the formula \eqref{dbm} for Dyson Brownian motion).  Using eigenvalue delocalisation as before, we see with overwhelming probability that the numerator is $O(n^{-1+o(1)})$.  To deal with the denominator, one has to exploit the hypothesis \eqref{qiat} and the local semicircle law (Theorem \ref{lsc}).  Using these tools, one can conclude \eqref{lat} in the $j=2$ case with overwhelming probability.

It turns out that one can continue this process for higher values of $j$, although the formulae for the derivatives for $\lambda_i(A(t))$ (and related quantities, such as $P_i(A(t))$ and $Q_i(A(t))$) become increasingly complicated, being given by a certain recursive formula in $j$.  See \cite{TVlocal1} for details.

Now we briefly discuss the proof\footnote{The argument here is taken from \cite{TVlocal1}.  In the case when the atom distributions are sufficiently smooth, one can also deduce this result from the level repulsion argument in \cite[Theorem 3.5]{ESY3} and the eigenvalue rigidity estimate \eqref{eigenrigid}, and by using Theorem \ref{trunc} one can extend the gap property to several other Wigner ensembles.  However, this argument does not cover the case of Bernoulli ensembles, which is perhaps the most difficult case of Theorem \ref{gap} or Theorem \ref{gap-special}.} of Theorem \ref{gap-special}.  For sake of discussion we restrict attention to the bulk case $\eps n \leq i \leq (1-\eps) n$; the changes needed to deal with the edge case are relatively minor and are discussed in \cite{TVlocal2}.  The objective here is to limit the probability of the event that the quantity $\lambda_{i+1}(A_n) - \lambda_i(A_n)$ is unexpectedly small.  The main difficulty here is the fact that one is comparing two adjacent eigenvalues.  If instead one was bounding $\lambda_{i+k}(A_n)-\lambda_i(A_n)$ for a larger value of $k$, say $k \geq \log^C n$ for a large value of $C$, then one could obtain such a bound from the local semicircle law (Theorem \ref{lsc}) without much difficulty.  To reduce $k$ all the way down to $1$, the idea is to exploit the following phenomenon:

\centerline{\emph{If $\lambda_{i+1}(A_n) - \lambda_i(A_n)$ is small, then $\lambda_{i+1}(A_{n-1}) - \lambda_{i-1}(A_{n-1})$ is also likely to be small.}}

Here $A_{n-1}$ denotes\footnote{Strictly speaking, one has to multiply $A_{n-1}$ also by $\frac{\sqrt{n-1}}{\sqrt{n}}$ to be consistent with our conventions for $M_n$ and $M_{n-1}$, although this factor turns out to make very little difference to the analysis.} the top left $n-1 \times n-1$ minor of $A_n$.  This phenomenon can be viewed as a sort of converse to the classical \emph{Cauchy interlacing law}
\begin{equation}\label{inter}
 \lambda_{i-1}(A_{n-1}) \leq \lambda_i(A_n) \leq \lambda_i(A_{n-1}) \leq \lambda_{i+1}(A_n) \leq \lambda_{i+1}(A_{n-1})
\end{equation}
(cf. \eqref{interlace}), since this law clearly shows that $\lambda_{i+1}(A_n) - \lambda_i(A_n)$ will be small whenever $\lambda_{i+1}(A_{n-1}) - \lambda_{i-1}(A_{n-1})$ is.  In principle, if one iterates (generalisations of) the above principle $k=\log^C n$ times, one eventually reaches an event that can be shown to be highly unlikely by the local semicircle law.

To explain why we expect such a phenomenon to be true, let us expand $A_n$ as
$$ A_n = \begin{pmatrix} A_{n-1} & X \\ X^* & \sqrt{n} \xi_{nn} \end{pmatrix}$$
where $X \in \C^{n-1}$ is the random vector with entries $\sqrt{n} \xi_{nj}$ for $j=1,\ldots,n-1$.  By expanding out the eigenvalue equation $A_n u_i(A_n) = \lambda_i(A_n) u_i(A_n)$, one eventually obtains the formula
\begin{equation}\label{joan}
 \sum_{j=1}^{n-1} \frac{|u_j(A_{n-1})^* X|^2}{\lambda_j(A_{n-1}) - \lambda_i(A_n)} = \sqrt{n} \xi_{nn} - \lambda_i(A_n)
\end{equation}
that relates $\lambda_i(A_n)$ to the various eigenvalues $\lambda_j(A_{n-1})$ of $A_{n-1}$ (ignoring for sake of discussion the non-generic case when one or more of the denominators in \eqref{joan} vanish); compare with \eqref{arrange}.  Using concentration of measure tools (such as Talagrand's inequality, see e.g. \cite{ledoux}), one expects $|u_j(A_{n-1})^* X|^2$ to concentrate around its mean, which can be computed to be $n(n-1)$.  In view of this and, one expects the largest (and thus, presumably, the most dominant) terms in \eqref{joan} to be the summands on the left-hand side when $j$ is equal to either $i-1$ or $i$.  In particular, if $\lambda_{i+1}(A_n)-\lambda_i(A_n)$ is unexpectedly small (e.g. smaller than $n^{-c_0}$), then by \eqref{inter} $\lambda_i(A_{n-1})-\lambda_i(A_n)$ is also small.  This causes the $j=i$ summand in \eqref{joan} to (usually) be large and positive; to counterbalance this, one then typically expects the $j=i-1$ summand to be large and negative, so that $\lambda_i(A_n)-\lambda_{i-1}(A_{n-1})$ is small; in particular, $\lambda_i(A_{n-1})-\lambda_{i-1}(A_{n-1})$ is small.  A similar heuristic argument (based on \eqref{joan} but with $\lambda_i(A_n)$ replaced by $\lambda_{i+1}(A_n)$ predicts that $\lambda_{i+1}(A_{n-1})-\lambda_i(A_{n-1})$ is also small; summing, we conclude that $\lambda_{i+1}(A_{n-1})-\lambda_{i-1}(A_{n-1})$ is also small, thus giving heuristic support to the above phenomenon.

One can make the above arguments more rigorous, but the details are rather complicated.  One of the complications arises from the slow decay of the term $\frac{1}{\lambda_j(A_{n-1}) -\lambda_i(A_n)}$ as $i$ moves away from $j$.  Because of this, a large positive term (such as the $j=i$ summand) in \eqref{joan} need not be balanced primarily by the negative $j=i-1$ summand, but instead by a dyadic block $i-2^k \leq j < i-2^{k-1}$ of such summands; but this can be addressed by replacing the gap $\lambda_{i+1}(A_n)-\lambda_i(A_n)$ by a more complicated quantity (called the \emph{regularized gap} in \cite{TVlocal1}) that is an infimum of a moderately large number of (normalised) gaps $\lambda_{i_+}(A_n) - \lambda_{i_-}(A_n)$.  A more serious issue is that the numerators $|u_j(A_{n-1})^* X|$ can sometimes be much smaller than their expected value of $\sim n$, which can cause the gap at $A_{n-1}$ to be significantly larger than that at $A_n$.  By carefully counting all the possible cases and estimating all the error probabilities, one can still keep the net error of this situation to be of the form $O(n^{-c})$ for some $c>0$.  It is in this delicate analysis that one must rely rather heavily on the exponential decay hypothesis in Condition \condo, as opposed to the polynomial decay hypothesis in Condition \condone.

This concludes the sketch of Theorem \ref{gap-special}.  We remarked earlier that the extension to the edge case is fairly routine.  In part, this is because the expected eigenvalue gap $\lambda_{i+1}(A_n)-\lambda_i(A_n)$ becomes much wider at the edge (as large as $n^{1/3}$, for instance, when $i=1$ or $i=n-1$), and so Theorem \ref{gap-special} and Theorem \ref{gap} becomes a weaker statement.   There is however an interesting ``bias'' phenomenon that is worth pointing out at the edge, for instance with regard with the interlacing
\begin{equation}\label{lacing}
 \lambda_{n-1}(A_n) \leq \lambda_{n-1}(A_{n-1}) \leq \lambda_n(A_n)
\end{equation}
of the very largest eigenvalues.  On the one hand, the gap $\lambda_n(A_n)-\lambda_{n-1}(A_n)$ between the top two eigenvalues of $A_n$ is expected (and known, in many cases) to be comparable to $n^{1/3}$ on the average; see \eqref{bogan} below.  On the other hand, from the semi-circular law one expects $\lambda_n(A_n)$ to grow like $2n$, which suggests that $\lambda_n(A_n)-\lambda_{n-1}(A_{n-1})$ should be comparable to $1$, rather than to $n^{1/3}$.  In other words, the interlacing \eqref{lacing} is \emph{biased}; the intermediate quantity $\lambda_{n-1}(A_{n-1})$ should be far closer to the right-most quantity $\lambda_n(A_n)$ than the left-most quantity $\lambda_{n-1}(A_n)$.  This bias can in fact be demonstrated by using the fundamental equation \eqref{joan}; the point is that in the edge case (when $i$ is close to $n$) the term $-\lambda_i(A_n)$ on the right-hand side plays a major role, and has to be balanced by $\lambda_i(A_n)-\lambda_i(A_{n-1})$ being as small as $O(1)$.  

This bias phenomenon is not purely of academic interest; it turns out to be an essential ingredient in the proof of eigenvalue delocalisation (Corollary \ref{deloc}) in the edge case, as discussed in Section \ref{semi-sec}. See \cite{TVlocal2} for more discussion.  It would be of interest to understand the precise relationship between the various eigenvalues in \eqref{inter} or \eqref{lacing}; the asymptotic joint distribution for, say, $\lambda_i(A_n)$ and $\lambda_i(A_{n-1})$ is currently not known, even in the GUE case.

\section{Applications}  

By combining the heat flow methods with swapping tools such as the Four Moment Theorem, one can extend a variety of results from the GUE (or gauss divisible) regime to wider classes of Wigner ensembles.  We now give some examples of such extensions.
 
\subsection{Distribution of individual eigenvalues}\label{dist-sec}

One of the simplest instances of the method arises when extending the central limit theorem \eqref{gustav} of Gustavsson \cite{Gus} for eigenvalues $\lambda_i(A_n)$ in the bulk from GUE to more general ensembles:

\begin{theorem}\label{gust}  The gaussian fluctuation law \eqref{gustav} continues to hold for Wigner matrices obeying Condition \condone\ for a sufficiently large $C_0$, and whose atom distributions match that of GUE to second order on the diagonal and fourth order off the diagonal; thus, one has
$$
 \frac{\lambda_i(A_n) - \lambda_i^\cl(A_n)}{\sqrt{\log n/2\pi} / \rho_\sc(u)} \to N(0,1)_\R
$$ 
whenever $\lambda_i^\cl(A_n)=n(u+o(1))$ for some fixed $-2 < u < 2$.
\end{theorem}

\begin{proof}  Let $M'_n$ be drawn from GUE, thus by \eqref{gustav} one already has
$$
 \frac{\lambda_i(A'_n) - \lambda_i^\cl(A_n)}{\sqrt{\log n/2\pi} / \rho_\sc(u)} \to N(0,1)_\R
$$ 
(note that $\lambda_i^\cl(A_n) = \lambda_i^\cl(A'_n)$.  To conclude the analogous claim for $A_n$, it suffices to show that
\begin{equation}\label{lama}
\P( \lambda_i( A'_n ) \in I_- ) - n^{-c_0} \leq
\P( \lambda_i( A_n ) \in I ) \leq \P( \lambda_i( A'_n ) \in I_+ ) + n^{-c_0}
\end{equation}
for all intervals $I = [a,b]$, and $n$ sufficiently large, where $I_+ := [a-n^{-c_0/10}, b+n^{-c_0/10}]$ and $I_- := [a+n^{-c_0/10},b-n^{-c_0/10}]$.

We will just prove the second inequality in \eqref{lama}, as the first is very similar.
We define a smooth bump function $G:\R \to \R^+$ equal to one on $I_-$ and vanishing outside of $I_+$.  Then we have
$$ \P( \lambda_i( A_n ) \in I ) \leq \E G( \lambda_i(A_n) ) $$
and
$$ \E G( \lambda_i(A'_n) ) \leq \P( \lambda_i( A'_n ) \in I )$$
On the other hand, one can choose $G$ to obey \eqref{G-deriv}.  Thus by Theorem \ref{theorem:Four} we have
$$ |\E G( \lambda_i(A_n) ) - \E G( \lambda_i(A'_n) )| \leq n^{-c_0}$$
and the second inequality in \eqref{lama} follows from the triangle inequality.  The first inequality is similarly proven using a smooth function that equals $1$ on $I_-$ and vanishes outside of $I$.
\end{proof}

\begin{remark}  In \cite{Gus} the asymptotic joint distribution of $k$ distinct eigenvalues $\lambda_{i_1}(M_n),\ldots,\lambda_{i_k}(M_n)$ in the bulk of a GUE matrix $M_n$ was computed (it is a gaussian $k$-tuple with an explicit covariance matrix).  By using the above argument, one can extend that asymptotic for any fixed $k$ to other Wigner matrices, so long as they match GUE to fourth order off the diagonal and to second order on the diagonal.

If one could extend the results in \cite{Gus} to broader ensembles of matrices, such as gauss divisible matrices, then the above argument would allow some of the moment matching hypotheses to be dropped, using tools such as Lemma \ref{match}.
\end{remark}

\begin{remark} Recently in \cite{Doering}, a moderate deviations property of the distribution of the eigenvalues $\lambda_i(A_n)$ was established first for GUE, and then extended to the same class of matrices considered in Theorem \ref{gust} by using the Four Moment Theorem.  An analogue of Theorem \ref{gust} for real symmetric matrices (using GOE instead of GUE) was established in \cite{rourke}.
\end{remark}

A similar argument to the one given in Theorem \ref{gust} also applies at the edge of the spectrum.  For sake of discussion we shall just discuss the distribution of the largest eigenvalue $\lambda_n(A_n)$.  In the case of a GUE ensemble, this largest eigenvalue is famously governed by the \emph{Tracy-Widom law} \cite{TW, TWbook}, which asserts that
\begin{equation}\label{bogan}
 \P( \frac{\lambda_n(A_n)-2n}{n^{1/3}} \leq t ) \to \det(1-T_{[t,+\infty)})
\end{equation}
for any fixed $t \in \R$, where $T_{[t,+\infty)}: L^2([t,+\infty)) \to L^2([t,+\infty))$ is the integral operator
$$ T_{[t,+\infty)} f(x) := \int_t^{+\infty} \frac{\Ai(x) \Ai'(y) - \Ai'(x) \Ai(y)}{x-y} f(y)\ dy$$
and $\Ai:\R \to \R$ is the Airy function
$$ \Ai(x) := \frac{1}{\pi} \int_0^\infty \cos(\frac{t^3}{3}+xt)\ dt.$$
Interestingly, the limiting distribution in \eqref{bogan} also occurs in many other seemingly unrelated contexts, such as the longest increasing subsequence in a random permutation \cite{BDJ, TWbook}.

It is conjectured that the Tracy-Widom law in fact holds for all Wigner matrices obeying Condition {\condone} with $C_0=4$; this value of $C_0$ is optimal, as one does not expect $\lambda_1(A_n)$ to stay near $2\sqrt{n}$ without this hypothesis (see \cite{bai-yin}).  While this conjecture is not yet fully resolved, there has now been a substantial amount of partial progress on the problem \cite{Sos1,  Ruz, Khor, TVlocal2, Joh2, EYY2}.
Soshnikov \cite{Sos1} was the first to obtain the Tracy-Widom law for a large class of Wigner matrices; thanks to subsequent refinements in \cite{Ruz, Khor}, we know that \eqref{bogan} holds for all Wigner matrices whose entries are iid with symmetric distribution and obeying Condition {\condone} with $C_0=12$.  On the other hand, by using the argument used to prove Theorem \ref{gust}, one can also obtain the asymptotic \eqref{bogan} for Wigner matrices obeying Condition \condone\ for a sufficiently large $C_0$, provided that the entries match that of GUE to fourth order.  Actually, since the asymptotic \eqref{bogan} applies at scale $n^{1/3}$ rather than at scale $1$, it is possible to modify the arguments to reduce the amount of moment matching required, that one only needs the entries to match GUE to third order (which in particular subsumes the case when the distribution is symmetric); see Appendix \ref{erratum}.  

More recently, Johansson \cite{Joh1} established \eqref{bogan} for gauss divisible Wigner ensembles obeying Condition \condone with the optimal decay condition $C_0=4$.  Combining this result with the three moment theorem (and noting that any Wigner matrix can be matched up to order three with a gauss divisible matrix, see Lemma \ref{match} below), one can then obtain \eqref{bogan} for any Wigner matrix obeying Condition \condone for sufficiently large $C_0$.  An independent proof of this claim (which also applied to generalized Wigner matrix models in which the variance of the entries was non-constant) was also established in \cite{EYY2}.  Finally, it was shown very recently in \cite{knowles} that there is a version of the Four Moment Theorem for the edge that only requires two matching moments, which allows one to establish the Tracy-Widom law for all Wigner matrices obeying Condition \condo.

\subsection{Universality of the sine kernel  for Hermitian Wigner matrices}

We now turn to the question of the extent to which the asymptotic \eqref{k-asym}, which asserts that the normalised $k$-point correlation functions $\rho_{n,u}^{(k)}$ converge to the universal limit $\rho^{(k)}_\Dyson$, can be extended to more general Wigner ensembles.  A long-standing conjecture of Wigner, Dyson, and Mehta (see e.g. \cite{Meh}) asserts (informally speaking) that \eqref{k-asym} is valid for all fixed $k$, all Wigner matrices and all fixed energy levels $-2 < u < 2$ in the bulk. 
(They also make the same conjecture for random symmetric and random symplectic matrices.)  However, to make this conjecture precise one has to specify the nature of convergence in \eqref{k-asym}.  For GUE, the convergence is quite strong (in the local uniform sense), but one cannot expect such strong convergence in general, particularly in the case of discrete ensembles in which $\rho_{n,u}^{(k)}$ is a discrete probability distribution (i.e. a linear combination of Dirac masses) and thus is unable to converge uniformly or pointwise to the continuous limiting distribution $\rho^{(k)}_\Dyson$.  We will thus instead settle for the weaker notion of \emph{vague convergence}.  More precisely, we say that \eqref{k-asym} holds in the vague sense if one has
\begin{equation}\label{test}
\lim_{n \to \infty} \int_{\R^k} F(x_1,\ldots,x_k) \rho^{(k)}_{n,u}(x_1,\ldots,x_k)\ dx_1 \ldots dx_k = \int_{\R^k} F(x_1,\ldots,x_k) \rho^{(k)}_{\Dyson}(x_1,\ldots,x_k)\ dx_1 \ldots dx_k
 \end{equation} P
for all continuous, compactly supported functions $F: \R^k \to \R$.  By the Stone-Weierstrass theorem we may take $F$ to be a test function (i.e. smooth and compactly supported) without loss of generality.

\begin{remark}  Vague convergence is not the only notion of convergence studied in the literature.  Another commonly studied notion of convergence is \emph{averaged vague convergence}, in which one averages over the energy parameter $u$ as well, thus replacing \eqref{test} with the weaker claim that
\begin{equation}\label{avg}
\begin{split}
\lim_{b\to 0}\lim_{n \to \infty} &\frac{1}{2b} \int_{E-b}^{E+b} \int_{\R^k} F(x_1,\ldots,x_k) \rho^{(k)}_{n,u}(x_1,\ldots,x_k)\ dx_1 \ldots dx_k du \\
&\quad = \int_{\R^k} F(x_1,\ldots,x_k) \rho^{(k)}_{\Dyson}(x_1,\ldots,x_k)\ dx_1 \ldots dx_k
\end{split}
\end{equation}
for all $-2 < E < 2$.  It can be argued (as was done in \cite{EY}) that as the original conjectures of Wigner, Dyson, and Mehta did not precisely specify the nature of convergence, that any one of these notions of convergence would be equally valid for the purposes of claiming a proof of the conjecture.  However, the distinction between such notions is not purely a technical one, as certain applications of the Wigner-Dyson-Mehta conjecture are only available if the convergence notion is strong enough.  

Consider for instance the {\it gap} problem of determining the probability that there is no eigenvalue in the interval $[-t/2n,t/2n]$
(in other words the distribution of the least singular value). This is an important problem 
which in the GUE case was studied in \cite{JMM} and discussed in length in a number of important books in the field, including Mehta's  (see \cite[Chapter 5]{Meh}), Deift's
(see \cite[Section 5.4]{Deibook}), Deift-Gioev's (see \cite[Section 4.2]{DG})
and Anderson-Guionnet-Zeitouni (see \cite[Section 3.5]{AGZ}), Forrester 's (see \cite[Section 9.6]{For}) and the Oxford handbook of matrix theory
edited by Akemann et. al. (see \cite[Section 4.6]{Hand} by Anderson).  This distribution can be determined from the correlation functions $\rho^{(k)}_{n,0}$ at the energy level $u=0$ by a standard combinatorial argument.  In particular, if one has the Wigner-Dyson-Mehta conjecture in the sense of vague convergence, one can make the limiting law in \cite{JMM} universal over the class of matrices for which that conjecture is verified; see Corollary \ref{lsv-gue} below.  However, if one only knows the Wigner-Dyson-Mehta conjecture in the sense of averaged vague convergence \eqref{avg} instead of vague convergence \eqref{test}, one cannot make this conclusion, because the distributional information at $u=0$ is lost in the averaging process.  See also Theorem \ref{ni-asym} for a further example of a statistic which can be controlled by the vague convergence version of the Wigner-Dyson-Mehta conjecture, but {\it not} the averaged vague convergence version. 
 
  For this reason, it is our opinion that a solution of Wigner-Dyson-Mehta conjecture in the averaged vague convergence sense should be viewed as an important partial resolution to that conjecture, but one that falls short of a complete solution to that conjecture\footnote{In particular, we view this conjecture as only partially resolved in the real symmetric and symplectic cases, as opposed to the Hermitian cases, because the available results in that setting are either only in the averaged vague convergence sense, or require some additional moment matching hypotheses on the coefficients.  Completing the proof of the Wigner-Dyson-Mehta conjecture in these categories remains an interesting future direction of research.}.
\end{remark}

As a consequence of the heat flow and swapping techniques, the Wigner-Dyson-Mehta conjecture for (Hermitian) Wigner matrices
 is largely resolved in the vague convergence category:

\begin{theorem}\label{vague-o}   Let $M_n$ be a Wigner matrix obeying Condition \condone\ for a sufficiently large absolute constant $C_0$ which matches moments with GUE to second order, and let $-2 < u < 2$ and $k \geq 1$ be fixed.  Then \eqref{k-asym} holds in the vague sense.
\end{theorem}

This theorem has been established as a result of a long sequence of partial results towards the Wigner-Dyson-Mehta conjecture \cite{Joh1, ERSY, EPRSY, TVlocal1, ERSTVY, EYY, EYY2, TVmeh}, which we will summarise (in a slightly non-chronological order) below; see Remark \ref{chron} for a more detailed discussion of the chronology.  The precise result stated above was first proven explicitly in \cite{TVmeh}, but relies heavily on the previous works just cited.  As recalled in Section \ref{gue-sec}, the asymptotic \eqref{k-asym} for GUE (in the sense of locally uniform convergence, which is far stronger than vague convergence) follows as a consequence of the Gaudin-Mehta formula and the Plancherel-Rotach asymptotics for Hermite polynomials\footnote{Analogous results are known for much wider classes of invariant random matrix ensembles, see e.g. \cite{DKMVZ}, \cite{PS}, \cite{BI}.  However, we will not discuss these results further here, as they do not directly impact on the case of Wigner ensembles.}.  

The next major breakthrough was by Johansson \cite{Joh1}, who, as discussed in Section \ref{heatflow}, establshed \eqref{k-asym} for gauss divisible ensembles at some fixed time parameter $t>0$ independent of $n$, obtained \eqref{k-asym} in the vague sense (in fact, the slightly stronger convergence of \emph{weak convergence} was established in that paper, in which the function $F$ in \eqref{test} was allowed to merely be $L^\infty$ and compactly supported, rather than continuous and compactly supported).    The main tool used in \cite{Joh1} was an explicit determinantal formula for the correlation functions in the gauss divisible case, essentially due to Br\'ezin and Hikami \cite{brezin}.

In Johansson's result, the time parameter $t > 0$ had to be independent of $n$.  It was realized by Erd\H{o}s, Ramirez, Schlein, and Yau that one could obtain many further cases of the Wigner-Dyson-Mehta conjecture if one could extend Johansson's result to much shorter times $t$ that decayed at a polynomial rate in $n$.  This was first achieved (again in the context of weak convergence) for $t > n^{-3/4+\eps}$ for an arbitrary fixed $\eps>0$ in \cite{ERSY}, and then to the essentially optimal case $t > n^{-1+\eps}$ (for weak convergence, and (implicitly) in the local $L^1$ sense as well) in \cite{EPRSY}.  By combining this with the method of reverse heat flow discussed in Section \ref{swap-sec}, the asymptotic \eqref{k-asym} (again in the sense of weak convergence) was established for all Wigner matrices whose distribution obeyed certain smoothness conditions (e.g. when $k=2$ one needs a $C^6$ type condition), and also decayed exponentially.  The methods used in \cite{EPRSY} were an extension of those in \cite{Joh1}, combined with an approximation argument (the ``method of time reversal'') that approximated a continuous distribution by a gauss divisible one (with a small value of $t$); the arguments in \cite{ERSY} are based instead on an analysis of the Dyson Brownian motion.

Note from the eigenvalue rigidity property \eqref{eigenrigid} that only a small number of eigenvalues (at most $n^{o(1)}$ or so) make\footnote{Strictly speaking, the results in \cite{EYY2} only establish the eigenvalue rigidity property \eqref{eigenrigid} assuming Condition \condo.  However, by using the Four Moment Theorem one can relax this to Condition \condone\ for a sufficiently large $C_0$, at the cost of making \eqref{eigenrigid} hold only with high probability rather than overwhelming probability.} a significant contribution to the normalised correlation function $\rho^{(k)}_{n,u}$ on any fixed compact set, and any fixed $u$.  Because of this, the Four Moment Theorem (Theorem \ref{theorem:Four}) can be used to show that\footnote{Very recently, it was observed in \cite{EY} that if one uses the Green's function Four Moment Theorem from \cite{EYY} in place of the earlier eigenvalue Four Moment Theorem from \cite{TVlocal1} at this juncture, then one can reach the same conclusion here without the need to invoke the eigenvalue rigidity theorem, thus providing a further simplification to the argument.} if one Wigner matrix $M_n$ obeyed the asymptotics \eqref{k-asym} in the vague sense, then any other Wigner matrix $M'_n$ that matched $M_n$ to fourth order would also obey \eqref{k-asym} in the vague sense, assuming that $M_n, M'_n$ both obeyed Condition \condone\ for a sufficiently large $C_0$ (so that the eigenvalue rigidity and four moment theorems are applicable).  

By combining the above observation with the moment matching lemma presented below, one immediately concludes Theorem \ref{vague-o} assuming that the off-diagonal atom distributions are supported on at least three points.

\begin{lemma}[Moment matching lemma]\label{match}  Let $\xi$ be a real random variable with mean zero, variance one, finite fourth moment, and which is supported on at least three points.  Then there exists a gauss divisible, exponentially decaying real random variable $\xi'$ that matches $\xi$ to fourth order.
\end{lemma}

For a proof of this elementary lemma, see \cite[Lemma 28]{TVlocal1}.  The requirement of support on at least three points is necessary; indeed, if $\xi$ is supported in just two points $a,b$, then $\E (\xi-a)^2 (\xi-b)^2 = 0$, and so any other distribution that matches $\xi$ to fourth order must also be supported on $a,b$ and thus cannot be gauss divisible.  

To remove the requirement that the atom distributions be supported on at least three points, one can use the observation from Proposition \ref{approx-moment} that one only needs the moments of $M_n$ and $M'_n$ to \emph{approximately} match to fourth order in order to be able to transfer results on the distribution of spectra of $M_n$ to that of $M'_n$.  In particular, if $t = n^{-1+\eps}$ for some small $\eps > 0$, then the Ornstein-Uhlenbeck flow $M^t_n$ of $M_n$ by time $t$ is already close enough to matching the first four moments of $M_n$ to apply Proposition \ref{approx-moment}.  The results of \cite{EPRSY} give the asymptotic \eqref{k-asym} for $M^t_n$, and the eigenvalue rigidity property \eqref{eigenrigid} then allows one to transfer this property to $M_n$, giving Theorem \ref{vague-o}.

\begin{remark}\label{chron}  The above presentation (drawn from the most recent paper \cite{TVmeh}) is somewhat ahistorical, as the arguments used above eemerged from a sequence of papers, which obtained partial results using the best technology available at the time.  In the paper \cite{TVlocal1}, where the first version of the Four Moment Theorem was introduced, the asymptotic \eqref{k-asym} was established under the additional assumptions of Condition \condo, and matching the GUE to fourth order\footnote{In \cite{TVlocal1} it was claimed that one only needed a matching condition of GUE to third order, but this was not rigorously proven in that paper due to an issue with the Three Moment Theorem that we discuss in Appendix \ref{erratum}.}; the former hypothesis was due to the weaker form of the four moment theorem known at the time, and the latter was due to the fact that the eigenvalue rigidity result \eqref{eigenrigid} was not yet established (and was instead deduced from the results of Gustavsson \cite{Gus} combined with the Four Moment Theorem, thus necessitating the matching moment hypothesis).  For related reasons, the paper in \cite{ERSTVY} (which first introduced the use of Proposition \ref{approx-moment}) was only able to establish \eqref{k-asym} after an additional averaging in the energy parameter $u$ (and with Condition \condo).  The subsequent progress in \cite{ESY} via heat flow methods gave an alternate approach to establishing \eqref{k-asym}, but also required an averaging in the energy and a hypothesis that the atom distributions be supported on at least three points, although the latter condition was then removed in \cite{EYY2}.  In a very recent paper \cite{ekyy2}, the exponent $C_0$ in Condition \condone\ has been relaxed to as low as $4+\eps$ for any fixed $\eps>0$, though still at the cost of averaging in the energy parameter.  Some generalisations in other directions (e.g. to covariance matrices, or to generalised Wigner ensembles with non-constant variances) were also established in \cite{BenP}, \cite{TVlocal3}, \cite{ESYY}, \cite{EYY}, \cite{EYY2}, \cite{ekyy}, \cite{ekyy2}, \cite{wang}, \cite{EY}.  For instance, it was very recently observed in \cite{EY} that by using a variant of the argument from  \cite{TVmeh}
(replacing  the asymptotic four moment theorem for eigenvalues by an asymptotic four moment theorem for the Green function)  together with a careful inspection of the arguments in \cite{EPRSY} and invoking some results from \cite{EYY}, one can extend Theorem \ref{vague-o} to generalised Wigner ensembles in which the entries are allowed to have variable variance (subject to some additional hypotheses); see \cite{EY} for details. 
\end{remark}

To close this section, we remark that while Theorem \ref{vague-o} is the ``right'' result for discrete Wigner ensembles (except for the large value of $C_0$ in Condition \condone, which in view of the results in \cite{ekyy2} should be reducible to $4+\eps$), one expects stronger notions of convergence when one has more smoothness hypotheses on the atom distribution; in particular, one should have local uniform convergence of the correlation functions when the distribution is smooth enough.  Some very recent progress in this direction in the $k=1$ case was obtained by Maltsev and Schlein \cite{maltsev}, \cite{maltsev2}.

For results concerning symmetric\footnote{With our conventions, the symmetric case is not a sub-case of the Hermitian case, because the matrix would now be required to match GOE to second order, rather than GUE to second order.} and symplectic  random matrices, we refer to \cite{Erd} and the references therein. In these cases, the Wigner-Dyson-Mehta conjecture is established in the context of averaged vague convergence; the case of vague convergence remains open even for gauss divisible distributions (although in the case when four moments agree with GOE, one can recover the vague convergence version of the conjecture thanks to the Four Moment Theorem).  This appears to be an inherent limitation to the relaxation flow method, at least with the current state of technology.  In the Hermitian case, this limitation is overcome by using the explicit formulae of Brezis-Hikami and Johansson (see \cite{Joh1}), but no such tool appears to be available at present in the symmetric and symplectic cases.  It remains an interesting future direction of research to find a way to overcome this limitation and obtain a complete solution to the Wigner-Dyson-Mehta conjecture in these cases, as this could lead to a number of interesting applications, such as the determination of the asymptotic distribution of the least singular value of the adjacency matrix of an Erd\H{o}s-Renyi graph, which remains open currently despite the significant partial progress (see \cite{ekyy2}) on the Wigner-Dyson-Mehta conjecture for such matrices (though see the recent papers \cite{nguyen}, \cite{vershynin} for some lower bounds relating to this problem).
 
\subsection{Distribution of the gaps}
  
In Section \ref{gue} the averaged gap distribution
$$ S_n(s,u,t_n) := \frac{1}{t_n} |\{ 1 \leq i \leq n: |\lambda_i(A_n)-nu| \leq t_n/\rho_\sc(u); \lambda_{i+1}(A_n)-\lambda_i(A_n) \leq s/\rho_\sc(u) \}.$$
was defined for a given energy level $-2 < u < 2$ and a scale window $t_n$ with $1/t_n, t_n/n$ both going to zero as $n \to \infty$.  For GUE, it is known that the expected value $\E S_n(s,u,t_n)$ of this distribution converges as $n \to\infty$ (keeping $u$, $s$ fixed) to the (cumulative) Gaudin distribution \eqref{snsu}; see \cite{DKMVZ}.  Informally, this result asserts that a typical normalised gap $\rho_\sc(u) ( \lambda_{i+1}(A_n)-\lambda_i(A_n) )$, where $\lambda_i^\sc(W_n) = u + o(1)$, is asymptotically distributed on the average according to the Gaudin distribution.

The eigenvalue  gap distribution has received much attention in the mathematics
community, partially thanks to  the fascinating (numerical)
coincidence with the gap distribution of the zeros of the zeta
functions. For more discussions, we refer to  \cite{Deibook, KS,
Deisur} and the references therein. \footnote{We would like to thank P. Sarnak for  
enlightening conversations regarding the gap distribution and for constantly encouraging us to work on the universality problem.} 

It is possible to use an inclusion-exclusion argument to deduce information about $\E S_n(s,u,t_n)$ from information on the $k$-point correlation functions; see e.g. \cite{Deibook}, \cite{DKMVZ}, \cite{Joh1}, \cite{EPRSY}, \cite{ERSTVY}.  In particular, one can establish universality of the Gaudin distribution for the Wigner matrix ensembles considered in those papers, sometimes assuming additional hypotheses due to the averaging of the correlation function in the $u$ parameter; for instance, in \cite{ERSTVY} universality is shown for all Wigner matrices obeying Condition \condo, assuming that $t_n/n$ decays very slowly to zero, by utilising the universality of the averaged $k$-point correlation function established in that paper.

A slightly different approach proceeds by expressing the moments of $S_n(s,u,t_n)$ in terms of the joint distribution of the eigenvalues.
Observe that
$$ \E S_n(s,u,t_n) = \frac{1}{t_n} \sum_{i=1}^n \E 1_{|\lambda_i(A_n)-nu| \leq t_n/\rho_\sc(u)} 1_{\lambda_{i+1}(A_n)-\lambda_i(A_n) \leq s/\rho_\sc(u)}.$$
Replacing the sharp cutoffs in the expectation by smoothed out versions, and applying Theorem \ref{theorem:Four} (and also using \eqref{eigenrigid} to effectively localise the $i$ summation to about $O(t_n) + O(n^{o(1)})$ indices), we see that if $M_n, M'_n$ have four matching moments and both obey Condition \condone\ for a sufficiently large $C_0$, then one has
$$ \E S_n(s,u,t_n) \leq \E S'_n(s+o(1), u,(1+o(1))t_n) + o(1)$$
and similarly
$$ \E S_n(s,u,t_n) \geq \E S'_n(s-o(1), u,(1-o(1))t_n) - o(1)$$
for suitable choices of decaying quantities $o(1)$.  In particular, if $M'_n$ is known to exhibit Gaudin asymptotics \eqref{snsu} for any $-2 < u < 2$ and any $t_n$ 
with $1/t_n, t_n/n = o(1)$, then $M_n$ does as well.  In \cite{Joh1}, the Gaudin asymptotics \eqref{snsu} were established for gauss divisible ensembles (with fixed time parameter $t$) and any $t_n$ with $1/t_n, t_n/n=o(1)$, and thus by Lemma \ref{match} and the above argument we conclude	that \eqref{snsu} also holds for Wigner matrices obeying Condition \condone\ for a sufficiently large $C_0$ whose off-diagonal atom distributions are supported on at least three points.  This last condition can be removed by using Proposition \ref{approx-moment} as in the previous section (using the results in \cite{EPRSY} instead of \cite{Joh1}), thus giving \eqref{snsu} with no hypotheses on the Wigner matrix other than Condition \condone\ for a sufficiently large $C_0$. 

\begin{remark} This argument appeared previously in \cite{TVlocal1}, but at the time the eigenvalue rigidity result \eqref{eigenrigid} was not available, and the Four Moment Theorem required Condition \condo\ instead of Condition \condone, so the statement was weaker, requiring both Condition \condo\ and a vanishing third moment hypothesis (for the same reason as in Remark \ref{chron}).  By averaging over all $u$ (effectively setting $t_n = n$), the need for eigenvalue rigidity could be avoided in \cite{TVlocal1} (but note that the formulation of the averaged universality for the eigenvalue gap in \cite[(5)]{TVlocal1} is not quite correct\footnote{More precisely, instead of
$$
S_n(s; x):=\frac{1}{n} |\{1 \leq i \leq n: x_{i+1}-x_i \le s \} |,
$$
the gap distribution should instead be expressed as
$$
S_n(s; x):=\frac{1}{n} |\{1 \leq i \leq n: x_{i+1}-x_i \le \frac{s}{\rho_\sc(t_i)} \} |,
$$
where $t_i \in [-2,2]$ is the classical location of the $i^{\operatorname{th}}$ eigenvalue:
$$ \int_{-2}^{t_i} \rho_\sc(x)\ dx := \frac{i}{n}.$$} as stated, due to the absence of the normalisation by $\rho_\sc(u)$).
\end{remark}

\begin{remark} A similar argument also applies to higher moments $\E S_n(s,u,t_n)^k$ of $S_n(s,u,t_n)$, and so in principle one can also use the moment method to obtain universal statistics for the full distribution of $S_n(s,u,t_n)$, and not just the expectation.  However, this argument would require extending the results in \cite{DKMVZ}, \cite{Joh1}, or \cite{EPRSY} to control the distribution (and not just the expectation) of $S_n(s,u,t_n)$ for GUE, gauss divisible matrices (with fixed time parameter), or gauss divisible matrices (with $t=n^{-1+\eps}$) respectively.
\end{remark}

\begin{remark} It is natural to ask whether the averaging over the window $t_n$ can be dispensed with entirely.  Indeed, one expects that the distribution of the individual normalised eigenvalue gaps $\rho_\sc(u) ( \lambda_{i+1}(A_n)-\lambda_i(A_n) )$, where $-2<u<2$ is fixed and $\lambda^\sc_i(W_n)=u+o(1)$, should asymptotically converge to the Gaudin distribution in the vague topology, without any averaging in $i$.  The Four moment theorem allows one to readily deduce such a fact for general Wigner ensembles once one has established it for special ensembles such as GUE or gauss divisible ensembles (and to remove all moment matching and support hypotheses on the Wigner ensemble, one would need to treat gauss divisible ensembles with time parameter equal to a negative power of $n$).  However, control of the individual normalised eigenvalue gaps in the bulk are not presently in the literature, even for GUE, though they are in principle obtainable from determinantal process methods.
\end{remark}

\subsection{Universality of the counting function and gap probability} 

Recall for Section \ref{gue-sec} that one has an asymptotic \eqref{nnu} for the number of eigenvalues of a fine-scale normalised GUE matrix for an interval $I := [nu+a/\rho_\sc(u),nu+b/\rho_\sc(u)]$ in the bulk, where $-2<u<2$ and $a,b \in \R$ are fixed independently of $n$; this asymptotic is controlled by the spectral of the Dyson sine kernel $K_\Sine$ on the interval $[a,b]$.  In fact one can allow $u$ to vary in $n$, so long as it lies inside an interval $[-2+\eps,2-\eps]$ for some fixed $\eps>0$.

Using the four moment theorem, one can extend this asymptotic to more general Wigner ensembles.

\begin{theorem}[Asymptotic for $N_I$]\label{ni-asym}  
Consider a Wigner matrix satisfying Condition \condone\ for a sufficiently large constant $C_0$.   Let $\eps>0$ and $a < b$ be independent of $n$.  For any $n$, let $u = u_n$ be an element of $[-2+\eps,2-\eps]$.  Then the asymptotic \eqref{nnu} (in the sense of convergence in distribution) holds.
\end{theorem} 

\begin{proof}  See \cite[Theorem 8]{TVmeh}.  The basic idea is to use the moment method, combining Theorem \ref{vague-o} with the identity
$$ \E \binom{N_I}{k} = \frac{1}{k!} \int_{[a,b]} \rho^{(k)}_{n,u}(x_1,\ldots,x_k)\ dx_1 \ldots dx_k$$
for any $k \geq 1$.  We remark that the method also allows one to control the asymptotic joint distribution of several intervals $N_{I_1}(A_n), \ldots, N_{I_k}(A_n)$; for instance, one can show that $N_I(A_n)$ and $N_J(A_n)$ are asymptotically uncorrelated if $I, J$ have bounded lengths, lie in the bulk, and have separation tending to infinity as $n \to \infty$.  We omit the details.
\end{proof}

Specialising to the case $u=0$, one can obtain universal behaviour for the probability that $M_n$ has no eigenvalue in an interval of the form $(-t/2\sqrt{n}, t/2\sqrt{n})$ for any fixed $t$:

\begin{corollary}\label{lsv-gue} For any fixed $t > 0$, and $M_n$ satisfy the conditions of Theorem \ref{ni-asym}, one has
$$ \P( \frac{1}{\sqrt n} M_n \,\, \hbox{has no eigenvalues in } \,\, (-t/2n, t/2n)  ) \to \exp( \int_0^t  \frac{f(x)}{x} dx )$$
as $n \to \infty$, where $f:\R \to \R$ is the solution of the differential equation
$$(tf'')^2 + 4 (tf'-f) (tf' -f + (f')^2) =0 $$
with the asymptotics  $f(t)  = \frac{-t}{\pi} -\frac{t^2}{\pi^2} -\frac{t^3}{\pi^3} + O(t^4)$ as $t \to 0$.
\end{corollary}

\begin{proof}  In the case of GUE, this was established in \cite[Theorem 3.1.2]{AGZ}, \cite{JMM}.  The general case then follows from Theorem \ref{ni-asym}.  A weaker extension was established in \cite{TVlocal1}, assuming matching moments with GUE to fourth order as well as Condition \condo.
\end{proof}

This corollary is the Hermitian version of the Goldstine-Von Neumann least singular value problem (see \cite{vonG, edel, RV, RV2, TVhard} for more details).  Recently, some additional bounds on the least singular value of Wigner matrices were established in \cite{nguyen}, \cite{vershynin}.

The above universality results concerned intervals whose length was comparable to the fine-scale eigenvalue spacing (which, using the fine-scale normalisation $A_n$, is $1/\rho_\sc(u)$).  One can also use the Four Moment Theorem to obtain similar universality results for the counting function on much larger scales, such as $N_I(W_n)$, where $I := [y,+\infty)$ with $y \in (-2,2)$ in the bulk.  For instance, in \cite{Gus} the mean and variance of this statistic for GUE matrices was computed as
\begin {equation} \label{Gustavsson}
\begin{split}
\E[N_I(W_n)]&= n\int_I \rho_{\sc} +O \Big (\frac{\log n}{n} \Big ) \\
\Var(N_I(W_n))&=\Big(\frac{1}{2\pi^2}+o(1)\Big)\log n.
\end{split}
\end{equation}
Combining this with a general central limit theorem for determinantal processes due to Costin and Lebowitz \cite{CLe}, one obtains
\begin {equation} \label {CLT}
\frac{N_I(W_n)-\E[N_I(W_n')]}{\sqrt{\Var(N_I(W_n'))}} \underset{n \to \infty}{\to} \CN(0,1).
\end{equation}
and hence
\begin{equation} \label {CLTnumerics}
\frac{N_I(W_n)-n \int_I \rho_{\sc}}{\sqrt{\frac{1}{2\pi^2}\log n}} \underset{n \to \infty}{\to} \CN(0,1)
\end{equation}

This result was extended to more general Wigner ensembles:

\begin{theorem}[Central limit theorem for $N_I(W_n)$]
Let $M_n$ be a Wigner matrix obeying Condition \condo\ which matches the corresponding entries of GUE up to order $4$.   Let  $y \in (-2,2)$ be fixed, and set $I := [y,+\infty)$.  Then the asymptotic \eqref{CLTnumerics} hold (in the sense of probability distributions), as do the mean and variance bounds \eqref{Gustavsson}.
\end{theorem}

\begin{proof} See \cite[Theorem 2]{DV}.  The first component of this theorem is established using the Four Moment Theorem; the second part also uses the eigenvalue rigidity estimate \eqref{eigenrigid}.
\end{proof}
  
\subsection{Distribution of the eigenvectors} 

Let $M_n$ be a matrix drawn from the gaussian orthogonal ensemble (GOE).  Then the eigenvalues $\lambda_i(M_n)$ are almost surely simple, and the unit eigenvectors $u_i(M_n) \in S^{n-1} \subset \R^{n-1}$ are well-defined up to sign.  To deal with this sign ambiguity, let us select each eigenvector $u_i(M_n)$ independetly and uniformly at random among the two possible choices.  Since the GOE ensemble is invariant under orthogonal transformations, the eigenvectors $u_i(M_n)$ must each be uniformly distributed on the unit sphere.  As is well known, this implies that the normalised coefficients $\sqrt{n} u_{i,j}(M_n) := \sqrt{n} u_i(M_n) e_j^*$ of these eigenvectors are asymptotically normally distributed; see \cite{Jiang-2006}, \cite{Jiang} for more precise statements in this direction.  

By combining the results of \cite{Jiang-2006}, \cite{Jiang} with the four moment theorem for eigenvectors (Theorem \ref{theorem:Four2}), one can extend this fact to other Wigner ensembles in \cite{TVeigenvector}.  Here is a typical result:
  
\begin{theorem}  \label{theorem:CLT2} Let $M_n$ be a random real symmetric matrix obeying hypothesis {\condone} for a sufficiently large constant $C_0$, 
 which matches GOE to fourth order.  Assume furthermore that the atom distributions of $M_n$ are symmetric (i.e. $\xi_{ij} \equiv -\xi_{ij}$ for all $1 \leq i,j \leq n$).  Let $i = i_n$ be an index (or more precisely, a sequence of indices) between $1$ and $n$, and let $a= a_n \in S^{n-1}$ be a unit vector in $\R^n$ (or more precisely, a sequence of unit vectors).  For each $i$, let $u_i(M_n) \in S^{n-1}$ be chosen randomly among all unit eigenvectors with eigenvalue $\lambda_i(M_n)$.  Then $\sqrt n u_{i}(M_n) \cdot a$ tends to $N(0,1)_\R$ in distribution as $n \to \infty$. 
\end{theorem} 

\begin{proof} See \cite[Theorem 13]{TVeigenvector}.
\end{proof}

As an example to illustrate Theorem \ref{theorem:CLT2}, we can take $a = a_n := \frac{1}{\sqrt n}(1,\ldots, 1) \in S^{n-1}$, and $i := \lfloor n/2\rfloor$.  Then Theorem \ref{theorem:CLT2} asserts that the sum of the entries of the middle eigenvector $u_{\lfloor n/2\rfloor}(M_n)$ is gaussian in the limit.

\subsection{Central limit theorem for log-determinant} 

One of most natural and important matrix functionals is the \emph{determinant}.  As such, the study of determinants of random matrices has a long and rich history. The earlier papers on this study focused on the determinant $\det A_n$ of the non-Hermitian iid model $A_n$, where the entries $\zeta_{ij}$ of the matrix were independent random variables with mean $0$ and variance $1$ \cite{SzT, FT, NRR, Turan, Pre, Kom, Kom1, 
TVdet, Goodman, Girko, G2, Dembo} (see \cite{NVdet}  for a brief survey of these results). 

In \cite{Goodman}, Goodman considered random gaussian matrices $A_n = (\zeta_{ij})_{1 \leq i,j \leq n}$ where the atom variables $\zeta_{ij}$ are iid standard real gaussian variables, $\zeta_{ij} \equiv N(0,1)_\R$. He noticed that in this case the square of the determinant can be expressed as the product of independent 
chi-square variables. Therefore, its logarithm is the sum of independent variables and thus one expects
a central limit theorem to hold. In fact, using properties of the chi-square distribution,  it is not hard to prove\footnote{Here and in the sequel, $\rightarrow$ denotes convergence in distribution.}
\begin{equation} \label{Girko} \frac{\log (|\det A_n|)- \frac{1}{2} \log n! + \frac{1}{2} \log n}{\sqrt{\frac{1}{2} \log n}}  \rightarrow N(0,1)_\R, \end{equation}
where $N(0,1)_\R$ denotes the law of the real gaussian with mean $0$ and variance $1$.

A similar analysis (but with the real chi distribution replaced by a complex chi distribution) also works for complex gaussian matrices, in which $\zeta_{ij}$ remain jointly independent but now have the distribution of the complex gaussian $N(0,1)_\C$ (or equivalently, the real and imaginary parts of $\zeta_{ij}$ are independent and have the distribution of $N(0,\frac{1}{2})_\R$).  In that case, one has a slightly different law
\begin{equation}\label{Girko-2} \frac{\log (|\det A_n|)- \frac{1}{2} \log n! + \frac{1}{4} \log n}{\sqrt{\frac{1}{4} \log n}}  \rightarrow N(0,1)_\R. \end{equation}

We turn now to real iid matrices, in which the $\zeta_{ij}$ are jointly independent and real with mean zero and variance one.
In \cite{G2},  Girko stated that \eqref{Girko} holds 
for such random matrices  under the additional assumption
that the fourth moment of the atom variables is $3$.  Twenty years later,  he claimed a  
much stronger result which replaced the above assumption by the assumption that the atom variables have bounded $(4 +\delta)$-th moment \cite{Girko}. However, there are several points which are  not clear in 
these papers. Recently, Nguyen and the second author \cite{NVdet} gave a  new proof for  \eqref{Girko}. Their approach also results in an estimate for the rate of convergence and is easily extended to handle to complex case. 

The analysis of the above random determinants relies crucially on the fact that the rows of the matrix are jointly independent.  This independence no longer holds for Hermitian random matrix models, which makes the analysis of determinants of Hermitian random matrices more challenging. 
Even showing that the determinant of  a random symmetric Bernoulli matrix is non-zero (almost surely) was  a
long standing open question by Weiss in the 1980s and solved only five years ago \cite{CTV}; for more recent developments see 
\cite[Theorem 31]{TVlocal1}, \cite{Costello}, \cite{nguyen}, \cite{vershynin}. These results give good lower bound for the absolute value of the determinant or the least singular number, but do not reveal any information about the distribution.

Even in the GUE case, it is highly non-trivial to 
prove an analogue of the central limit theorem \eqref{Girko-2}. Notice that  the observation 
of Goodman does not apply due to the dependence between the rows and so it is not even clear why a central limit theorem must hold for the log-determinant. 
In \cite{DC}, Delannay and Le Caer made use of the explicit distribution of GUE and GOE to prove the central limit theorem for these cases. 
   
While it does not seem to be possible to express the log-determinant of GUE as a sum of independent random variables, in \cite{TVdet2}, the authors  found 
 a way to approximate the log-determinant as a sum of weakly dependent terms, based on  analysing a tridiagonal form of GUE due to Trotter \cite{trotter}\footnote{We would like to thank R. Killip for suggesting 
the use of Trotter's form.}. 
Using stochastic calculus and the martingale central limit theorem, we gave another proof (see \cite{TVdet2}) for the central limit theorem for GUE and GOE: 

\begin{theorem} 
[Central limit theorem for log-determinant of GUE and GOE]\label{clt-gue}\cite{DC} Let $M_n$ be drawn from GUE.  Then
$$ \frac{\log |\det(M_n)| - \frac{1}{2}  \log {n!} + \frac{1}{4} \log n }{\sqrt{\frac{1}{2} \log n}} \rightarrow N(0,1)_\R. $$

Similarly, if $M_n$ is drawn from GOE rather than GUE, one has 
$$ \frac{\log |\det(M_n)| - \frac{1}{2}  \log {n!} + \frac{1}{4} \log n }{\sqrt{\log n}} \rightarrow N(0,1)_\R. $$
\end{theorem}

The next task is to extend  beyond the GUE or GOE case.  Our main tool for this is a four moment theorem for log-determinants of Wigner matrices, analogous 
to Theorem \ref{theorem:Four}.

\begin{theorem}[Four moment theorem for determinant]\label{four-moment}  Let $M_n, M'_n$ be Wigner matrices whose atom distributions have independent real and imaginary parts that match to fourth order off the diagonal and to second order on the diagonal, are bounded by $n^{O(c_0)}$ for some sufficiently small but fixed $c_0 > 0$, and are supported on at least three points.  
Let $G: \R \to \R$ obey the derivative estimates
\begin{equation}\label{g-bound}
 |\frac{d^j}{dx^j} G(x)| =O( n^{c_0})
\end{equation}
for $0 \leq j \leq 5$.  Let $z_0 = E + \sqrt{-1}\eta_0$ be a complex number with $|E| \leq 2-\delta$ for some fixed $\delta>0$.  Then
$$ \E G( \log|\det (M_n-\sqrt{n} z_0)| ) - \E G( \log|\det(M'_n-\sqrt{n} z_0)| ) = O( n^{-c} )$$
for some fixed $c>0$, adopting the convention that $G(-\infty)=0$. \end{theorem} 

  The requirements that $M_n, M'_n$ be supported on at least three points, and that $E$ lie in the bulk region $|E| < 2-\delta$ are artificial, due to the state of current literature on level repulsion estimates.  It is likely that with further progress on those estimates that these hypotheses can be removed.  The hypothesis that the atom distributions have independent real and imaginary parts is mostly for notational convenience and can also be removed with some additional effort.

By combining Theorem \ref{four-moment} with Theorem \ref{clt-gue} we obtain

\begin{corollary}[Central limit theorem for log-determinant of Wigner matrices]\label{clt-wigner}  Let $M_n$ be a Wigner matrix whose atom distributions $\zeta_{ij}$ are independent of $n$, have real and imaginary parts that are independent and match GUE to fourth order, and obey Condition {\condone} for some sufficiently large $C_0$.  Then
$$ \frac{\log |\det(M_n)| -  \frac{1}{2} \log n! + \frac{1}{4} \log n}{\sqrt{\frac{1}{2} \log n}} \rightarrow N(0,1)_\R. $$
If $M_n$ matches GOE instead of GUE, then one instead has
$$ \frac{\log |\det(M_n)| -  \frac{1}{2} \log n! + \frac{1}{4} \log n}{\sqrt{\log n}} \rightarrow N(0,1)_\R. $$
\end{corollary}

The deduction of this proposition from Theorem \ref{four-moment} and Theorem \ref{clt-gue} is standard (closely analogous, for instance, to the proof the central limit theorem for individual eigenvalues) and is omitted.
(Notice that in order for the atom variables of $M_n$ match those of GUE 
to fourth order, these variables most have at least three points in their supports.)

\subsection{Concentration of eigenvalues}  \label{rigidity} 

We first discuss the case of the Gaussian Unitary Ensemble (GUE), which is the most well-understood case, as the joint distribution of the eigenvalues is given by a determinantal point process.  Because of this, it is known that for any interval $I$, the random variable $N_I(W_n)$ in the GUE case obeys a law of the form
\begin{equation}\label{etain}
 N_I(W_n) \equiv \sum_{i=1}^{\infty} \eta_i
\end{equation}
where the $\eta_i = \eta_{i,n,I}$ are jointly independent indicator random variables (i.e. they take values in $\{0,1\}$); see e.g. \cite[Corollary 4.2.24]{AGZ}.  The mean and variance of $N_I(W_n)$ can also be computed in the GUE case with a high degree of accuracy:

\begin{theorem}[Mean and variance for GUE]\label{mevar}  \cite{Gus} Let $M_n$ be drawn from GUE, let $W_n := \frac{1}{\sqrt{n}} M_n$, and let $I = [-\infty,x]$ for some real number $x$ (which may depend on $n$).  Let $\eps>0$ be independent of $n$.
\begin{itemize}
\item[(i)] (Bulk case) If $x \in [-2+\eps, 2-\eps]$, then
$$ \E N_I(W_n) = n \int_I \rho_\sc(y)\ dy + O( \frac{\log n}{n} ).$$
\item[(ii)] (Edge case) If $x \in [-2,2]$, then
$$ \E N_I(W_n) = n \int_I \rho_\sc(y)\ dy + O(1).$$
\item[(iii)] (Variance bound) If one has $x \in [-2,2-\eps]$ and $n^{2/3} (2+x) \to \infty$ as $n \to \infty$, one has
$$ \Var N_I(W_n) = (\frac{1}{2\pi^2} + o(1)) \log (n (2+x)^{3/2}).$$
In particular, one has $\Var N_I(W_n) = O( \log n )$ in this regime.
\end{itemize}
\end{theorem}

By combining these estimates with a well-known inequality of Bennett \cite{Ben62} 
(see \cite{TVcont} for details)  we obtain a concentration estimate for $N_I(W_n)$ in the GUE case:

\begin{corollary}[Concentration for GUE]\label{conc} Let $M_n$ be drawn from GUE, let $W_n := \frac{1}{\sqrt{n}} M_n$, and let $I$ be an interval.  Then one has
$$ \P( |N_I(W_n) - n \int_I \rho_\sc(y)\ dy| \geq T) \ll \exp( - c T )$$
for all $T \gg \log n$.
\end{corollary}

From the above corollary we see in particular that in the GUE case, one has
$$ N_I(W_n) = n \int_I \rho_\sc(y)\ dy + O(\log^{1+o(1)} n)$$
with overwhelming probability for each fixed $I$, and an easy union bound argument (ranging over all intervals $I$ in, say, $[-3,3]$ whose endpoints are a multiple of $n^{-100}$ (say)) then shows that this is also true uniformly in $I$ as well.

Now we turn from the GUE case to more general Wigner ensembles.  
As already mentioned, there has been much interest in recent years in obtaining concentration results for $N_I(W_n)$ (and for closely related objects, such as the Stieltjes transform $s_{W_n}(z) := \frac{1}{n} \tr (W_n - z)^{-1}$ of $W_n$) for short intervals $I$, due to the applicability of such results to establishing various universality properties of such matrices; see \cite{ESY1, ESY2, ESY3, TVlocal1, TVedge, ESY, EYY, EYY2}.  The previous best result in this direction was by Erd\H{o}s, Yau, and Yin \cite{EYY2} (see also \cite{ekyy} for a variant):

\begin{theorem}\label{ween}\cite{EYY2} Let $M_n$ be a Wigner matrix obeying Condition {\condo}, and let $W_n := \frac{1}{\sqrt{n}} M_n$.  Then, for any interval $I$, one has
\begin{equation}\label{niwna}
\P( |N_I(W_n) - n \int_I \rho_\sc(y)\ dy| \geq T) \ll \exp( - c T^c )
\end{equation}
for all $T \geq \log^{A \log \log n} n$, and some constant $A>0$.
\end{theorem}

One can reformulate \eqref{niwna} equivalently as the assertion that
$$ \P( |N_I(W_n) - n \int_I \rho_\sc(y)\ dy| \geq T) \ll \exp(\log^{O(\log \log n)} n) \exp( - c T^c )$$
for all $T>0$.

In particular, this theorem asserts that with overwhelming probability one has
$$N_I(W_n) = n \int_I \rho_\sc(y)\ dy + O( \log^{O(\log\log n)} n )$$
for all intervals $I$.  The proof of the above theorem is somewhat lengthy, requiring a delicate analysis of the self-consistent equation of the Stieltjes transform of $W_n$.

Comparing this result with the previous results for the GUE case, we see that there is a loss of a double logarithm $\log \log n$ in the exponent.  It has turned out that using the swapping method one can remove this double logarithmic loss, at least under an additional vanishing moment assumption\footnote{We thank  M. Ledoux for a conversation leading to this study.}

\begin{theorem}[Improved concentration of eigenvalues]\label{main}  \cite{TVcont} Let $M_n$ be a Wigner matrix obeying Condition {\bf C0}, and let $W_n := \frac{1}{\sqrt{n}} M_n$.  Assume that $M_n$ matches moments with GUE to third order off the diagonal (i.e. $\Re \xi_{ij}, \Im \xi_{ij}$ have variance $1/2$ and third moment zero).  Then, for any interval $I$, one has
$$
\P( |N_I(W_n) - n \int_I \rho_\sc(y)\ dy| \geq T) \ll n^{O(1)} \exp( - c T^c )$$
for any $T > 0$.
\end{theorem}

This estimate is phrased for any $T$, but the bound only becomes non-trivial when $T \gg \log^C n$ for some sufficiently large $C$.   In that regime, we see that this result removes the double-logarithmic factor from Theorem \ref{ween}.  In particular, this theorem implies that with overwhelming probability one has
$$ N_I(W_n) = n \int_I \rho_\sc(y)\ dy + O(\log^{O(1)} n)$$
for all intervals $I$; in particular, for any $I$, $N_I(W_n)$ has variance $O(\log^{O(1)} n)$.

\begin{remark}  As we are assuming $\Re(\xi_{ij})$ and $\Im(\xi_{ij})$ to be independent, the moment matching condition simplifies to the constraints that $\E \Re(\xi_{ij})^2 = \E \Im(\xi_{ij})^2 = \frac{1}{2}$ and $\E \Re(\xi_{ij})^3 = \E \Im(\xi_{ij})^3 = 0$.  However, it is possible to extend this theorem to the case when the real and imaginary parts of $\xi_{ij}$ are not independent. \end{remark}

\begin{remark}
The constant $c$ in the bound in  Theorem \ref{main} is quite decent in several cases. For instance, if the atom variables of $M_n$ are Bernoulli or have sub-gaussian tail, then we can set $c= 2/5-o(1)$ by optimizing our arguments (details omitted).  If we assume 4 matching moments rather than 3, then we can set $c=1$,  matching the bound in Corollary \ref{conc}.  It is an  interesting question to determine the best value of $c$. The value of $c$ in \cite{EYY} is implicit and rather small.
\end{remark} 

The proof of the above theorem is different from that in \cite{EYY2} in that it only uses a relatively crude analysis of the self-consistent equation to obtain some preliminary bounds on the Stieltjes transform and on $N_I$ (which were also essentially implicit in previous literature).  Instead, the bulk of the argument relies on using the Lindeberg swapping strategy to deduce concentration of $N_I(W_n)$ in the non-GUE case from the concentration results in the GUE case provided by Corollary \ref{conc}.  In order to keep the error terms in this swapping under control, three matching moments\footnote{Compare with Theorem \ref{theorem:Four}.}.  We need one less moment here because we are working at ``mesoscopic'' scales (in which the number of eigenvalues involved is much larger than $1$) rather than at ``microscopic'' scales. 

Very roughly speaking, the main idea of the argument is to show that high moments such as
$$ \E |N_I(W_n) - n \int_I \rho_\sc(y)\ dy|^k$$
are quite stable (in a multiplicative sense) if one swaps (the real or imaginary part of) one of the entries of $W_n$ (and its adjoint) with another random variable that matches the moments of the original entry to third order.  For technical reasons, however, we do not quite manipulate $N_I(W_n)$ directly, but instead work with a proxy for this quantity, namely a certain integral of the Stieltjes transform of $W_n$.  As observed in \cite{EYY}, the Lindeberg swapping argument is quite simple to implement at the level of the Stieltjes transform (due to the simplicity of the resolvent identities, when compared against the rather complicated Taylor expansions of individual eigenvalues used in \cite{TVlocal1}).

As a corollary, we obtain the following rigidity of eigenvalues result, improving upon \eqref{eigenrigid} when one has a matching moment hypothesis:

\begin{corollary}[Concentration of eigenvalues] \label{cor:rigid} Let $M_n$ be a Wigner matrix obeying Condition \condo, and let $W_n := \frac{1}{\sqrt{n}} M_n$.  Assume that $M_n$ matches moments with GUE to three order off the diagonal and second order on the diagonal.  
Then for any $i$ in the bulk

$$ \P(  |\lambda_i(W_n) - \gamma_i| \geq T/n ) \ll n^{O(1)} \exp(-cT^c)$$
for any $T>0$, where the \emph{classical location} $\gamma_i \in [-2,2]$ is defined by the formula
$$ \int_{-2}^{\gamma_i} \rho_\sc(y)\ dy = \frac{i}{n}.$$

\end{corollary}

This corollary improves \cite[Theorem 2.2]{EYY2} as it allows $T$ to be as small as $\log^{O(1)} n$, instead of $\log^{O(\log \log n)} n$, under the extra third moment assumption.  In particular, in the Bernoulli case, this shows that the variance of the bulk eigenvalues is of order $\log^{O(1)} n/ n$. We believe that this is sharp, up to the hidden constant in $O(1)$.

This corollary also significantly improves \cite[Theorem 29]{TVlocal1}.  (As a matter of fact, the original proof 
of this theorem has a gap in it;  see Appendix A for a further discussion.)  One can have analogous results for the edge case, under four moment assumption; see \cite{TVcont} for details.

\section{Open questions}  
  
While the universality of many spectral statistics of Wigner matrices have now been established, there are still several open questions remaining.  Some of these have already been raised in earlier sections; we collect some further such questions in this section.

In one direction, one can continue generalising the class of matrices for which the universality result holds, for instance by lowering the exponent $C_0$ in Condition \condone, and allowing the entries $\xi_{ij}$ of the Wigner matrix to have different variances $\sigma^2_{ij}$, or for the Wigner matrix to be quite sparse.  For recent work in these directions, see \cite{ekyy}, \cite{ekyy2}, \cite{EYY}, \cite{EYY2}.  With regards to the different variances case, one key assumption that is still needed for existing arguments to work is a \emph{spectral gap hypothesis}, namely that the matrix of variances $(\sigma_{ij}^2)_{1 \leq i,j \leq n}$ has a significant gap between its largest eigenvalue and its second largest one; in addition, for the most complete results one also needs the variances to be bounded away from zero.  This omits some interesting classes of Wigner-type matrices, such as those with large blocks of zeroes.  However, the spectral statistics of $p+n \times p+n$ matrices of the form
$$ \begin{pmatrix} 0 & M \\ M^* & 0 \end{pmatrix}$$
for rectangular ($p \times n$) matrices with iid matrices $M$ are well understood, as the problem is equivalent to that of understanding the singular values of $M$ (or the eigenvalues of the covariance matrix $MM^*$); in particular, analogues of the key tools discussed here (such as the four moment theorem, the local semicircle law, and heat flow methods) are known \cite{TVlocal3}, \cite{ESYY}; this suggests that other block-type variants of Wigner matrices could be analysed by these methods.  A related problem would be to understand the spectral properties of various self-adjoint polynomial combinations of random matrices, e.g. the commutator $AB-BA$ of two Wigner matrices $A,B$.  The global coarse-scale nature of the spectrum for such matrices can be analysed by the tools of free probability \cite{voiculescu}, but there are still very few rigorous results for the local theory.

Another direction of generalisation is to consider generalised Wigner matrices whose entries have non-zero mean, or equivalently to consider the spectral properties of a random matrix $M_n+D_n$ that is the sum of an ordinary Wigner matrix $M_n$ and a deterministic Hermitian matrix $D_n$.  Large portions of the theory seem amenable to extension in this direction, although the global and local semicircular law would need to be replaced by a more complicated variant (in particular, the semicircular distribution $\rho_\sc$ should be replaced by the free convolution of $\rho_\sc$ with the empirical spectral distribution of $D_n$, see \cite{voiculescu, voi2}).

In yet another direction, one could consider non-Hermitian analogues of these problems, for instance by considering the statistics of eigenvalues of iid random matrices (in which the entries are not constrained to be Hermitian, but are instead independent and identically distributed).  The analogue of the semicircular law in this setting is the \emph{circular law}, which has been analysed intensively in recent years (see \cite{TVsur} for a survey).  There are in fact a number of close connections between the Hermitian and non-Hermitian ensembles, and so it is likely that the progress in the former can be applied to some extent to the latter.

Another natural question to ask is to see if the universality theory for Wigner ensembles can be unified in some way with the older, but very extensively developed, universality theory for invariant ensembles (as covered for instance in \cite{Deibook}).  Significant progress in this direction has recently been achieved in \cite{bourgade}, in which heat flow methods are adapted to show that the local spectral statistics of $\beta$-ensembles are asymptotically independent of the choice of potential function (assuming some analyticity conditions on the potential).  This reduces the problem to the gaussian case when the potential is quadratic, which can be handled by existing methods, similarly to how the methods discussed here reduce the statistics of general Wigner matrices to those of invariant ensembles such as GUE or GOE.  Note though that these techniques do not provide an independent explanation as to why these invariant ensembles have the limiting statistics they do (e.g. governed by the sine determinantal process in the bulk, and the Airy determinantal process in the edge, in the case of GUE); for that, one still needs to rely on the theory of determinantal processes.

Returning now to Wigner matrices, one of the major limitations of the methods discussed here is the heavy reliance on the hypothesis that the (upper-triangular) entries are jointly independent; even weak coupling between entries makes many of the existing methods (such as the swapping technique used in the Four Moment Theorem, or the use of identities such as \eqref{snn}) break down.  A good test case would be the asymptotic statistics of the adjacency matrices of random regular graphs, where the fixed degree $d$ is a constant multiple of $n$, such as $n/2$.  This is essentially equivalent to a Wigner matrix model (such as the real symmetric Bernoulli matrix ensemble) in which the row and column sums have been constrained to be zero.  For this model, the global semicircular law and eigenvector delocalisation has recently been established for such matrices; see \cite{dimitriu}, \cite{tran}.

Recall that the Central limit theorem for the log-determinant of non-Hermitian matrices requires only two moment matching. However, in the Hermitian case, 
Corollary \ref{clt-wigner} requires four matching moments. We believe that this requirement can be weakened. For instance, the central limit theorem must hold for random Bernoulli matrices. 

Another interesting problem is to determine the distribution of bulk eigenvalues of a random Bernoulli matrix.  This matrix has only three matching moments, but perhaps a central limit theorem like Theorem \ref{gust} also holds here. We have proved \cite{TVcont} that the variance of any bulk eigenvalue is $\log^{O(1)} n/n$. A good first step would be to determine the right value of the hidden constant in $O(1)$.

\appendix

\section{Some errata}\label{erratum}

In this appendix we report an issue with the \emph{Three Moment Theorem}, which first appeared as the second conclusion of \cite[Theorem 15]{TVlocal1}.  We extract this theorem from that paper for reference:

\begin{theorem}[Three Moment Theorem]\label{theorem:main} There is a small positive constant $c_0$ such that for every $0 < \eps < 1$ and $k \geq 1$ the following holds.
 Let $M_n = (\zeta_{ij})_{1 \leq i,j \leq n}$ and $M'_n = (\zeta'_{ij})_{1 \leq i,j \leq n}$ be
 two random matrices satisfying {\bf C0}. Assume furthermore that for any $1 \le  i<j \le n$, $\zeta_{ij}$ and
 $\zeta'_{ij}$  match to order $3$.  Set $A_n := \sqrt{n} M_n$ and $A'_n := \sqrt{n} M'_n$,
  and let $G: \R^k \to \R$ be a smooth function obeying the derivative bounds
\begin{equation}\label{G-deriv-improv}
|\nabla^j G(x)| \leq n^{-C j c_0}
\end{equation}
for all $0 \leq j \leq 5$ and $x \in \R^k$, and some sufficiently large absolute constant $C>0$.
 Then for any $\eps n \le i_1 < i_2 \dots < i_k \le (1-\eps)n$, and for $n$ sufficiently large depending on $\eps,k$ (and the constants $C, C'$ in the definition of Condition {\bf C0}) we have
\begin{equation} \label{eqn:approximation-improv}
 |\E ( G(\lambda_{i_1}(A_n), \dots, \lambda_{i_k}(A_n))) -
 \E ( G(\lambda_{i_1}(A'_n), \dots, \lambda_{i_k}(A'_n)))| \le n^{-c_0}.
\end{equation}
\end{theorem}

Unfortunately, the proof given of the Three Moment Theorem in \cite[\S 3.3]{TVlocal1} is not correct, although as we shall see shortly, the theorem can be proven by other means, and the proof itself can be repaired in some other instances.  The closely related Four Moment Theorem, which is of more importance in most applications, is not affected by this problem.

The issue affects some subsequent papers \cite{TVlocal2}, \cite{TVlocal3}, \cite{TVnec}, \cite{Joh2}, \cite{wang} where the Three Moment Theorem (or a variant thereof) was used, and we will discuss the alterations needed to correct those papers below.

\subsection{The issue}

In the proof of the Four Moment Theorem in \cite[\S 3.3]{TVlocal1}, the analogue of the function $G$ that appears in Theorem \ref{theorem:main} is replaced with a truncated variant
\begin{equation}\label{tildeg}
 \tilde G( \lambda_{i_1},\ldots,\lambda_{i_k}, Q_{i_1},\ldots,Q_{i_k}) := G(\lambda_{i_1},\ldots,\lambda_{i_k}) \prod_{j=1}^k \eta( Q_{i_j} )
\end{equation}
(see \cite[\S 3.3]{TVlocal1} for notation).  This truncated function $\tilde G$ then obeys the derivative estimates
$$ |\nabla^j \tilde G| \ll n^{c_0}$$
for $j=0,\ldots,5$.  These bounds are sufficient to establish the Four Moment Theorem (the first part of \cite[Theorem 15]{TVlocal1}), by invoking the first part of \cite[Proposition 46]{TVlocal1}.  However, for the purposes of establishing the Three Moment Theorem (the second part of \cite[Theorem 15]{TVlocal1}), these bounds are not sufficient to be able to invoke the second part of \cite[Proposition 46]{TVlocal1}, and so the proof breaks down in this case.

A similar issue also affects the Three Moment Theorem in the bulk of the spectrum for covariance matrices (\cite[Theorem 6]{TVlocal3}), the Three Moment Theorem at the edge of the spectrum of Wigner matrices \cite[Theorem 1.5]{TVlocal2}, and the Three Moment Theorem at the edge of the spectrum for covariance matrices \cite[Theorem 1.5]{wang}.  



\subsection{First fix: strengthen the gap property}\label{fix-four}

The three moment theorem is most useful when applied to eigenvalues at the edge.  In this case, we can 
reprove this theorem by modifying the argument in \cite{TVlocal1} as follows.  

The proof of the Four Moment Theorem relies on the \emph{gap property} (Theorem \ref{gap}).  
Suppose temporarily that one could improve this property to the bound
\begin{equation}\label{lain}
 |\lambda_{i+1}(A_n) - \lambda_i(A_n)| \geq n^{Cc_0}
\end{equation}
with high probability.  This is not actually possible in the bulk of the spectrum, where the mean eigenvalue spacing is comparable to $1$ (by the Wigner semicircle law), but is certainly plausible at the edge of the spectrum (where the mean eigenvalue spacing is $n^{1/3}$).  Assuming this improved gap property, one can fix the original proof of the Three Moment Theorem in \cite{TVlocal1}.  To do this, one replaces the quantities $Q_i$ introduced in \cite[\S 3.2]{TVlocal1} with the modified variant $\tilde Q_i := Q_i / n^{Cc_0/10}$, and replaces the function $\tilde G$ in \eqref{tildeg} with the variant
$$
 \tilde G( \lambda_{i_1},\ldots,\lambda_{i_k}, \tilde Q_{i_1},\ldots,\tilde Q_{i_k}) := G(\lambda_{i_1},\ldots,\lambda_{i_k}) \prod_{j=1}^k \eta( \tilde Q_{i_j} / n^{Cc_0/100} ).$$
The function $\tilde G$ then obeys the improved bounds
$$ |\nabla^j \tilde G| \ll n^{-Cc_0 j/100}.$$
As such, it is possible to adapt the arguments in \cite{TVlocal1} (replacing $Q_i$ with $\tilde Q_i$, and also making some other minor changes
to constants such as $C$) to recover the Three Moment Theorem.

It remains to establish \eqref{lain} with high probability at the edge of the spectrum.  In the case when $A_n$ comes from the GUE ensemble and $i$ is at the edge of the spectrum (thus $i=O(1)$ or $i=n-O(1)$), one can obtain \eqref{lain} with high probability for sufficiently small $c_0$ by a careful inspection of the analysis used to establish the Tracy-Widom law for GUE \cite{TWbook} (or alternatively, one can use the more general analysis in \cite{Joh2}).  By using the \emph{truncated} Three Moment Theorem at the edge (i.e. \cite[Proposition 6.1]{TVlocal2}, after replacing $Q_{i_j}$ with $\tilde Q_{i_j}$ as indicated above), we can then deduce \eqref{lain} with high probability for any Wigner matrix (obeying Condition {\bf C1} for a sufficiently high moment) that matches GUE to third order.  Using this, one can recover the Three Moment Theorem at the edge for any pair $M_n, M'_n$ of matrices that match GUE to third order (i.e. have vanishing third moment). 

\begin{theorem}[Three Moment Theorem]\label{theorem:main1} There is a small positive constant $c_0$ such that for every $0 < \eps < 1$ and $k \geq 1$ the following holds.
 Let $M_n = (\zeta_{ij})_{1 \leq i,j \leq n}$ be  a  random matrix satisfying {\bf C0}  where $\zeta_{ij}, i \neq j$ have vanishing third moment. 
  Set $A_n := \sqrt{n} M_n$ and $A'_n := \sqrt{n} M'_n$, where $A'_n$ is GUE, 
  and let $G: \R^k \to \R$ be a smooth function obeying the derivative bounds
\begin{equation}\label{G-deriv-improv1}
|\nabla^j G(x)| \leq n^{-C j c_0}
\end{equation}
for all $0 \leq j \leq 5$ and $x \in \R^k$, and some sufficiently large absolute constant $C>0$.
 Then for any $ i_1 < i_2 \dots < i_k =O(1) $, and for $n$ sufficiently large depending on $\eps,k$ (and the constants $C, C'$ in the definition of Condition {\bf C0}) we have
\begin{equation} \label{eqn:approximation-improv1}
 |\E ( G(\lambda_{i_1}(A_n), \dots, \lambda_{i_k}(A_n))) -
 \E ( G(\lambda_{i_1}(A'_n), \dots, \lambda_{i_k}(A'_n)))| \le n^{-c_0}.
\end{equation}

The same statement holds at the other edge of the spectrum. 
\end{theorem}

This method allows one reprove one of the main applications of the three moment theorem, \cite[Theorem 1.13]{TVedge},  which establishes 
the universality of spectral statistics  at the edge of a Wigner matrix with vanishing third moment. 

In Theorem \ref{theorem:main1}, the vanishing third moment assumption was needed only to guarantee \eqref{lain}, since GUE satisfies this bound. 
In a recent paper \cite{Joh2} Johansson considered gauss divisible  matrices. While not stated explicitly, his analysis seems to 
show that the convergence rate of the density of states of gauss divisible matrices at the edge is polynomial, and this would imply that  \eqref{lain}
also holds for this larger class of matrices. As a consequence, one would be able to remove the 
vanishing third moment assumption and Theorem \ref{theorem:main1} and also obtain  \cite[Theorem 1.4]{Joh2}.

In a similar manner, one can recover 
 the covariance matrix universality in \cite{wang} (after one verifies the required polynomial convergence rate of the density of states at the edge for Wishart ensembles (or for gauss divisible covariance matrices), which can be extracted from the analysis in \cite{BenP}). 

\subsection{Second fix: use the Knowles-Yin Two Moment Theorem}

 Another way to handle the above applications is to use 
  the very recent Two Moment Theorem at the edge of Knowles and Yin \cite[Theorem 1.1]{knowles}.  The conclusion of this theorem is stronger than that of  the Three Moment Theorem for the edge, for it only requires two matching moments rather than three (and it can also control the eigenvectors as well as the eigenvalues, cf. \cite{TVeigenvector}).  However, its assumption is more restricted, as the  hypotheses on $G$ are stronger  than \eqref{G-deriv-improv1}, being adapted to the scale $n^{1/3}$ of the mean eigenvalue spacing at the edge, rather than to the scale $n^{C c_0}$. 
   As such, the Two Moment Theorem at the edge can be used as a substitute for the Three Moment Theorem to establish universality for the distribution of $k$ eigenvalues at the edge for any fixed $k=O(1)$, as was done in \cite{TVlocal3} and \cite{Joh2}.  In principle, this also recovers the analogous universality results for covariance matrices in \cite{wang}, although this would require extending the Knowles-Yin Two Moment Theorem in \cite{knowles} to the covariance matrix case; this is almost certainly a routine matter, but is not currently in the literature.

\subsection{Third fix: use of eigenvalue rigidity}

This fix applies to the original Three Moment Theorem in the bulk for Wigner matrices (Theorem \ref{theorem:main}).  One can reprove this theorem using the powerful \emph{eigenvalue rigidity} result \eqref{eigenrigid}.  From this bound and \eqref{G-deriv-improv} (with sufficiently large $C$) we conclude that
$$ 
G(\lambda_{i_1}(A_n), \dots, \lambda_{i_k}(A_n))) =
G(\lambda_{i_1}^{\operatorname{cl}}(A_n), \dots, \lambda_{i_k}^{\operatorname{cl}}(A_n))) + o(n^{-c_0})$$
with overwhelming probability, and similarly for $A'_n$.  Since $A_n$ and $A'_n$ have the same classical locations, the three moment theorem follows.

This fix can be used to resolve all the places where the Three Moment Theorem from \cite{TVlocal1} are invoked.  Of course, in all such cases stronger results can be deduced directly from the subsequent results of \cite{EYY2}.  

In the case when the third moment is zero, one can also use Corollary \ref{cor:rigid} instead of \eqref{eigenrigid}. 
The point is that the proof of this corollary 
(which relies on the swapping method)  is much simpler than that of \eqref{eigenrigid}. 

In principle, this fix should also repair the Three Moment Theorem in the bulk for covariance matrices \cite[Theorem 6]{TVlocal3}.  However, the required eigenvalue rigidity result for covariance matrices in the bulk is not yet in the literature, though it can almost certainly be established\footnote{There is a technical issue, namely that the augmented matrix associated to a covariance matrix does not quite obey the spectral gap hypothesis from \cite[Assumption (B)]{EYY2} due to a double eigenvalue at $1$, but this is likely to be an artificial obstruction, given that much of the rest of the universality theory for Wigner matrices is already known to extend to the covariance matrix case (see e.g. \cite{TVlocal3}, \cite{wang}, \cite{ESYY}).} by modifying the arguments in \cite{EYY2}.  In any event, this case is not of particular importance because, to the authors knowledge, there are no applications of the Three Moment Theorem in the bulk for covariance matrices that are explicitly stated in the literature at this time.

\begin{remark}\label{concat}
In \cite{TVlocal3}, it was observed that the exponential decay condition ${\bf C0}$ in the Four Moment Theorem could be relaxed to a finite moment condition ${\bf C1}$ for a sufficiently large moment exponent.  If one could obtain a similar relaxation for the eigenvalue rigidity result in \cite{EYY2}, then the above argument could then be used to recover the Three Moment Theorem in this case also.  In fact, one can achieve this using the Four Moment Theorem.  If $M_n$ is a matrix obeying {\bf C1}, then it can be matched to fourth order (or to approximate fourth order, as in \cite{ERSTVY}) to a matrix $M'_n$ obeying {\bf C0}.  The latter matrix obeys \eqref{eigenrigid} with overwhelming probability; applying the Four Moment Theorem (cf. the proof of \cite[Theorem 32]{TVlocal1}) we conclude that the former matrix also obeys \eqref{eigenrigid} with high probability, and one can then adapt the preceding arguments.
\end{remark}

At the edge of the spectrum, the eigenvalue rigidity property is adapted to a broader scale than that used in the Three Moment Theorem; with respect to the fine-scale normalisation $A_n$, the latter is adapted to the scale $n^{1/3}$ at the edge whilst the former is adapted to the scale $n^{Cc_0}$.  As such, the eigenvalue rigidity property cannot be directly used as a fix for applications at the edge.

\subsection{Corrections to specific papers}

In this section we record specific statements in a number of papers that need repairing, and which of the above fixes are applicable to these papers.

\begin{itemize}
\item The second part of \cite[Theorem 15]{TVlocal1} (i.e. the Three Moment Theorem)  is true, but the proof requires a result that occurred after the publication of this paper, namely the eigenvalue rigidity result in \cite{EYY2}).  References to this part of the theorem in \cite[\S 3.3]{TVlocal1} should  be deleted.

\item Similarly for  \cite[Theorem 32]{TVlocal1}. This theorem is true, but  its proof 
 requires using the rigidity result in  \cite{EYY2} or Corollary \ref{cor:rigid}. 
 
 \item Similarly for the asymptotic for the determinant in
\cite[Theorem 34]{TVlocal1}. In fact, by using the results in \cite{EYY2}, one can even drop the third moment assumption.

\item The second part of \cite[Theorem 6]{TVlocal3} (i.e. the Three Moment Theorem) should be used with care. This statement would hold if one 
has a rigidity result for eigenvalues of Wishart matrices (similar to those in \cite{EYY2}). We believe the proof of such a result would be a routine, but rather tedious, 
modification of the proof in \cite{EYY2}.  If we assume that the third moment vanishes (which is the case in many application), then it suffices 
to obtain an analogue of Corollary \ref{cor:rigid} and the proof of this would be simpler.

\item The eigenvalue localisation result in \cite[Theorem 1.4]{TVnec} now requires $M_n$ to match GUE to fourth order rather than third order.  Again, the moment matching hypotheses can be dropped entirely if one is willing to use the results in \cite{EYY2}.

\item The second part of \cite[Theorem 1.13]{TVlocal2} should add the additional hypothesis that $M_n, M'_n$ match GUE to third order (or at least obey the improved gap condition \eqref{lain} with high probability).  The proof of this theorem then needs to be modified as per Section \ref{fix-four}.

\item In the proof of \cite[Theorem 1.4]{Joh2}, one needs to replace the invocation of the Three Moment Theorem with the more complicated argument indicated in Section \ref{fix-four}.  Alternatively, one can use here the Two Moment Theorem of Knowles and Yin \cite{knowles}.

\item The second part of \cite[Theorem 1.5]{wang} should add the additional hypothesis that $M_n, M'_n$ match the Wishart ensemble to third order (or at least obey the improved gap condition \eqref{lain} with high probability).  The proof of this theorem then needs to be modified as per Section \ref{fix-four}.
\end{itemize}


\begin{thebibliography}{10}


\bibitem{AGZ} G. Anderson, A. Guionnet and O. Zeitouni, An introduction to random matrices,  Cambridge Studies in Advanced Mathematics, 118. Cambridge University Press, Cambridge, 2010.

\bibitem{Bai93a}
Z. D. Bai, Convergence rate of expected spectral distributions of large random matrices. I. Wigner matrices, \emph{Ann. Probab.} \textbf{21} (1993), no. 2, 625–-648. 

\bibitem{Bai93b} 
Z. D. Bai, Convergence rate of expected spectral distributions of large random matrices. II. Sample covariance matrices. \emph{Ann. Probab.} \textbf{21} (1993), no. 2, 649–-672

\bibitem{BDJ}
J. Baik, P. Deift, and K. Johansson, On the distribution of the longest
increasing subsequence of random permutations. \emph{J. Amer. Math. Soc.} \textbf{12} (1999),
1119--1178.

\bibitem{bakry}
D. Bakry, M. \'Emery, \emph{Diffusions hypercontractives}, in: S\'eminaire de probabilit\'es, XIX, 1983/84,
1123 Lecture Notes in Mathematics, Springer, Berlin, 1985, 177–-206.

\bibitem{BSbook}
Z. D. Bai and J. Silverstein, Spectral analysis of large dimensional random matrices, Mathematics Monograph Series \textbf{2}, Science Press, Beijing 2006.

\bibitem{bai-yin}
Z.D. Bai, Y.Q. Yin, {Necessary and sufficient conditions for almost sure convergence of the largest eigenvalue of a Wigner matrix},
\emph{Ann. Probab.} \textbf{16} (1988), no. 4, 1729–-1741. 

\bibitem{BenP} G. Ben Arous and S. P\'ech\'e, Universality of local eigenvalue statistics for some sample covariance matrices, {\it Comm. Pure Appl. Math.} 58 (2005), no. 10, 1316--1357.


\bibitem{Ben62} 
G. Bennett, \emph{Probability Inequalities for the Sum of Independent Random Variables}, Journal of the American Statistical Association \textbf{57} (1962), 33-–45.



\bibitem{berry}
A. Berry, The Accuracy of the Gaussian Approximation to the Sum of Independent Variates, \emph{Trans. Amer. Math. Soc.} \textbf{49} (1941), 122-–136.

\bibitem{BI} P. Bleher and A. Its, Semiclassical asymptotics of orthogonal polynomials, Riemann-Hilbert problem, and universality in the matrix model, {\it Ann. of Math.} (2) \textbf{150} (1999), no. 1, 185--266.

\bibitem{bourgade}
P. Bourgade, L. Erd{\H o}s, H. T. Yau, Universality of General $\beta$-Ensembles, {\it arXiv:1104.2272}

\bibitem{brezin}
E. Br\'ezin, S. Hikami, S., An extension of level-spacing universality, {\it cond-mat/9702213}.

\bibitem{CLe} O. Costin and  J. Lebowitz, Gaussian fluctuations in random matrices, {\it Phys. Rev. Lett. } 75 (1) (1995) 69--72.



\bibitem{CTV} K. Costello, T. Tao and V. Vu, Random
symmetric matrices are alsmot surely singular, {Duke Math. J.} \textbf{135} (2006), 395--413.

\bibitem{Costello} K. Costello, Bilinear and quadratic variants on the Littlewood-Offord problem, {\it preprint}. 

\bibitem{Chat} S. Chatterjee, A generalization of the Lindenberg principle,  \emph{Ann. Probab.} \textbf{34} (2006), no. 6, 2061–-2076. 

\bibitem{DV}
S. Dallaporta, V. Vu, A Note on the Central Limit Theorem for the Eigenvalue Counting Function of Wigner Matrices, {\it to appear Elect Comm. Prob.}

\bibitem{DKMVZ} P. Deift, T. Kriecherbauer, K.T.-R.  McLaughlin,  S. Venakides
and X. Zhou, Uniform asymptotics for polynomials orthogonal with respect to varying exponential weights and applications to universality questions in random matrix theory,
{\it Comm. Pure Appl. Math.}  \textbf{52} (1999), no. 11, 1335--1425.

\bibitem{Deibook} P. Deift, Orthogonal polynomials and random matrices: a Riemann-Hilbert approach. Courant Lecture Notes in Mathematics, 3. New York University, Courant Institute of Mathematical Sciences, New York; American Mathematical Society, Providence, RI, 1999.


\bibitem{DG} P. Deift and D Gioev, Random matrix theory: Invariant ensembles and Universality,  Courant Lecture Notes in Mathematics, 18. New York University, Courant Institute of Mathematical Sciences, New York; American Mathematical Society, Providence, RI, 2009.


\bibitem{Deisur}
P. Deift, Universality for mathematical and physical systems.
\emph{International Congress of Mathematicians} Vol. I,  125--152,
Eur. Math. Soc., Z\"urich, 2007.

\bibitem{Dembo} A.~Dembo, {On random determinants}, \emph{Quart. Appl. Math.}  \textbf{47} (1989), no. 2, 185--195.



\bibitem{dimitriu}
I. Dumitriu, S. Pal, Sparse regular random graphs: spectral density and eigenvectors, {\it arXiv:0910.5306}

\bibitem{Doering}
H. Doering, P. Eichelsbacher, {Moderate deviations for the eigenvalue counting function of Wigner matrices}, {\it arXiv:1104.0221}.


\bibitem{Dys} F. Dyson, Correlations between eigenvalues of a random matrix,
{\it Comm. Math. Phys.}  19 1970 235--250.

\bibitem{DC}  R. Delannay, G. Le Caer,  Distribution of the determinant of a random real-symmetric matrix 
from the Gaussian orthogonal ensemble, {\it  Phys. Rev. E } 62, 1526Ð1536 (2000). 

\bibitem{edel}
A. Edelman, Eigenvalues and condition numbers of random matrices, \emph{SIAM J. Matrix Anal.
Appl.} \textbf{9} (1988), 543–-560.

\bibitem{Erd} L. Erd\H os, Universality of Wigner random matrices: a Survey of Recent Results, {\it arXiv:1004.0861}.

\bibitem{ekyy}
L. Erd\H os, A. Knowles, H.-T. Yau, J. Yin, {Spectral Statistics of Erd\H os-R\'enyi Graphs I: Local Semicircle Law}, {\it arXiv:1103.1919}.

\bibitem{ekyy2}
L. Erd\H os, A. Knowles, H.-T. Yau, J. Yin, {Spectral statistics of Erd\H os-R\'enyi graphs II: eigenvalue
spacing and the extreme eigenvalues}, {\it arXiv:1103.3869}.

\bibitem{ERSY}  L. Erd\H{o}s, J. Ramirez,
  B. Schlein and  H-T. Yau, Universality of sine-kernel for Wigner matrices with a small Gaussian perturbation, {\it arXiv:0905.2089}.

\bibitem{EPRSY}
L. Erd\H{o}s, S. Peche, J. Ramirez,
  B. Schlein and H.-T. Yau, Bulk universality for Wigner matrices, {\it arXiv:0905.4176}

\bibitem{ERSTVY}
L. Erd\H{o}s, J. Ramirez,
  B. Schlein, T. Tao, V. Vu, and  H.-T. Yau, Bulk universality for Wigner hermitian matrices with subexponential decay, {\it arxiv:0906.4400}, {\it To appear in Math. Research Letters}.

\bibitem{ESY1}
L. Erd\H{o}s, B. Schlein and  H.-T. Yau,
Semicircle law on short scales and delocalization of eigenvectors for Wigner random matrices. {\it Ann. Probab.} \textbf{37} (2009), 815-852 .

\bibitem{ESY2}  L. Erd\H{o}s, B. Schlein and H-T. Yau, Local semicircle law and complete delocalization for Wigner random matrices,  \emph{Comm. Math. Phys.} \textbf{287} (2009), no. 2, 641–655. 

\bibitem{ESY3}
L. Erd\H{o}s, B. Schlein and  H.-T. Yau, Wegner estimate and level repulsion for Wigner random matrices.
Int. Math. Res. Notices \textbf{2010} (2010), 436--479.

\bibitem{ESY}
L. Erd\H{o}s, B. Schlein and  H.-T. Yau, Universality of Random Matrices and Local Relaxation Flow, {\it arXiv:0907.5605}

\bibitem{ESYY}
L. Erd\H{o}s, B. Schlein, H.-T. Yau and J. Yin, The local relaxation flow approach to universality of the
local statistics for random matrices. {\it arXiv:0911.3687}

\bibitem{EY}
L. Erd\H{o}s and  H.-T. Yau, A comment on the Wigner-Dyson-Mehta bulk universality conjecture for Wigner matrices. {\it arXiv:1201.5619}

\bibitem{EYY}
L. Erd\H{o}s, H.-T. Yau, and J. Yin, Bulk universality for generalized Wigner matrices. {\it arXiv:1001.3453}


\bibitem{EYY2}
L. Erd\H{o}s, H.-T.Yau, and J. Yin, Rigidity of Eigenvalues of Generalized Wigner Matrices. {\it arXiv:1007.4652}

\bibitem{esseen}
C-G. Esseen, On the Liapunoff limit of error in the theory of probability, \emph{Arkiv f\"ur matematik}, astronomi och fysik A28 (1942), 1-–19.

\bibitem{FS} O. Feldheim and S.  Sodin,  A universality result for the smallest eigenvalues of certain sample covariance matrices, {\it Geom. Funct. Anal. } 20 (2010), no. 1, 88--123.

\bibitem{rains}
P. J. Forrester, E. M. Rains, Interrelationships between orthogonal, unitary and symplectic matrix
ensembles, \emph{Random matrix models and their applications}, Math. Sci. Res. Inst. Publ., vol. 40, Cambridge
Univ. Press, Cambridge, 2001, pp. 171–-207.

\bibitem{For} 
P. J. Forrester,  Log-gases and random matrices. London Mathematical Society Monographs Series, 34. Princeton University Press, Princeton, NJ, 2010.

\bibitem{FT}  G.~E.~Forsythe and J.~W.~Tukey, {The extent of $n$ random unit vectors}, \emph{Bull.~Amer.~Math.~Soc.}~\textbf{58} (1952), 502. 

\bibitem {FK} Z, F\"uredi and J. Koml\'os,
The eigenvalues of random symmetric matrices, {\it Combinatorica}
{\bf 1}  (1981), no. 3, 233--241.




\bibitem{gin}
J. Ginibre, {Statistical ensembles of complex, quaternion, and real matrices}, 
\emph{J. Mathematical Phys.} \textbf{6} (1965), 440-–449. 



\bibitem{G2} V.~L.~Girko, {The central limit theorem for random determinants} (Russian), translation in \emph{Theory Probab.~Appl.~}  \textbf{24}  (1979), no. 4, 729--740. 

\bibitem{Girko} V.~L.~Girko, {A Refinement of the central limit theorem for random determinants} (Russian), translation in \emph{Theory Probab.~Appl.~} \textbf{42} (1998), no. 1, 121--129. 

\bibitem{Gbook} V.~L.~Girko, {Theory of random determinants}, Kluwer Academic Publishers, 1990.

\bibitem{Goodman} N. R. Goodman, {Distribution of the determinant of a complex Wishart distributed matrix}, \emph{Annals of Statistics} \textbf{34} (1963), 178--180.



\bibitem{Alice}
A. Guionnet, Grandes matrices al\'eatoires et th\'eor\`emes d'universalit\'e, \emph{S\'eminaire BOURBAKI}. Avril 2010. 62\`eme ann\'ee, 2009-2010, no 1019.


\bibitem{Gus}  J.  Gustavsson,  Gaussian fluctuations of eigenvalues in the GUE,
 {\it Ann. Inst. H. Poincar\'e Probab. Statist.} 41 (2005), no. 2, 151--178.
 
 \bibitem{Hand} The Oxford handbook of random matrix theory, edited by Akeman, Baik and Di Fancesco, Oxford 2011. 

\bibitem{harish}
Harish-Chandra, Differential operators on a semisimple Lie algebra, \emph{Amer. J. Math.}, \textbf{79}, (1957) 87--120.


\bibitem{HKPV}
J. Ben Hough, M. Krishnapur, Y. Peres, B. Vir\'ag, Determinantal processes and independence. Probab. Surv. \textbf{3} (2006), 206–-229.

\bibitem{Jiang-2006} T. Jiang. How Many Entries of A Typical Orthogonal Matrix Can Be Approximated By Independent Normals?, \emph{Ann. Probab.} \textbf{34}  (2006), 1497--1529.

\bibitem{Jiang} T. Jiang, The Entries of Haar-invariant Matrices from the Classical Compact Groups, \emph{Journal of Theoretical Probability}, \textbf{23} (2010), 1227--1243. 


\bibitem{JMM} M. Jimbo, T.  Miwa, Tetsuji, Y.  Mori and M.  Sato,  Density matrix of an impenetrable Bose gas and the fifth Painlev\'e transcendent,
{\it  Phys. D  1.}  (1980), no. 1, 80--158.


\bibitem{Joh1} K. Johansson, Universality of the local spacing
distribution in certain ensembles of Hermitian Wigner matrices,
{\it Comm. Math. Phys.}  215 (2001), no. 3, 683--705.

\bibitem{Joh2} K. Johansson, Universality for certain Hermitian Wigner matrices under weak moment conditions, {\it preprint}. re

\bibitem{KS}
N. Katz, P. Sarnak, Random matrices, Frobenius eigenvalues, and monodromy. American Mathematical Society Colloquium Publications, 45. American Mathematical Society, Providence, RI, 1999.


\bibitem{Khor}
O. Khorunzhiy, {High Moments of Large Wigner Random Matrices and Asymptotic Properties of the Spectral Norm}, preprint.

\bibitem{knowles}
A. Knowles, J. Yin, Eigenvector Distribution of Wigner Matrices, {\it arXiv:1102.0057}


\bibitem{krasovsky}
I. V. Krasovsky, {Correlations of the characteristic polynomials in the Gaussian unitary ensemble or a singular Hankel determinant}, \emph{Duke Math. J.} {\bf 139} (2007), no. 3, 581-619. 
\bibitem{KKS} J.~Kahn, J.~Koml\'os and E.~Szemer\'edi, {On the probability that a random $\pm 1$ matrix is singular}, \emph{J.~Amer.~Math.~Soc.}~\textbf{8} (1995), 223--240.

\bibitem{Kom}  J.~Koml\'os, {On the determinant of $(0, 1)$ matrices}, \emph{Studia Sci.~Math.~Hungar.}~\textbf{2} (1967), 7--21. 

\bibitem{Kom1}   J.~Koml\'os, {On the determinant of random matrices}, \emph{Studia Sci.~Math.~Hungar.}~\textbf{3} (1968), 387--399. 
 
\bibitem{keating}
J. P. Keating, N. C. Snaith, Random matrix theory and $\zeta(1/2+it)$, \emph{Comm. Math. Phys.} \textbf{214} (2000), no. 1, 57–-89.

\bibitem{killip}
R. Killip, {Gaussian fluctuations for $\beta$ ensembles}, \emph{Int. Math. Res. Not.} 2008, no. 8, Art. ID rnn007, 19 pp. 



\bibitem{ledoux}
M. Ledoux, The concentration of measure phenomenon, Mathematical Surveys and Monographs, 89. American Mathematical Society, Providence, RI, 2001. 

\bibitem{lindeberg}
J. W. Lindeberg, \emph{Eine neue Herleitung des Exponentialgesetzes in der Wahrscheinlichkeitsrechnung}, Math. Z. \textbf{15} (1922), 211--225.

\bibitem{maltsev}
A. Maltsev, B. Schlein, Average Density of States for Hermitian Wigner Matrices, {\it arXiv:1011.5594}.

\bibitem{maltsev2}
A. Maltsev, B. Schlein, A Wegner estimate for Wigner matrices, {\it arXiv:1103.1473}.

\bibitem{Meh} M.L. Mehta, Random Matrices and the Statistical Theory of Energy Levels, Academic Press, New York, NY, 1967.

\bibitem{nguyen}
H. Nguyen, On the least singular value of random symmetric matrices, {\it arXiv:1102.1476}.

\bibitem{NVdet} H. Nguyen and V. Vu, Random matrix: Law of the determinant, {\it submitted}. 

\bibitem{NRR} H.~Nyquist, S.~O.~Rice and J.~Riordan, {The distribution of random determinants}, Quart. Appl. Math. \textbf{12} (1954), 97--104. 





\bibitem{rourke}
S. O'Rourke, {Gaussian fluctuations of eigenvalues in Wigner random matrices}, \emph{J. Stat. Phys.} \textbf{138} (2010), no. 6, 1045–-1066. 

\bibitem{Pas} L. Pastur, On the spectrum of random matrices,  {\it Teoret. Mat.Fiz.} \textbf{10}, 102-112 (1973).

\bibitem{PS} L. Pastur and M. Shcherbina, Universality of the local eigenvalue statistics for a class of unitary invariant random matrix ensembles, {\it J. Statist. Phys. } \textbf{86} (1997), no. 1-2, 109--147.

\bibitem{SP}  S. Peche and A. Soshnikov, 
Wigner Random Matrices with Non-symmetrically Distributed Entries 
{\it J. Stat. Phys. }, vol. 129, No. 5/6, 857-884, (2007). 



\bibitem{Pre}  A.~Pr\'ekopa, {On random determinants I}, \emph{Studia Sci.~Math.~Hungar.} \textbf{2} (1967), 125--132. 

\bibitem{RV}
M. Rudelson, R. Vershynin, The Littlewood-Offord Problem and invertibility of random matrices,
\emph{Advances in Mathematics} \textbf{218} (2008), 600–-633.


 
\bibitem{Ro} A.~Rouault, {Asymptotic behavior of random determinants in the Laguerre, Gram and Jacobi ensembles}, \emph{Latin American Journal of Probability and Mathematical Statistics} (ALEA),  \textbf{3} (2007) 181--230.

\bibitem{SzT} G.~Szekeres and P.~Tur\'an,  {On an extremal problem in the theory of determinants}, \emph{Math. Naturwiss.~Am.~Ungar.~Akad.~Wiss.}~\textbf{56} (1937), 796--806. 

\bibitem{RV2}
M. Rudelson, R. Vershynin, The least singular value of a random square matrix is $O(n^{-1/2})$, \emph{C.
R. Math. Acad. Sci. Paris} \textbf{346} (2008), no. 15-16, 893–-896.

\bibitem{Ruz}
A. Ruzmaikina, {Universality of the edge distribution of eigenvalues of Wigner random matrices with polynomially decaying distributions of entries}, \emph{Comm. Math. Phys}, \textbf{261} (2006), no. 2, 277–-296. 

\bibitem{schlein}
B. Schlein, Spectral Properties of Wigner Matrices, \emph{Proceedings of the Conference QMath 11}, Hradec Kralove, September 2010.

\bibitem{sinai1} Y. Sinai, A. Soshnikov, Central limit theorem for traces of large symmetric
matrices with independent matrix elements, {\it Bol. Soc. Brazil. Mat.} 29, 1--24 (1998).

\bibitem{sinai2} Y. Sinai, A. Soshnikov, A refinement of Wigner's semicircle law in a neighborhood
of the spectrum edge for random symmetric matrices, {\it Func. Anal. Appl.} 32, 114--131 (1998).



\bibitem{Sos1} A. Soshnikov, Universality at the edge of the spectrum in Wigner random matrices,
{\it Comm. Math. Phys.} 207 (1999), no. 3, 697--733.

\bibitem{Sos2} A. Soshnikov, Gaussian limit for determinantal random point fields, {\it Ann. Probab. } 30 (1) (2002) 171--187.


\bibitem{szego}
G. Szeg\"o, On certain Hermitian forms associated with the Fourier series of a positive function, \emph{Comm. Sem. Math. Univ. Lund} 1952(1952), tome supplementaire, 228--238.



\bibitem{TVsur} T. Tao and V. Vu, From the Littlewood-Offord problem to the circular law: universality of the spectral distribution of random matrices, {\it Bull. Amer. Math. Soc. (N.S.)}  46 (2009), no. 3, 377--396. 

\bibitem{TVlocal2}  T. Tao and V. Vu,  Random matrices: universality of local eigenvalue statistics up to the edge, {\it  Comm. Math. Phys}. 298 (2010), no. 2, 549--572. 

\bibitem{TVedge} T. Tao and V. Vu, Random matrices: universality of local eigenvalue statistics up to the edge, {\it Comm. Math. Phys.}, \textbf{298} (2010), 549--572. 

\bibitem{TVcir} 
T.  Tao and V. Vu,  Random matrices: universality of ESDs and the circular law. With an appendix by Manjunath Krishnapur, {\it  Ann. Probab.} 38 (2010), no. 5, 2023--2065. 

\bibitem{TVhard} T. Tao and V. Vu,  Random matrices: the distribution of the smallest singular values, {\it  Geom. Funct. Anal.} 20 (2010), no. 1, 260--297.

\bibitem{TVnec} T. Tao and V. Vu, Random matrices: Localization of the eigenvalues and the necessity of four moments, {\it to appear in Acta Math. Vietnamica}. 

\bibitem{TV-vector} T. Tao and V. Vu, Random matrices: Universal properties of eigenvectors, {\tt arXiv:arXiv:1103.2801}.

\bibitem{TVlocal1}
T. Tao and V. Vu, Random matrices: Universality of the local eigenvalue statistics, {\it Acta Mathematica}
 206 (2011), 127-204.

\bibitem{TVlocal3} 
T. Tao, V. Vu, Random covariance matrices: university of local statistics of eigenvalues, {\it to appear in Annals of Probability}. 

\bibitem{TVmeh} T. Tao and V. Vu, The Wigner-Dyson-Mehta bulk universality conjecture for Wigner matrices, {\it  Electronic Journal of Probability}, 
vol 16 (2011), 2104-2121. 



\bibitem{TVdet} { T.~Tao} and {V.~Vu}, {On random $\pm 1$ matrices: singularity and determinant}, \emph{Random Structures Algorithms} \textbf{28} (2006), 1--23.

\bibitem{TVdet2}  T. Tao and V. Vu, A central limit theorem for the determinant of a Wigner matrix, {\it preprint} http://arxiv.org/pdf/1111.6300.pdf.

\bibitem{TVeigenvector} T. Tao and V. Vu, Random matrices: Universal properties of eigenvectors,
 {\it to appear in Random Matrices: Theory and Applications}. 
 
 \bibitem{TVcont}  T. Tao and V. Vu, Random matrices: Sharp concentration  of eigenvalues, {\it preprint} http://arxiv.org/pdf/1201.4789.pdf.

\bibitem{TW} C. Tracy and H. Widom, On orthogonal and symplectic matrix ensembles, {\it Commun. Math. Phys.} 177 (1996) 727--754.



 \bibitem{TWbook}  C. Tracy and H. Widom, Distribution functions for largest eigenvalues and their applications,  Proceedings of the International Congress of Mathematicians, Vol. I (Beijing, 2002), 587--596.

\bibitem{tran}
L. Tran, V. Vu, K. Wang, Sparse random graphs: Eigenvalues and Eigenvectors, {\it to appear in Random Structures and Algorithms.}


\bibitem{trotter}
H. Trotter, Eigenvalue distributions of large Hermitian matrices; Wigner's
semicircle law and a theorem of Kac, Murdock, and Szeg\"o, \emph{Adv. in Math.} \textbf{54}(1):67–
82, 1984.

\bibitem{Turan} P.~Tur\'an, {On a problem in the theory of determinants} (Chinese), \emph{Acta Math.~Sinica} \textbf{5} (1955), 41l--423. 



\bibitem{vershynin}
R. Vershynin, Invertibility of symmetric random matrices, {\it to appear in Random Structures and Algorithms}

\bibitem{voi2}
Voiculescu, D., Addition of certain non-commuting random variables, \emph{J. Funct. Anal.} \textbf{66} (1986), 323–-346.

\bibitem{voiculescu}
D. Voiculescu, Limit laws for random matrices and free products, \emph{Invent. Math.} \textbf{104} (1991), no. 1, 201–-220.

\bibitem{vonG} 
J. von Neumann and H. Goldstine, {\it Numerical inverting matrices of high order,}
{Bull. Amer. Math. Soc.} 53, 1021-1099, 1947.

\bibitem{Vunorm} V. Vu, Spectral norm of random matrices, {\it Combinatorica}  27 (6), 2007, 721-736.



\bibitem{wang} K. Wang, Random covariance matrices: Universality of local statistics of eigenvalues up to the edge, {\it arXiv:1104.4832}

\bibitem{wig} P. Wigner,
On the distribution of the roots of certain symmetric matrices, {\it
The Annals of Mathematics} {\bf 67} (1958) 325-327.


















\end{thebibliography}
\end{document}